\documentclass{amsart}
\usepackage{amsmath, amsthm, amssymb, amsfonts}
\usepackage[normalem]{ulem}
\usepackage{hyperref}
\usepackage{tikz-cd}
\usepackage{combelow}

\usepackage{enumerate}

\usepackage{geometry}
\newgeometry{margin=1.3in}

\usepackage{verbatim} 
\usepackage{longtable}
\usepackage{caption}
\usepackage{mathtools}

\usepackage{tikz}

\usepackage{caption}

\theoremstyle{plain}
\newtheorem{thm}{Theorem}[section]
\newtheorem{problem}{Problem}
\newtheorem{cor}[thm]{Corollary}

\newtheorem{lemma}[thm]{Lemma}
\newtheorem{prop}[thm]{Proposition}
\newtheorem{conj}[thm]{Conjecture}

\theoremstyle{definition}
\newtheorem{defn}[thm]{Definition}
\newtheorem{example}[thm]{Example}

\theoremstyle{remark}
\newtheorem{rmk}[thm]{Remark}

\newcommand{\BA}{{\mathbb{A}}}

\newcommand{\BC}{{\mathbb{C}}}

\newcommand{\BE}{{\mathbb{E}}}
\newcommand{\BF}{{\mathbb{F}}}
\newcommand{\BG}{{\mathbb{G}}}

\newcommand{\BN}{{\mathbb{N}}}

\newcommand{\BP}{{\mathbb{P}}}
\newcommand{\BQ}{{\mathbb{Q}}}

\newcommand{\BZ}{{\mathbb{Z}}}

\newcommand{\CA}{{\mathcal A}}

\newcommand{\CC}{{\mathcal C}}

\newcommand{\CL}{{\mathcal L}}
\newcommand{\CM}{{\mathcal M}}

\newcommand{\CO}{{\mathcal O}}

\newcommand{\CS}{{\mathcal S}}

\newcommand{\CW}{{\mathcal W}}
\newcommand{\CX}{{\mathcal X}}

\newcommand{\CZ}{{\mathcal Z}}

\newcommand{\Fq}{{\mathfrak{q}}}

\newcommand{\pt}{{\mathsf{p}}}
\newcommand{\ch}{{\mathrm{ch}}}
\DeclareMathOperator{\Hilb}{Hilb}

\DeclareFontFamily{OT1}{rsfs}{}
\DeclareFontShape{OT1}{rsfs}{n}{it}{<-> rsfs10}{}
\DeclareMathAlphabet{\curly}{OT1}{rsfs}{n}{it}

\newcommand{\p}{\mathbb{P}}
\newcommand\Id{\operatorname{Id}}

\newcommand{\Mbar}{{\overline M}}

\newcommand{\Pic}{\mathop{\rm Pic}\nolimits}
\newcommand{\PT}{\mathsf{PT}}

\newcommand{\GW}{\mathsf{GW}}
\newcommand{\Sym}{{\mathrm{Sym}}}

\newcommand{\ev}{{\mathrm{ev}}}
\newcommand{\DR}{\mathsf{DR}}

\newcommand{\Aut}{\operatorname{Aut}}
\newcommand{\aaa}{\alpha}
\newcommand{\bbb}{\beta}
\newcommand{\A}{\mathsf{A}}
\newcommand{\QJac}{\mathrm{QJac}}
\newcommand{\QMod}{\mathrm{QMod}}

\newcommand{\wt}{\mathsf{wt}}
\newcommand{\pr}{\mathrm{pr}}

\newcommand{\1}{\mathsf{1}}

\newcommand{\vacuum}{v_{\varnothing}}

\newcommand{\vir}{\text{vir}}
\newcommand{\red}{\text{red}}

\newcommand{\fq}{{\mathfrak{q}}}

\newcommand{\Nak}{\mathrm{Nak}}
\newcommand{\Deg}{\mathrm{Deg}}

\newcommand{\bo}{\mathbf{0}}

\newcommand{\QPT}{\mathsf{Q}^{\PT}}
\newcommand{\QHilb}{\mathsf{Q}^{\Hilb}}

\newcommand{\Fock}{\mathfrak{F}}

\begin{document}
\title{Quantum cohomology of the Hilbert scheme of points on an elliptic surface}
\date{\today}


\author{Georg Oberdieck}
\address{KTH Royal Institute of Technology, Department of Mathematics}
\email{georgo@kth.se}

\author{Aaron Pixton}
\address{University of Michigan, Department of Mathematics}
\email{pixton@umich.edu}

\begin{abstract}
We determine the quantum multiplication with divisor classes on the Hilbert scheme of points on an elliptic surface $S \to \Sigma$ for all curve classes which are contracted by the induced fibration $S^{[n]} \to \Sigma^{[n]}$.
The formula is expressed in terms of explicit operators on Fock space.
The structure constants are meromorphic quasi-Jacobi forms of index $0$.
Combining with work of Hu-Li-Qin, this determines the quantum multiplication with divisors on the Hilbert scheme of elliptic surfaces with $p_g(S)>0$.
We also determine the equivariant quantum multiplication with divisor classes for the Hilbert scheme of points on the product $E \times \BC$.

The proof of our formula is based on Nesterov's Hilb/PT wall-crossing, a newly established GW/PT correspondence for the product of an elliptic surface times a curve,
and new computations in the Gromov-Witten theory of an elliptic curve.
\end{abstract}

\maketitle
\baselineskip=14pt

\setcounter{tocdepth}{1} 
\tableofcontents

\section{Introduction}
\subsection{Background}
Let $S^{[n]}$ be the Hilbert scheme of $n$ points on a complex smooth projective surface $S$.
It parametrizes the $0$-dimensional closed subschemes of $S$ of length $n$ and is a crepant resolution of the $n$-th symmetric product $\Sym^n(S)$.
The cohomology of $S^{[n]}$ has been completely described,
first additively as an irreducible representation of a Heisenberg algebra in \cite{Nak, Groj}, and later on multiplicatively in \cite{Lehn, LQW, LQW2} with connections to Virasoro and $\CW$-algebras.

For any smooth projective variety $X$, quantum cohomology is a commutative associative deformation of the classical cup product defined by
\[ (\gamma_1 \ast \gamma_2, \gamma_3) = \sum_{\beta \geq  0} q^{\beta} \langle \gamma_1, \gamma_2, \gamma_3 \rangle^X_{0,\beta} \]
where $(\gamma_1, \gamma_2) = \int_X \gamma_1 \cup \gamma_2$ is the intersection pairing, $\beta \in H_2(X,\BZ)$ runs over the cone of semieffective curves of $X$ modulo torsion, $q^{\beta}$ is a formal variable (an element of the Novikov ring), and 
$\langle \gamma_1, \gamma_2, \gamma_3 \rangle^X_{0,\beta}$ are the $3$-pointed genus $0$ Gromov-Witten invariants, which are virtual counts of rational curves in class $\beta$ meeting cycles Poincar\'e dual to $\gamma_1, \gamma_2, \gamma_3$.

\begin{problem} \label{problem:main}
Determine the quantum cohomology ring of the Hilbert schemes of points of $S$.
\end{problem}

This problem has received considerable attention over the years
and is of interest both from a geometric and representation-theoretic viewpoint.
Graber \cite{Graber} computed the quantum cohomology of $(\p^2)^{[2]}$
and used this to determine counts of hyperelliptic plane curves passing through an appropriate number of points.
Li-Li \cite{LiLi} computed for Hilbert schemes of $n$ points, for any $n$, on a simply connected surface
the $2$-pointed Gromov-Witten invariants for extremal curve classes, i.e. those contracted by the Hilbert-Chow morphism $S^{[n]} \to \Sym^n(S)$,
This led subsequently to a proof of Ruan's crepant resolution conjecture for the Hilbert-Chow morphism \cite{Cheong, LiQin}.
 For surfaces $S$ with $p_g(S)>0$, Hu-Li-Qin \cite{HLQ} showed that all Gromov-Witten invariants of $S^{[n]}$ vanish unless they are a linear combination of $(K_S)_n$ and extremal curve classes (see Section~\ref{sec:main formulas} for notation).
For K3 surfaces $S$ the reduced\footnote{Defined by the $3$-pointed reducd Gromov-Witten invariants, see \cite{HilbK3}.} quantum cohomology of $S^{[2]}$ was computed in \cite{HilbK3}, and a conjectural formula for quantum multiplication with divisor classes for any $S^{[n]}$ was given in \cite{HilbK3,vIOP}. 
In \cite{HilbHAE} the structure constants of the reduced quantum product of $K3^{[n]}$ were shown to be quasi-Jacobi forms. This determines the reduced quantum cohomology of $K3^{[n]}$ up to finitely many coefficients.
For further partial results, see \cite{Pontoni, AQ, LiQin1pointed, Fu}.

By using equivariant quantum cohomology, Problem~\ref{problem:main} can also be considered for quasi-projective surfaces which admit a $\BC^{\ast}$-action with proper fixed loci.
Here Okounkov-Pandharipande \cite{OkPandHilbC2}
determined the equivariant quantum cohomology of $(\BC^2)^{[n]}$, and Maulik-Oblomkov \cite{MO} determined the quantum multiplication with divisors on the Hilbert scheme of an $\CA_n$-surface.
These two works have been subsumed by the work of Maulik-Okounkov \cite{MOk} linking the quantum cohomology of Nakajima quiver varieties to quantum groups.

\subsection{Main results}
An {\em elliptic surface} is a smooth complex projective surface $S$ equipped with a relatively minimal genus one fibration $\pi : S \to \Sigma$ to a smooth curve, see \cite{Miranda_Lectures, FM} for introductions. We also always assume the existence of a section $\Sigma \hookrightarrow S$. 
%
The fibration $\pi$ induces a fibration
\[ \pi^{[n]} : S^{[n]} \to \Sym^n(\Sigma). \]
Consider the $\pi$-restricted quantum product on $S^{[n]}$,
\[ (\gamma_1 \ast_{\pi} \gamma_2, \gamma_3) = \sum_{\substack{\beta > 0 \\ \pi^{[n]}_{\ast} \beta = 0}} q^{\beta} \langle \gamma_1, \gamma_2, \gamma_3 \rangle^{S^{[n]}}_{0,\beta}. \]

The main result of this paper (stated as Corollary~\ref{cor:quantum multiplication by divisor}) is an explicit formula for the 
$\pi$-restricted quantum multiplication by divisor classes on all $S^{[n]}$.
By the vanishing result of \cite{HLQ},
this determines the (full) quantum multiplication with divisor classes
whenever $S$ is an elliptic surface with $p_g(S)>0$.
Thus we obtain a first example of a projective surface with non-trivial Gromov-Witten theory where the quantum multiplication with divisor classes on its Hilbert schemes is known.

We also obtain a formula for equivariant quantum multiplication by divisor classes on $(E \times \BC)^{[n]}$,
see Section~\ref{subsec:formula for local elliptic curve}.
The case $(E \times \BC)^{[n]}$ is of particular interest to representation theory, because it is a holomophic-symplectic variety admitting a symplectic-form preserving torus action with proper fixed locus, but which falls outside the class of Nakajima quiver varieties studied in \cite{MO}.

We mention here the following basic qualitative property of the invariants of $S^{[n]}$, which gives an affirmative answer to \cite[Question 1.4]{HilbHAE} for $(E \times \BC)^{[n]}$.

\begin{thm} \label{thm:structrue constants are quasi Jacobi forms}
Let $\pi : S \to \Sigma$ be an elliptic surface. 
All structure constants 
\[ (D \ast_{\pi} \gamma, \gamma') \]
of $\pi$-restricted quantum multiplication by a divisor 
$D \in H^2(S^{[n]})$ are meromorphic quasi-Jacobi forms of index $0$ with poles at torsion points $z = a \tau + b$.
\end{thm}

\subsection{Strategy of the proof}
The key new method we use is Nesterov's wall-crossing formula relating the Gromov-Witten (GW) invariants of the Hilbert scheme of points to the Pandharipande-Thomas (PT) invariants of $S \times \mathrm{Curve}$. We review the theory in Section~\ref{sec:HilbPT wallcrossing}.
The remaining steps are a mostly straightforward but at times tricky computation.
The steps are as follows:
\begin{enumerate}
\item[(a)] We use Nesterov's wall-crossing to relate the $3$-pointed genus 0 GW invariants of $S^{[n]}$, where $S$ is an elliptic surface, to the PT invariants of the relative geometry $(S \times \p^1, S_{0,1,\infty})$.
The relevant wall-crossing terms are determined in Proposition~\ref{prop:I function}.
\item[(b)] We prove the GW/PT correspondence for the relative pair $(S \times C, S_{z})$, $S_{z} = \sqcup_i S \times \{ z_i \}$.

Hence we are reduced to the Gromov-Witten invariants of $(S \times \p^1, S_{0,1,\infty})$.
\item[(c)] For divisor insertions we can reduce further to $(S \times \p^1, S_{0, \infty})$. We then apply the product formula in relative Gromov-Witten theory \cite{LQ}. This reduces us to an integral of the form
\[ 
\int_{[ \Mbar_{g,r}(S, df) ]^{\vir}} \tau^{\ast}( \DR_g(a_1, \ldots, a_r)) \prod_{i=1}^{r} \ev_i^{\ast}(\gamma_i),
\]
where $\DR_g(a_1, \ldots, a_r)$ is the double ramification cycle \cite{JPPZ}.
\item[(d)] Using cosection localization \cite{KL} we describe the virtual class of $\Mbar_{g,r}(S,df)$ in terms of the virtual class of an elliptic curve, see Proposition~\ref{prop:GW class elliptic surface}.
This reduces us further to the following integral in the Gromov-Witten theory of an elliptic curve $E$:
\[ 
\int_{[ \Mbar_{g,n}(E, d) ]^{\vir}} \tau^{\ast}( \DR_g(a_1, \ldots, a_n)) \lambda_{g-1} \prod_{i=2}^{n} \ev_i^{\ast}(\pt),
\]
where $\pt \in H^2(E)$ is the point class.
The case of odd cohomology insertions reduces here to even ones by a geometric argument.
\item[(e)] Using the quasi-modularity and holomorphic anomaly equation for the Gromov-Witten theory of an elliptic curve \cite{HAE},
we reduce further to an integral on the moduli space of curves $\Mbar_{g,n}$ between $\DR_g(a_1,\ldots, a_n) \lambda_g \lambda_{g-1}$ and $\psi$-classes.
\item[(f)] Using Hain's formula \cite{Hain} for the DR cycle we finally reduce to the standard socle integrals on the moduli space of curves,
\[
\int_{\Mbar_{g,n}}\lambda_g\lambda_{g-1}\kappa_{b_0}\psi_1^{b_1+1}\cdots\psi_n^{b_n+1}, \quad
\int_{\Mbar_{g,n}}\lambda_g\lambda_{g-1}\psi_1^{b_1+1}\cdots\psi_n^{b_n+1}, \]
formulas for which were explained as consequences of the Virasoro conjecture for surfaces in \cite{Getzler-Pandharipande}.
\end{enumerate}

This concludes the proof for compact elliptic surfaces. The noncompact case $(E \times \BC)^{[n]}$ follows from $(\p^1 \times E)^{[n]}$ by a localization argument.

\subsection{Related work}
In a recent series of papers \cite{dH1, dH2} DeHority constructs an action of
an  infinite-dimensional Lie (super)algebra
on the equivariant cohomology of the moduli space of framed sheaves on $E \times \BC$ and certain $\BC^{\ast}$-equivariant rational elliptic surfaces.
In particular, he obtains an action of this algebra on the cohomology of the Hilbert scheme of points on these surfaces. In \cite{dH1} it was moreover announced that the quantum multiplication with divisors classes on $(E \times \BC)^{[n]}$ can be expressed in terms of this action.
The announced proof seems to be completely orthogonal to ours,
because according to \cite{dH1} it will use the monodromy of the Hilbert scheme and the WDVV equation. Our main input here is the Hilb/PT/GW link.

While our results cover all elliptic surfaces,
and \cite{dH1} only covers $E \times \BC$,
we do not address the representation-theoretic nature of the quantum multiplication with divisors in our paper.
It would be nice if the two approaches could be linked,
and a representation-theoretic interpretation of the quantum cohomology could be given for the Hilbert scheme of all elliptic surfaces.

\subsection{Conventions}
For a cohomology class $\gamma \in H^{i}(X)$ of degree $i$,
we denote by $\deg(\gamma) := i$ the cohomological degree
and by $\deg_{\BC}(\gamma) := i/2$ the complex cohomological degree.
Singular (co)homology with integral coefficients will be considered here modulo torsion; $H_k(X,\BZ)$ will stand for the $k$-th homology group modulo the torsion submodule, and similarly for cohomology.

\subsection{Plan of the paper}
In Section~\ref{sec:main formulas}, we state the details of the formulas in our main results. In Sections~\ref{sec:Modular and Jacobi forms} and \ref{sec:relativeGWPT}, we fix notation and review basic results about modular forms and GW/PT invariants. In Section~\ref{sec:HilbPT wallcrossing}, we review Nesterov's wall-crossing formula. In Section~\ref{sec:elliptic curves}, we evaluate certain integrals coming from the Gromov-Witten theory of an elliptic curve. In Section~\ref{sec:elliptic surfaces1}, we collect various results on elliptic surfaces, including proving the GW/PT correspondence for fiber classes on elliptic surfaces. In Sections~\ref{sec:elliptic surfaces2} and \ref{sec:localE}, we prove our main formulas.

\subsection{Acknowledgements}
We are grateful to Nikolas Kuhn for discussions about degenerations of elliptic fibrations. The project started during the participation of the first 
author in the Spring 2018 MSRI program "Enumerative geometry beyond numbers". We thank all participants and the organizers for fruitful conversations.

G. Oberdieck was supported by the starting grant 'Correspondences in enumerative geometry: Hilbert schemes, K3 surfaces and modular forms', No 101041491
 of the European Research Council, and a grant of the G\"oran Gustafsson Foundation. A. Pixton was supported by the NSF grant DMS-2301506.
 
\section{Main formulas} \label{sec:main formulas}

\subsection{Cohomology} \label{subsec:Nakajima operators}
We review Nakajima's \cite{Nak} construction of the Heisenberg algebra action on the cohomology of the Hilbert scheme of points on a smooth projective surface $S$.

For any~$n,k \in \BN$, consider the closed subscheme
$$
S^{[n,n+k]} = \left\{(I \supset I') \,\middle|\, I/I' \text{ is supported at a single point }x \in S \right\} \subset S^{[n]} \times S^{[n+k]}
$$
which is endowed with the projection maps
\begin{equation*}
\label{eqn:diagram zk}
\begin{tikzcd}
& S^{[n,n+k]} \ar[swap]{dl}{p_-} \ar{d}{p_S} \ar{dr}{p_+} & \\
S^{[n]} & S & S^{[n+k]}
\end{tikzcd}
\end{equation*}
which remember $I$, $x$, $I'$, respectively. 
For $\alpha \in H^{\ast}(S)$ and $k>0$ define
\[ \Fq_{\pm k}(\alpha) : H^{\ast}(S^{[n]}) \to H^{\ast}(S^{[n \pm k]}) \]
by
\begin{gather*}
\Fq_k(\alpha) \gamma = p_{+ \ast}( p_{-}^{\ast}(\gamma) \cdot p_S^{\ast}(\alpha) ),
\quad \quad 
\Fq_{-k}(\alpha) \gamma = (-1)^k p_{- \ast}( p_{+}^{\ast}(\gamma) \cdot p_S^{\ast}(\alpha) ).
\end{gather*}
We also set $\fq_0(\gamma) = 0$ for all $\gamma$.

Let $(\gamma, \gamma') = \int_{S^{[n]}} \gamma \cdot \gamma'$ denote the Poincar\'e pairing.
We have the relation
\[
( \Fq_k(\alpha) \gamma, \gamma' ) = (-1)^k ( \gamma, \Fq_{-k}(\alpha) \gamma' ), 
\]
for all
$\gamma \in H^{\ast}(S^{[n-k]})$ and $\gamma' \in H^{\ast}(S^{[n]})$.

The Fock space is the direct sum
\[ \mathfrak{F}_S = \bigoplus_{n \geq 0} H^{\ast}(S^{[n]}). \]
Because the correspondences above are defined for all $n$, we obtain operators
\[ \Fq_i(\alpha) : \mathfrak{F}_S \to \mathfrak{F}_S. \]
The following commutation relation of the Heisenberg algebra is satisfied \cite{Nak}:
\begin{equation*}
\label{eqn:heis op}
[\fq_k(\alpha), \fq_l(\beta)] 
= k \delta_{k+l,0} ( \alpha, \beta ) \Id_{\Fock_S}
\end{equation*}

The Fock space $\Fock_S$ is generated by the operators $\Fq_k(\alpha)$ for $k>0$ from the vacuum vector
\[ \vacuum \in H^{\ast}(S^{[0]}) = \BQ. \]

Given homogeneous classes $\alpha_i \in H^{\ast}(S)$ we have
\begin{equation*} 
\deg_{\BC}\big(   \Fq_{k_1}(\alpha_1) \cdots \Fq_{k_{\ell}}(\alpha_{\ell}) \vacuum \big) = n-\ell+\sum_{i} \deg_{\BC}(\alpha_i),
\end{equation*}
where $n = \sum k_i$.

For a class $v \in H^{\ast}(S^{\ell})$ decomposed as
$v=\sum_{i} \alpha^{(i)}_1 \otimes  \cdots \otimes \alpha^{(i)}_{\ell}$ we write
\[ \Fq_{k_1} \cdots \Fq_{k_{\ell}}(v) := \sum_{i} \prod_{j=1}^{\ell} \Fq_{k_j}(\alpha_j^{(i)}). \]

We define the normally ordered product $:\! - \!:$ by:
\begin{equation*}
:\! \Fq_{i_1}(\alpha_1) ... \Fq_{i_\ell}(\alpha_{\ell})\!: \,\, =  \Fq_{i_{\sigma(1)}}(\alpha_{\sigma(1)}) ... \Fq_{i_{\sigma(\ell)}}(\alpha_{\sigma(\ell)})
\end{equation*}
where $\sigma$ is any permutation such that $i_{\sigma(1)} \geq ... \geq i_{\sigma(\ell)}$.

\subsection{Classical multiplication}
Let $\CZ \subset S^{[n]} \times S$ be the universal subscheme and let $\pi_{S^{[n]}}, \pi_S$ be the projections of $S^{[n]} \times S$ to the factors.
For a vector bundle $V$ on $S$ let $V^{[n]} = \pi_{S^{[n]} \ast}(\CO_{\CZ} \otimes \pi_S^{\ast}(V))$
be the associated tautological bundle. 
Given $x \in H^{\ast}(S)$, define
\[
\delta = c_1(\CO_{S}^{[n]}), \quad \quad D(x) = \pi_{\ast}( \pi_S^{\ast}(x) \cdot \ch_2(\CO_{\CZ}) )
\]
Consider the operators
\[ e_{x}, e_{\delta} : \Fock_S \to \Fock_S \]
which act on $H^{\ast}(S^{[n]})$ by cup product with $D(x)$ and $c_1(\CO_{S}^{[n]})$ respectively. 
\begin{thm}[Lehn \cite{Lehn}, but see also \cite{MN} for another proof]\label{thm:lehn}
For all $\alpha \in H^{\ast}(S)$ we have
\begin{equation} \label{eq:classical multiplication by divisor}
\begin{gathered}
e_{\alpha} = -\sum_{n > 0} \Fq_{n} \Fq_{-n} ( \Delta_{\ast} \alpha),  \\
e_{\delta} = -\frac{1}{6} \sum_{i+j+k=0} :\! \Fq_i \Fq_j \Fq_k ( \Delta_{123} )\!:\ -\  \frac{1}{2} \sum_{n>0} (n-1) \, \Fq_{n} \Fq_{-n}(\Delta_{\ast}(K_S)),
\end{gathered}
\end{equation}
where $\Delta_{\ast} : H^{\ast}(S) \to H^{\ast}(S^2)$ is the pushforward along the diagonal and $\Delta_{123} \in H^{\ast}(S^3)$ is the class of the small diagonal.
\end{thm}

\subsection{Curve classes}
Curve classes on $S^{[n]}$ for $n \geq 2$ may be described by
the isomorphism
\begin{align*}
H_2(S, \BQ) \oplus \wedge^2 H_1(S, \BQ) \oplus \BQ & \xrightarrow{\cong} H_2( S^{[n]} ; \BQ) \\
(\beta, a \wedge b, k) & \mapsto \beta_n + (a \wedge b)_n + k A
\end{align*}
where we let
\[
\beta_n = \Fq_1(\beta) \Fq_{1}(\pt)^{n-1} \vacuum, \quad (a \wedge b)_n = \Fq_1(a) \Fq_1(b) \Fq_1(\pt)^{n-2} \vacuum,
\quad 
A = \Fq_2(\pt) \Fq_1(\pt)^{n-2} \vacuum,
\]
where $\pt \in H^4(S,\BZ)$ is the class of a point.
We often drop the subscript $n$ from the notation.

Geometrically, $A$ is the class of the locus of $2$ points centered at a fixed point $p \in S$, plus $n-2$ distinct points away from $p$.
Similarly, if $\beta$ is Poincar\'e dual to a curve $C \subset S$,
then $\beta_n$ is the class of the curve given by fixing $n-1$ distinct points away from $C$ plus a single point moving along $C$.

Here we will always work with the reduced cohomology group 
\[ \widetilde{H}_2(S^{[n]}, \BZ) \]
 defined as the quotient of $H_2(S^{[n]}, \BZ)$ by the image of $\wedge^2 H_1(S,\BZ)$.
This can be motivated by the following result.

\begin{lemma}
Let $f : \p^1 \to S^{[n]}$ be any morphism. The class $f_{\ast}[\p^1] \in H_2(S^{[n]},\BZ)$ is determined
by its image in $\widetilde{H}_2(S^{[n]},\BZ)$.
\end{lemma}
\begin{proof}
For any $a \in H^1(S,\BQ)$ consider the cohomology class
\[ D(a) = \frac{1}{(n-1)!} \Fq_1(a) \Fq_1( 1)^{n-1} \vacuum \quad \in H^1(S^{[n]}). \]
Since $\p^1$ does not have any odd cohomology, $f^{\ast} D(a) = 0$.
On the other hand,
\begin{equation} \label{eqn:D(a) times D(b)} D(a) \cup D(b) = \frac{1}{(n-1)!} \Fq_1(ab) \Fq_1( 1)^{n-1} \vacuum + \frac{1}{(n-2)!} \Fq_1(a) \Fq_1(b) \Fq_1( 1)^{n-2} \vacuum. \end{equation}
Hence writing $f_{\ast} [\p^1] = \beta + \xi + kA$ with $\xi \in \wedge^2 H_1(S,\BZ)$ we get the relation
\[ 0 = \langle D(a) \cup D(b), f_{\ast} [\p^1] \rangle = \int_S \beta \cup a \cup b + \langle a \wedge b, \xi \rangle. \]
Since $\wedge^2 H^1(S,\BZ)$ is dual to $\wedge^2 H_1(S,\BZ)$ we see that $\beta$ already determines the component $\xi$.
\end{proof}

Moreover, the quantum multiplication with $D(\gamma)$ for $\gamma \in H^2(S)$ already determines the quantum multiplication by $D(a) D(b)$ for $a,b \in H^1(S)$,
so we do not need to consider the latter one.

\subsection{Gromov-Witten invariants}
Let
\[ \Mbar_{g,N}(S^{[n]}, \beta+kA) \]
be the moduli space of connected $N$-marked genus $h$ stable maps $f : C \to S^{[n]}$ satisfying
\[ f_{\ast}[C] = \beta+kA \in \widetilde{H}_2(S^{[n]},\BZ). \]

The Gromov-Witten invariants of $S^{[n]}$ are defined by integration over the virtual class:
\[
\left\langle \, \lambda_1, \ldots, \lambda_N \right\rangle^{S^{[n]}}_{h,\beta+kA}
=
\int_{[ \Mbar_{h,N}(S^{[n]}, \beta+kA) ]^{\vir}} \prod_{i=1}^{N} \ev_{i}^{\ast}(\lambda_i),
\]
where $\ev_i : \Mbar_{h,N}(S^{[n]}, \beta+kA) \to S^{[n]}$ are the evaluation maps at the $i$-th marking.

\subsection{Main results}
Let $\Sigma$ be a smooth curve and let
\[ \pi : S \to \Sigma \]
be an elliptic surface with section $\Sigma \subset S$. Let $f \in H^2(S)$ denote the class of a fiber.

If $\beta \in H_2(S,\BZ)$ satisfies $\pi_{\ast} \beta = 0$ but is not a multiple of $f$, then after a suitable deformation of $S$, $\beta$ is no longer effective. It follows that all Gromov-Witten invariants of $S^{[n]}$ in class $\beta + kA$ vanish, see Proposition~\ref{prop:vanishing Hilbert scheme} for details.
In other words, if $\gamma \in \widetilde{H}_2(S^{[n]},\BZ)$ is a class satisfying $\pi^{[n]}_{\ast} \gamma = 0$, then we can have non-zero Gromov-Witten invariants for this class only in case $\gamma = df + kA$ for some $d \geq 0$ and $k \in \BZ$ (with $k \geq 0$ if $d=0$). Hence we restrict to this case below.

Consider the operator
\begin{equation} \QHilb \in \mathrm{End}( \Fock_S \otimes \BQ((p))[[q]]) \label{QHilb} \end{equation}
defined for all $\lambda, \mu \in H^{\ast}(S^{[n]})$ by
\begin{equation} \label{defn:Q hilb}
( \QHilb \lambda, \mu) = 
\sum_{d \geq 0} \sum_{k \in \BZ} q^d (-p)^k \langle \lambda, \mu \rangle^{S^{[n]}}_{0, df+kA}.
\end{equation}

For any class $\gamma \in H^2(S)$ and integers $b_1, \ldots, b_r \in \BZ$ define
a class in $H^{2r+2}(S^r)$ by
\begin{equation} \label{star class}
\star^{b_1,\ldots, b_r}(\gamma)
:=
\sum_{i=1}^{r} b_i^2 \pr_i^{\ast}(\pt) \prod_{\ell \neq i} \pr_{\ell}^{\ast}(\gamma)
- \sum_{1 \leq i < j \leq r} b_i b_j \pr_{ij}^{\ast}(\Delta^{\text{odd}}_{\ast}(\gamma))
\prod_{\ell \neq i, j } \pr_{\ell}^{\ast}(\gamma),
\end{equation}
where
\begin{itemize}
\item $\pr_i, \pr_{ij}$ are the projections from $S^r$ to the $i$-th and $ij$-th factor respectively, 
\item $\Delta^{\text{odd}}_{\ast}(\gamma)$ is the component of
$\Delta_{\ast}(\gamma) \in H^{\ast}(S \times S)$ which lies in $H^{\text{odd}}(S) \otimes H^{\text{odd}}(S)$.
\end{itemize}

For any $\gamma \in H^2(S)$ define the operator
\begin{equation} \label{omega gamma}
\omega_{\gamma}(p) = \sum_{\substack{b_1, \ldots, b_r \in \BZ \\ b_1 + \ldots + b_r = 0}}
\frac{1}{r!} \prod_{i=1}^{r} \frac{(p^{b_i/2} - p^{-b_i/2})}{b_i}
\ : \Fq_{b_1} \cdots \Fq_{b_r}( \star^{b_1, \ldots, b_r}(\gamma) ) :
\end{equation}

The first main result of the paper is the following:
\begin{thm} \label{thm:2point operator}
\begin{align*} \QHilb & = \sum_{k > 0} \ln\left( \frac{1-p^k}{1-p} \right) \Fq_k \Fq_{-k}(\Delta_{\ast} K_S) \\
& \ - \left( \int_{\Sigma} K_S \right) \sum_{m,d \geq 1} \frac{q^{md}}{md}
\left[ \omega_{df}(p^m) + (p^{m/2} - p^{-m/2})^2 e_{df} \right] \end{align*}
\end{thm}

The $\pi$-restricted quantum product $\ast_{\pi}$ on $H^{\ast}(S^{[n]}) \otimes \BQ((p))[[q]]$ takes the form
\[
( \lambda_1 \ast_\pi \lambda_2, \lambda_3)_{S^{[n]}} = 
\sum_{d \geq 0} \sum_{k \in \BZ}
q^d (-p)^k
\left\langle \, \lambda_1, \lambda_2, \lambda_3 \right\rangle^{S^{[n]}}_{0,\beta}.
\]
For $x \in H^2(S)$ consider the operators
\[ E_{x}, E_{\delta} \in \mathrm{End}( \Fock_S \otimes \BQ((p))[[q]] )
\]
which act on $H^{\ast}(S^{[n]})$ by $\ast_{\pi}$-multiplication by $D(x)$ and $\delta$ respecitvely.

\begin{cor} \label{cor:quantum multiplication by divisor} For any $\alpha \in H^2(S)$ we have
\begin{align*}
E_{\alpha} & = e_{\alpha} + (\alpha \cdot f) q \frac{d}{dq} \QHilb \\
E_{\delta} & = e_{\delta} + p \frac{d}{dp} \QHilb.
\end{align*}
\end{cor}
\begin{proof}
The result follows from Theorem~\ref{thm:2point operator} and the divisor equation.
\end{proof}

\begin{rmk}
Whenever\footnote{See Section~\ref{sec:invariants of elliptic surfaces} for when this is the case.} $p_g(S) > 0$,
the Hilbert scheme $S^{[n]}$ has non-zero Gromov-Witten invariants only for curve classes $a c_1(S) + kA$ by \cite{HLQ}, so the restricted quantum multiplication $\ast_{\pi}$ equals the full quantum multiplication $\ast$ in this case.
\end{rmk}

\begin{rmk}
For any surface $S$, define the extremal quantum product $\ast_\mathrm{ext}$ on $H^{\ast}(S^{[n]}) \otimes \BQ[[p]]$ by
\[
( \lambda_1 \ast_\mathrm{ext} \lambda_2, \lambda_3)_{S^{[n]}} = 
\sum_{k \geq 0}
(-p)^k
\left\langle \, \lambda_1, \lambda_2, \lambda_3 \right\rangle^{S^{[n]}}_{0,kA}.
\]
Define the operator
$E_{\delta}^{\mathrm{ext}} \in \mathrm{End}(\Fock_S \otimes \BQ[[p]])$
which acts on $H^{\ast}(S^{[n]})$ by 
$\ast_{\mathrm{ext}}$-multiplication by $\delta$.
Then J. Li and W.-P. Li proved that for any simply connected surface $S$ we have
\[
E_{\delta}^{\mathrm{ext}} = 
-\frac{1}{6} \sum_{i+j+k=0} :\! \Fq_i \Fq_j \Fq_k ( \Delta_{123} )\!:\ + \sum_{k>0} \left( \frac{k}{2} \cdot \frac{p^k+1}{p^k-1} - \frac{1}{2} \cdot \frac{p+1}{p-1} \right) \Fq_k \Fq_{-k}( \Delta_{\ast}K_S )
\]
Corollary~\ref{cor:quantum multiplication by divisor} specializes to this evaluation under $q=0$.
\end{rmk}

\subsection{The local elliptic curve} \label{subsec:formula for local elliptic curve}
Consider the elliptic surface 
\[ S = E \times \BC. \]
We view cohomology classes of $E$ as cohomology classes on $S$ via pullback.
The torus $\BG_m$ acts on $\BC$ with tangent weight $t$ at the origin and thus on $S$ and its Hilbert schemes.
The induced action on the moduli space of stable maps to $S^{[n]}$ has proper fixed locus. The Gromov-Witten invariants of $S^{[n]}$ are defined as the corresponding equivariant residue.
We define the $2$-point operator
\[ \QHilb \in \mathrm{End}( H^{\ast}(\Hilb) \otimes \BQ(t) \otimes \BQ((p))[[q]]). \]
exactly as before.

We state a complete evaluation of this operator.
Define the equivariant operators
\[
e_{t} = -t^2 \sum_{k > 0} \Fq_{k} \Fq_{-k} ( \Delta_E )
\]
where $\Delta_E \in H^{\ast}(E \times E)$ is the class Poincar\'e dual to the diagonal,
and let $\omega_{dt}(p)$ be defined by the same formula as $\omega_{\gamma}(p)$
but using the equivariant class
\[
\star^{b_1, \ldots, b_r}(dt) = d^{r-1} t^r
\left( \sum_{i=1}^{r} b_i^2 \pr_i^{\ast}(\pt_E) - \sum_{1 \leq i < j \leq r}
b_i b_j \pr_{ij}^{\ast}( \Delta_E^{\text{odd}} ) \right)
\]
where $\pt_E \in H^2(E)$ is the class of a point.

The second main result of the paper is the following:
\begin{thm} \label{thm:ExC}
Let $S = E \times \BC$. Then we have:
\begin{align*} 
\QHilb & = -t^2 \sum_{k > 0} \ln\left( \frac{1-p^k}{1-p} \right) \Fq_k \Fq_{-k}(\Delta_{E}) \\
& \ + \sum_{m,d \geq 1} \frac{q^{md}}{md}
\left[ \omega_{dt}(p^m) + (p^{m/2} - p^{-m/2})^2 e_{dt} \right].
\end{align*}
\end{thm}

This determines the quantum multiplication by divisor classes on $(E \times \BC)^{[n]}$ for all $n$.

\section{Modular and Jacobi forms} \label{sec:Modular and Jacobi forms}
For even $k \geq 2$ the weight $k$ Eisenstein series are defined by
\[ G_k(q) = - \frac{B_k}{2 \cdot k} + \sum_{n \geq 1} \sum_{d|n} d^{k-1} q^n, \]
where $B_n$ are the Bernoulli numbers defined here by the expansion\footnote{This convention only differs from the standard $z/(e^z-1)$ definition in the value of $B_1$, which we take to be zero.}
\[ \frac{z}{2}\cdot \frac{e^z+1}{e^z-1} = \sum_{n \geq 0} B_n \frac{z^{n}}{n!}. \]
In particular, $B_1=0$, $B_2=1/6$, and $B_4=-1/30$.
For odd $k$ we set $G_k=0$.

The algebra of quasi-modular forms is
\[ \QMod = \bigoplus_{k \geq 0} \QMod_k = \BC[G_2, G_4, G_6], \]
where the grading is by weight. We have 
$G_k \in \QMod_k$ for all $k$.

The differential operators $D_q = q \frac{d}{dq}$ acts on $\QMod$
and increases the weight by $2$.
Moreover, viewing a quasi-modular form as a formal polynomial in $G_2, G_4, G_6$ also $\frac{d}{dG_2}$ acts by lowering the weight by $2$.
Let $\wt \in \mathrm{End}(\QMod)$ be the operator which acts on $\QMod_k$ by multiplication by $k$. Then we have the commutation relation
\[ \left[ \frac{d}{dG_2}, D_q \right] = -2 \wt. \]

Let $z \in \BC$ and $p=e^z$. Consider the normalized Jacobi theta function
\begin{equation} \label{theta function}
\Theta(z) 
=  (p^{1/2}-p^{-1/2})\prod_{m\geq 1} \frac{(1-pq^m)(1-p^{-1}q^m)}{(1-q^m)^{2}}.
\end{equation}
Define weight $n$ functions $\A_n$ for all $n \in \BZ$ by the Taylor expansion:
\begin{equation} \label{function}
\frac{\Theta(z+w)}{\Theta(z) \Theta(w)}
=
\frac{1}{w} + \sum_{n \geq 1} \frac{\A_n(z,\tau)}{(n-1)!} w^{n-1}.
\end{equation}
By a result of Libgober \cite{Lib} the algebra of meromorphic quasi-Jacobi forms of index zero with poles in $z$ only at lattice points is the free polynomial algebra
\[ \QJac = \bigoplus_{k \geq 0} \QJac_k = \BQ[ \A_1, \A_2, \A_3, G_2, G_4], \]
where the grading is by weight of the generators. Moreover $\A_n \in \QJac_n$ for all $n$.

\begin{thm}[{\cite{Zagier}}] \label{thm:2132}
For all $n \geq 1$ the functions $\A_n$ satisfy the following properties:
\begin{enumerate}
\item[(i)] We have the Fourier expansion
\[
\A_n(z,\tau) = \frac{B_n}{n} + \delta_{n,1}\frac{1}{2} \frac{p+1}{p - 1} - \sum_{k,d \geq 1} d^{n-1} (p^{k} + (-1)^n p^{-k}) q^{kd} .
\]
\item[(ii)] Around $z=0$ we have the Taylor expansion
\begin{align*}
\A_n(z) = 
\frac{1}{z} \delta_{n,1} & -2 \sum_{0 \leq m < n-1}
\frac{z^m}{m!} \left(q \frac{d}{dq} \right)^m G_{n-m}(q) \\
& -2 \sum_{g \geq 1} \frac{z^{2g-2+n}}{(2g-2+n)!} \left(q \frac{d}{dq} \right)^{n-1} G_{2g}(q)
\end{align*}
\item[(iii)] The following differential equation holds:
\[
q \frac{d}{dq} \A_n = p \frac{d}{dp} \A_{n+1}.
\]
\end{enumerate}

\end{thm}
\begin{proof}
Part (i) and (ii) are proven in \cite[Sec.3]{Zagier},
part (iii) is immediate from (i). 
\end{proof}

\begin{example}
\begin{gather*} \A_1 = 
p \frac{d}{dp} \log \Theta(z)
= \frac{1}{z} - 2 \sum_{k \geq 1} G_{k}(\tau) \frac{z^{k-1}}{(k-1)!} \\
\mathsf{A}_1 = -\frac{1}{2} + \frac{p}{p-1} - \sum_{r \geq 1} \left[ \frac{pq^r}{1 - pq^r} - \frac{p^{-1}q^r}{1 - p^{-1} q^r} \right].
\end{gather*}
\end{example}


\section{Relative GW and PT invariants}\label{sec:relativeGWPT}
Let $S$ be a smooth projective surface and let $C$ be a smooth curve of genus $h$.
Let $z_1, \ldots, z_N \in C$ be distinct points and write $z=(z_1, \ldots, z_N)$.
We consider the Gromov-Witten and Pandharipande-Thomas theory of the relative geometry
\begin{equation} (S \times C, S_z), \quad S_{z} = \bigsqcup_{i=1}^{N} S \times \{ z_i \}. \label{relative geometry fixed} \end{equation}

\subsection{Weighted partitions}
A $H^{\ast}(S)$-weighted partition or simply cohomology weighted partition is an ordered list of pairs
\begin{equation} \label{weighted partition} \big( (\lambda_1, \delta_1) , \ldots, (\lambda_{\ell}, \delta_{\ell} ) \big), \quad \delta_i \in H^{\ast}(S), \quad \lambda_i \geq 1.  \end{equation}
The partition underlying $\lambda$ is denoted by $\vec{\lambda} = (\lambda_1, \lambda_2, \ldots, \lambda_{\ell})$.


\begin{defn} \label{defn:coh weighted to Nakajima}
The class in $H^{\ast}(S^{[n]})$ associated to the $H^{\ast}(S)$-weighted partition $\lambda = \{ (\lambda_i, \delta_{i}) \}$ of size $n$ is defined by
\begin{equation} \lambda = \frac{1}{\prod_i \lambda_i} \prod_{i} \Fq_{\lambda_i}(\delta_i) \vacuum. \label{identification} \end{equation}
\end{defn}

\subsection{Gromov-Witten theory}
Let $\beta \in H_2(S,\BZ)$ be a curve class.
For $i \in \{ 1, \ldots, N \}$, consider $H^{\ast}(S)$-weighted partitions
\[ \lambda_i =  \big((\lambda_{i,j}, \delta_{i,j}) \big)_{j=1}^{\ell(\lambda_i)} \]
of size $n$ with underlying partition $\vec{\lambda}_i$.
Let
\begin{equation} \label{moduli space} \Mbar^{\bullet}_{g,r,(\beta,n), (\vec{\lambda}_1, \ldots, \vec{\lambda}_N)}(S \times C, S_z) \end{equation}
be the moduli space which parametrizes $r$-marked genus $g$ degree $(\beta,n)$ relative stable maps to the pair                                                                                                                                                                                                                                                                                                                                                                                                                                                                                                                                                                                                                                                                                                                                                                                                                                                                                                                                                                                                                                                                                                                                                                                                                                                                                                                                                                                                                                                                                                                                                                                                                                                                                                                                                                                                                                                                                                                                                                                                                                                                                                                                                                                                                                                                                                                                                                                                                                                                                                                                                                                                                                                                                                                                                                                                                                                                                                                                                                                                                                                                                                                                                                                                                                                                                                                                                                                                                                                                                                                                                                                                                                                                                                                                                                                                                                                                                                                                                                                                                                                                                                                                                                     \eqref{relative geometry fixed} with ordered ramification profiles $\vec{\lambda}_1, \ldots, \vec{\lambda}_N$ along $S_{z_1}, \ldots, S_{z_N}$, see \cite{Li1}. The source curve here is allowed to be disconnected but we exclude contracted components (this condition is denoted by the bullet $\bullet$).
The moduli space has relative and interior evaluation maps:
\begin{gather*}
\ev_{z_i,j} : \Mbar^{\bullet}_{g,r,(\beta,n), (\vec{\lambda}_1, \ldots, \vec{\lambda}_N)}(S \times C, S_z) \to S \\
\ev_j : \Mbar^{\bullet}_{g,r,(\beta,n), (\vec{\lambda}_1, \ldots, \vec{\lambda}_N)}(S \times C, S_z) \to S \times C.
\end{gather*}

Let $\gamma_j \in H^{\ast}(S \times C)$ be cohomology classes and $k_i \geq 0$.

We define the Gromov-Witten invariant:
\begin{multline*}
\left\langle \, \lambda_1, \ldots, \lambda_N \, \middle| \, \tau_{k_1}(\gamma_1) \cdots \tau_{k_r}(\gamma_n) \right\rangle^{\GW, (S \times C, S_z), \bullet}_{g, (\beta,n)} \\
=
\int_{[ \Mbar^{\bullet}_{g,r,(\beta,n), (\vec{\lambda}_1, \ldots, \vec{\lambda}_N)}(S \times C, S_z) ]^{\vir}}
\prod_{j=1}^{r} \ev_j^{\ast}(\gamma_j) \psi_j^{k_j} \prod_{i=1}^{N} \prod_{j=1}^{\ell(\lambda_i)}
\ev_{z_i, j}^{\ast}(\delta_{i,j}).
\end{multline*}
The partition function of Gromov-Witten invariants is\footnote{The complicated prefactor is a special case of a general formula, see e.g. \cite[Sec.5.2]{Marked}.}
\begin{multline} \label{defn:ZGW}
Z^{(S \times C, S_z)}_{\GW, (\beta,n)}\left( \lambda_1, \ldots, \lambda_N \middle| \tau_{k_1}(\gamma_1) \cdots \tau_{k_r}(\gamma_r) \right) 
= \\
(- \mathbf{i} z)^{d_{\beta}}
(-1)^{(1-h-N)n + \sum_i \ell(\lambda_i)}
z^{(2-2h-N)n + \sum_i \ell(\lambda_i)} \\
\cdot \sum_{g \in \BZ} (-1)^{g-1} z^{2g-2}
\left\langle \, \lambda_1, \ldots, \lambda_N \, \middle| \, \tau_{k_1}(\gamma_1) \cdots \tau_{k_r}(\gamma_r) \right\rangle^{\GW, (S \times C, S_z), \bullet}_{g, (\beta,n)}
\end{multline}
where $d_{\beta} = \int_{\beta} c_1(T_S)$ and $\mathbf{i} = \sqrt{-1}$.

\subsection{Pandharipande-Thomas theory}
Let 
\[ P_{\chi,(\beta,n)}(S \times C, S_z) \]
be the moduli space of 
relative stable pairs $(F,s)$ on \eqref{relative geometry fixed}
with $\chi(F)=\chi$ and $\pi_{S \times C \ast} \ch_2(F) = (\beta,n)$.

The moduli space has evaluation maps over the relative divisors
\[
\ev_{z_i} : P_{\chi,(\beta,n)}(S \times C, S_z) \to S^{[n]}.
\]

Consider also the universal target over the moduli space
\[ \pi : \CX \to P_{\chi,(\beta,n)}(S \times C, S_z), \quad \rho : \CX \to X = S \times C, \]
where $\rho$ is the map that forgets the stable pair and contracts the bubbles.
Let $(\BF,s)$ be the universal stable pair on $\CX$.
For $\gamma \in H^{\ast}(S \times C)$, we define  the descendent classes
\[ \ch_k(\gamma) := \pi_{\ast}( \ch_k(\BF) \cdot \rho^{\ast}(\gamma)). \]

We define the Pandharipande-Thomas invariant:
\[
\left\langle \, \lambda_1, \ldots, \lambda_N | {\textstyle \prod_{i=1}^{r} \ch_{k_i}(\gamma_i)} \right\rangle^{\PT, (S \times C, S_z)}_{\chi, (\beta,n)}
=
\int_{[ P_{\chi,(\beta,n)}(S \times C, S_z) ]^{\vir}} \prod_{i=1}^{r} \ch_{k_i}(\gamma_i) \prod_{i=1}^{N} \ev_{z_i}^{\ast}(\lambda_i).
\]
The partition function of these invariants is
\[
Z^{(S \times C, S_z)}_{\PT, (\beta,n)}\left( \lambda_1, \ldots, \lambda_N
| {\textstyle \prod_{i=1}^{r} \ch_{k_i}(\gamma_i)} \right) 
=
\sum_{m \in \frac{1}{2} \BZ} \mathbf{i}^{2m} p^m 
\left\langle \, \lambda_1, \ldots, \lambda_N | {\textstyle \prod_{i=1}^{r} \ch_{k_i}(\gamma_i)} \right\rangle^{\PT, (S \times C,S_z)}_{m+n(1-h) + \frac{1}{2} d_{\beta}, (\beta,n)}
\]

\subsection{GW/PT correspondence}
The Gromov-Witten (GW) and Pandharipande-Thomas (PT) invariants are conjectured to satisfy the following relation:
\begin{conj}[{\cite{MNOP1}, \cite{PP_GWPT_Toric3folds}}]
$Z^{(S \times C, S_z)}_{\PT, (\beta,n)}\left( \lambda_1, \ldots, \lambda_N | {\textstyle \prod_{i=1}^{r} \ch_{k_i}(\gamma_i)} \right)$
is the expansion of a rational function in $p$ and under the variable change $p=e^{z}$ we have
\[
Z^{(S \times C, S_z)}_{\PT, (\beta,n)}\left( \lambda_1, \ldots, \lambda_N | {\textstyle \prod_{i=1}^{r} \ch_{k_i}(\gamma_i)} \right) 
=
Z^{(S \times C, S_z)}_{\GW, (\beta,n)}\left( \lambda_1, \ldots, \lambda_N | \overline{ {\textstyle \prod_{i=1}^{r} \ch_{k_i}(\gamma_i)}} \right),
\]
where $\overline{ {\textstyle \prod_{i=1}^{r} \ch_{k_i}(\gamma_i)}}$ stands for the application of a universal correspondence matrix.
\end{conj}

\subsection{Family of curves}
We generalize the previous situation to a family of curves.

Let $\CC_{h,N} \to \CM_{h,N}$ be the universal curve over the moduli space 
of semistable curves (i.e. connected nodal curves without rational tails).
Let 
$z_1, \ldots, z_N : \CM_{h,N} \to \CC_{h,N}$
be the sections and write $z=(z_1,\ldots, z_N)$.
Consider the family of relative geometries
\begin{equation} (S \times \CC_{h,N}/\CM_{h,N}, S_z), \quad S_{z} = \bigsqcup S \times \{ z_i \}. \label{relative geometry} \end{equation}

As discussed for example in \cite{TP,Nesterov3} there exists moduli spaces of relative stable maps and relative stable pairs to the geometry \eqref{relative geometry}.
The moduli spaces have relative evaluation maps as before. The interior evaluation are chosen here to take values $S$, by projecting away from the curve factor (this is all we need).
The invariants and partition functions are defined exactly as before,
with $C$ replaced by $\CC_{h,N}$ in the notation.

A special case we will consider is the family $\CC_{0,2} \to \CM_{0,2}$.
In this case the moduli spaces parametrizes maps/pairs to the 'rubber' target $(S \times \p^1, S_{0,\infty})^{\sim}$, see for example \cite[Sec.2.4]{OPVir} or \cite[Sec.3.5]{Marked}. 
We recall the basic rigidification statement
which can be found in \cite[Sec.1.5.3]{MP} or \cite[Prop.3.12]{Marked}.
Let $\omega \in H^2(\p^1)$ be the class of a point.

\begin{prop}[Rigidification] \label{prop:rigidification}
Let $\gamma_1, \ldots, \gamma_r \in H^{\ast}(S)$. Then
for both $\dagger \in \{ \GW, \PT \}$ we have
\[
\left\langle \, \lambda_1, \lambda_2 \, \middle| \, \tau_{k_1}(\gamma_1) \cdots \tau_{k_r}(\gamma_n) \right\rangle^{\dagger, (S \times \CC_{0,2}, S_{0,\infty}), \bullet}_{g, (\beta,n)}
=
\left\langle \, \lambda_1, \lambda_2 \, \middle| \, \tau_{k_1}(\omega \gamma_1) \tau_{k_2}(\gamma_2) \cdots \tau_{k_r}(\gamma_n) \right\rangle^{\dagger, (S \times \p^1, S_{0,\infty}), \bullet}_{g, (\beta,n)}.
\]
\end{prop}

\section{Nesterov's Hilb/PT wall-crossing} \label{sec:HilbPT wallcrossing}
We have the following three enumerative theories associated to the surface $S$:
\begin{enumerate}
\item[(i)] Gromov-Witten theory of the pair $(S \times \CC_{h,N}/\CM_{h,N}, S_z)$
\item[(ii)] Pandharipande-Thomas theory of $(S \times \CC_{h,N}/\CM_{h,N}, S_z)$
\item[(iii)] Gromov-Witten theory of $S^{[n]}$ with $N$ markings of genus $h$.
\end{enumerate}
In the last section we saw the relationship between the theories (i) and (ii).\footnote{Although we did not state it explicitly here, a GW/PT correspondence for the relative geometry \eqref{relative geometry} is expected and was proven for $\BC^2$ in \cite{TP}.}
In this section we review the wall-crossing formula between (ii) and (iii) obtained by Nesterov \cite{Nesterov1} using the theory of quasimaps. Our discussion follows the survey paper \cite{Nesterov3}.

\subsection{A note on conventions}
Let $f : C \to S^{[n]}$ be a genus $h$ (quasi)-map to the Hilbert scheme where $C$ is a curve. By pulling back the universal family $\CZ \to S^{[n]}$ by $f$ we can associate a stable pair $(F,s)$ on $S \times C$.
Then we have the three numerical invariants
\[ \chi=\chi(F), \quad m=\ch_3(F), \quad f_{\ast}[C] = \beta + kA \in \tilde{H}_2(S^{[n]},\BZ). \]
The three numbers $\chi,m,k$ are related by (see e.g. \cite[Lemma 2]{HilbK3})
\[
\chi = k+n(1-h), \quad k = m + \frac{1}{2} d_{\beta}, \quad \chi = m + n(1-h) + \frac{1}{2} d_{\beta}
\]
In the stable pairs literature one usually indexes the moduli space of stable pairs by $\chi$ (because this notation works well also in the quasi-projective setting).
On the other hand, we have indexed here the generating series of stable pairs invariants by the Chern character $m = \ch_3(F)$,
because then the GW/PT correspondence and degeneration formula is simple to state \cite{Marked}.
Finally, in the Hilbert scheme literature the curve classes of the Hilbert scheme are usually indexed by $k$.
Below we have to translate between these three conventions.

\subsection{Wall-crossing formula}
The first step is to introduce the wall-crossing term (sometimes denoted $I$-function or vertex term). Consider the open substack
\[ F_{\chi, (\beta,n)} \subset P_{\chi,(\beta,n)}(S \times \p^1, S_{\infty}) \]
which parametrizes stable pairs on the relative geometry $(S \times \p^1, S_{\infty})$ which do not require bubbling over $\infty$.
There is a natural evaluation map over $\infty$,
\[ \ev: F_{\chi, (\beta,n)} \to S^{[n]}. \]
Let $\BC^{\ast}_z$ act on $\p^1$ with weight $1$ at $0 \in \p^1$ and let
$z = e_{\BC^{\ast}_z}(\BC_{\mathrm{std}})$
be the class of the weight $1$ representation.
We define the virtual class of $F_{\chi, (\beta,n)}$ by $\BC^{\ast}$-localization:
\[
[ F_{\chi, (\beta,n)} ]^{\vir} :=
\frac{ [ F_{\chi, (\beta,n)}^{\BC^{\ast}_z} ]^{\vir}}{ e( N^{\vir}) }
\in H_{\ast}( F_{\chi, (\beta,n)}^{\BC^{\ast}_z})[ z^{\pm} ].
\]

The truncated $I$-function for class $(\beta,k)$ is defined by
\[
\mu_{(\beta,k)}(z) = \left[ z \cdot 
\ev_{\ast} [ F_{k+n, (\beta,n)} ]^{\vir} \right]_{z^{\geq 0}}
\]

Define also a modified PT bracket, where we shift parameters and change notation:
\[
'\left\langle \, \lambda_1, \ldots, \lambda_N \right\rangle^{\PT}_{h, (\beta,k)}
\ := \
\left\langle \, \lambda_1, \ldots, \lambda_N \right\rangle^{\PT,(S \times \CC_{h,N}/\CM_{h,N}, S_z)}_{k+n(1-h), (\beta,n)}
\]

\begin{thm}[Nesterov] \label{thm:Nesterov}
If $(h,N) \notin \{ (0,0), (0,1) \}$, then
\begin{multline*}
'\left\langle \, \lambda_1, \ldots, \lambda_N \right\rangle^{\PT}_{h, (\beta,k)}
=
\left\langle \, \lambda_1, \ldots, \lambda_N \right\rangle^{S^{[n]}}_{h,\beta+kA} \\
+
\sum_{\substack{\ell \geq 1 \\ (\beta_0, \ldots, \beta_{\ell}) \\ (k_0, \ldots, k_{\ell}) \\ \sum_i \beta_i = \beta, \sum_i k_i = k}}
\frac{1}{\ell!}
\left\langle \, \lambda_1, \ldots, \lambda_N,
\mu_{(\beta_1,k_1)}( - \psi_{n+1} ), \ \ldots \ ,
\mu_{(\beta_{\ell},k_{\ell})}( - \psi_{n+\ell} ) \right\rangle^{S^{[n]}}_{h,\beta_0+k_0A}
\end{multline*}
\end{thm}

\subsection{Semi-positive case}
Assume that $S$ is a semi-positive target, that is
for all effective curve classes $\beta \in H_2(S,\BZ)$ we have
\[ d_{\beta} = \int_{\beta} c_1(S) \geq 0. \]

We will see that the correspondence in Theorem~\ref{thm:Nesterov} simplifies.
Let us give a name to the class
\[ I_{(\beta,k)} := \ev_{\ast} [ F_{k+n, (\beta,n)} ]^{\text{vir}} \]
appearing in the definition of the truncated $I$-function $\mu_{(\beta,k)}(z)$ above.
The (untruncated) $I$-function is then
the element in $H^{\ast}(S^{[n]})[[ z^{\pm}, t, y ]]$ defined by
\[ I(t,y,z) = 1 + \sum_{(\beta,k)>0} I_{(\beta,k)} t^{\beta} y^k
\]
where $(\beta,k)>0$ stands for $\beta > 0$, or $\beta=0$ and $k > 0$ (the case $\beta=0$ and $k=0$ corresponds to the space $F_{k+n,(\beta,n)}$ which parametrizes tubes and is excluded).

Since $[ F_{\chi,(\beta,n)} ]^{\vir}$ is of dimension $2n+d_{\beta}$,
the class $I_{(\beta,k)}$ is of dimension $-d_{\beta}$.
If $d_{\beta} \geq 0$ (as we assumed)
we obtain an expansion of the form
\[ I_{(\beta,k)} = I_{(\beta,k),0} + \frac{I_{(\beta,k),1}}{z} + \frac{I_{(\beta,k),2}}{z^2} + \ldots \]
where $I_{(\beta,k),r} \in H^{2( r -d_{\beta})}(S^{[n]})$.
Taking the generating series we get
\[ I(t,y,z) = I_0 + \frac{I_1}{z} + \frac{I_2}{z^2} + \ldots . \]
where
\begin{equation}
\begin{aligned}
\label{I_0 I_1}
I_0 & = 1 + \sum_{ \substack{ \beta,k:\ d_{\beta}=0 }} I_{(\beta,k),0} t^{\beta} y^k  \in H^0(S^{[n]}) \\
I_1 & = \underbrace{\sum_{ \substack{ \beta,k :\ d_{\beta}=0 }} I_{(\beta,k),1} t^{\beta} y^k}_{\in H^2(S^{[n]})} + \underbrace{\sum_{ \substack{ \beta,k :\ d_{\beta}=1 }} I_{(\beta,k),1} t^{\beta} y^k}_{\in H^0(S^{[n]}}.
\end{aligned}
\end{equation}

\begin{prop}  \label{prop:wallcrossing semipositive}
Let $S$ be semi-positive. 
\begin{itemize}
\item[(a)] If  $2h-2+N >0$ then
\[
I_0^{2h-2+N} \sum_{\beta,k} {'\left\langle \lambda_1, \ldots, \lambda_N \right\rangle^{\PT}_{h, (\beta,k)}} t^{\beta} y^k
=
\sum_{\beta, k} t^{\beta} y^k \exp\left( \frac{I_1}{I_0} \cdot (\beta + kA) \right) \left\langle \lambda_1, \ldots, \lambda_N \right\rangle^{S^{[n]}}_{h,\beta + kA}.
\]
\item[(b)] For $\deg_{\BC}(\lambda_1) + \deg_{\BC}(\lambda_2) \geq 2n-1$ we have
\[
\sum_{\beta,k} {'\left\langle \lambda_1, \lambda_2 \right\rangle^{\PT}_{0, (\beta,k)}} t^{\beta} y^k 
=
\sum_{(\beta, k)>0} t^{\beta} y^k \exp\left( \frac{I_1}{I_0} \cdot (\beta + kA) \right) \left\langle \lambda_1, \lambda_2 \right\rangle^{S^{[n]}}_{0,\beta + kA}
+ \frac{1}{I_0} \int_{S^{[n]}} \lambda_1 \lambda_2 I_1.
\]
\end{itemize}
\end{prop}

\begin{proof}
Let $\tilde{I}_0 = I_0 - 1$. Theorem~\ref{thm:Nesterov} can then be rewritten as
\begin{multline*} \sum_{\beta,k} \left\langle \lambda_1, \ldots, \lambda_N \right\rangle^{\PT}_{h, (\beta,k)} t^{\beta} y^k \\
=
\sum_{\beta', k'} \sum_{\ell_0, \ell_1 \geq 0} \frac{1}{\ell_0! \ell_1!}
\left\langle \lambda_1, \ldots, \lambda_N, -\psi_{N+1} \tilde{I}_0, \ldots, -\psi_{N+\ell_0} \tilde{I}_0, I_1, \ldots, I_1 \right\rangle^{S^{[n]}}_{h,\beta' + k'A} t^{\beta'} y^{k'}
\end{multline*}
Part (a) follows by the string and divisor equation and the identity:
\[ (-1)^{\ell} \binom{m+\ell-1}{\ell} = \binom{-m}{\ell}. \]
For part (b) we can apply string and divisor equation whenever $(\beta', k') > 0$.
In case $(\beta',k')=0$ we evaluate the right hand side directly: We need to have precisely one insertion of the form $I_1$, and then we can remove the $-\psi_{N+\ell} \widetilde{I}_0$ insertions using the string equation.
\end{proof}

\begin{cor} [{\cite[Prop.6.9]{Nesterov1}}] \label{cor:I function determination}
Let $\pt \in H^{4n}(S^{[n]})$ be the class of a point and $\gamma \in H_2(S^{[n]})$. Then
\begin{gather*}
I_0(t,y)^{-1} = \sum_{\beta,k} {'\langle \pt, 1, 1 \rangle^{\PT}_{0, (\beta,k)}} t^{\beta} y^k \\
\frac{\left[ I_1(t,y) \right]_{\deg = 0}}{I_0(t,y)} = \sum_{(\beta,k)>0} {'\langle \pt, 1 \rangle^{\PT}_{0,(\beta,k)}} t^{\beta} y^k \\
\frac{\gamma \cdot \left[ I_1(t,y) \right]_{\deg = 2}}{I_0(t,y)}
= 
\sum_{(\beta,k)>0} {'\langle \gamma, 1 \rangle^{\PT}_{0,(\beta,k)}} t^{\beta} y^k
\end{gather*}
\end{cor}
\begin{proof}
See also {\cite[Prop.6.9]{Nesterov1}}. One simply applies the previous prooposition in the first case with insertions $(\pt,1,1)$, and in the other case with insertions $(\pt,1)$ and $(\gamma,1)$,
and uses that in the presence of an insertion $1$ all primary invariants of $S^{[n]}$ for non-zero curve classes vanish.
\end{proof}

\begin{example} \label{example:Fano}
If $S$ is Fano, then one has \cite[Prop.6.10]{Nesterov1}
\begin{align*}
I_0(y) & = 1 \\
I_1(y) & = \log( 1 + y) D( c_1(S) ) - \frac{y}{1+y} \Big( \sum_{\substack{\alpha \in H_2(S,\BZ)_{>0} \\ c_1(S) \cdot \alpha = 1}} t^{\alpha} \Big) 1_{\Hilb_n}.
\end{align*}
For $2h-2+N>0$ the wall-crossing formula is then:
\[
\sum_{k} {'\left\langle \lambda_1, \ldots, \lambda_N \right\rangle^{\PT}_{h, (\beta,k)}} y^k
=
(1+y)^{d_{\beta}} 
\sum_{k} \left\langle \lambda_1, \ldots, \lambda_N \right\rangle^{S^{[n]}}_{h,\beta + kA} y^k
\]
\end{example}

\section{Elliptic curves}\label{sec:elliptic curves}
\subsection{Cohomology}
Let $E$ be a non-singular elliptic curve and let
$\1, \aaa,\bbb, \pt$
be a basis of $H^{\ast}(E)$ with the following properties:
\begin{itemize}
\item $\1 \in H^0(E)$ is the unit
\item $\aaa \in H^{1,0}(E)$ and $\bbb \in H^{0,1}(E)$ determine
a symplectic basis of $H^1(E)$,
\[ \int_E \aaa \cup \bbb = 1, \]
\item $\pt \in H^2(E)$ is the class Poincar\'e dual to a point.
\end{itemize}

\subsection{Gromov-Witten class}\label{subsec:GWE}
For $2g-2+n > 0$ let
\[ \tau : \Mbar_{g,n}(E,d) \to \Mbar_{g,n} \]
be the forgetful morphism. Define the Gromov-Witten class
\[
\CC_g(\gamma_1, \ldots, \gamma_n) =
\sum_{d \geq 0} \tau_{\ast}\left( [ \Mbar_{g,n}(E,d) ]^{\vir} \prod_{i=1}^{n} \ev_i^{\ast}(\gamma_i) \right) q^d \in H^*(\Mbar_{g,n})\otimes\BQ[[q]].
\]

\begin{lemma} \label{lemma:unit vanishing} $\CC_g(1^{\times n}) = 0$. \end{lemma}
\begin{proof}
This follows by the methods of
\cite[Sec.5.4]{OPVir}.
\end{proof}

\begin{thm}[\cite{HAE}] \label{thm:quasimodularity}
For any $\gamma_1, \ldots, \gamma_n \in H^{\ast}(E)$
the series
$\CC_g( \gamma_1, \ldots, \gamma_n )$
is a cycle-valued quasi-modular form:
\[
\CC_g( \gamma_1, \ldots, \gamma_n )
\in 
H^{\ast}(\Mbar_{g,n}) \otimes \QMod.
\]
\end{thm}
\begin{thm}[\cite{HAE}] \label{thm:HAE}
Considering $\CC_g( \gamma_1, \ldots, \gamma_n )$
as a polynomial in $G_2, G_4, G_6$ with coefficients cycles on $\Mbar_{g,n}$,
we have
\begin{align*}
\frac{d}{dG_2} \CC_g( \gamma_1, \ldots, \gamma_n )
\ =\  & 
\iota_{\ast} \CC_{g-1}( \gamma_1, \ldots, \gamma_n, \1 , \1 )
\\
& + \sum_{\substack{ g= g_1 + g_2 \\ \{1,\ldots, n\} = S_1 \sqcup S_2}}
j_{\ast} \left( \CC_{g_1}( \gamma_{S_1}, \1 ) \boxtimes
\CC_{g_2}( \gamma_{S_2}, \1 ) \right) \\
& - 2 \sum_{i=1}^{n} \left( \int_{E} \gamma_i \right) \CC_g( \gamma_1, \ldots, \gamma_{i-1}, \1, \gamma_{i+1}, \ldots, \gamma_n ) \cdot \psi_i,
\end{align*}
where $\gamma_{S_i} = ( \gamma_k )_{k \in S_i}$
and $\iota : \Mbar_{g-1,n+2} \to \Mbar_{g,n}$ and $j:\Mbar_{g_1,|S_1|+1} \times \Mbar_{g_2, |S_2|+1} \to \Mbar_{g,n}$ are the natural gluing maps.
\end{thm}

\subsection{Hodge products}

Let $\BE \to \Mbar_{g,n}$ be the rank $g$ Hodge bundle
and let $\lambda_i = c_i(\BE)$ be its Chern classes. The main result of this subsection, Proposition~\ref{prop:lambdaformula}, is an explicit formula (in terms of tautological classes) for the product
\[
\CC_g(\pt^{\times n}) \cdot \lambda_{g-1}.
\]
We will need some lemmas about such products (Lemmas~\ref{lemma:basic facts} and \ref{lem:conjstationary}), and along the way we also state a general conjecture of note about such products (Conjecture~\ref{conj:app}) and give some evidence supporting it.

\begin{lemma} \label{lemma:basic facts}
For all $g>0$ we have $\CC_g(\gamma_1, \ldots, \gamma_n) \cdot \lambda_g = 0$.
\end{lemma}

\begin{proof}
See also \cite[Lemma 4.4.1]{Pixton_Senior}.
Assume first $d>0$. Let $\CC \to \Mbar_{g,n}(E,d)$ be the universal curve and let $f : \CC \to E$ be the universal map. Let $dz \in \Gamma(E, \Omega_E)$ be the global holomorphic 1-form. Then $f^{\ast} dz$ defines a global nowhere vanishing section of (the pullback by $\tau$ of ) $\BE$.
Hence
\[ [\Mbar_{g,n}(E,d) ]^{\vir} \lambda_g = \frac{1}{d} \ev_{n+1 \ast}( c_g( \BE) f^{\ast}(\pt) [ \Mbar_{g,n+1}(E,d) ]^{\vir} ) = 0, \]
where  we suppressed the pullbacks by $\tau$.
In case $d=0$ we have $[ \Mbar_{g,n}(E,0) ]^{\vir} = [ \Mbar_{g,n} \times E] (-1)^g \lambda_g$, so the claim follows from Mumford's relation 
$c(\BE \oplus \BE^{\vee})=1$ which gives $\lambda_g^2=0$.
\end{proof}

We don't have this vanishing when multiplying by $\lambda_{g-1}$, but we can say something about which quasi-modular forms can appear.

Let $\QMod^E\subset \QMod$ denote the $\BC$-linear span of all derivatives of Eisenstein series
\[
\left( q \frac{d}{dq} \right)^{i} G_k(q),
\]
where $i\ge 0$ and $k\ge 2$ is even.

\begin{conj}\label{conj:app}
For any $2g-2+n > 0$ and $\gamma_1,\ldots,\gamma_n\in H^\ast(E)$,
\[
\CC_g(\gamma_1,\ldots,\gamma_n)\cdot\lambda_{g-1}\in H^{\ast}(\Mbar_{g,n}) \otimes \QMod^E.
\]
\end{conj}

We give evidence that this conjecture is likely true by checking it below both when $\gamma_i = \pt$ for all $i$ (which is all we need for our application) and numerically, i.e. after integration against an arbitrary tautological class of complementary codimension. Proving Conjecture~\ref{conj:app} when there are odd classes $\gamma_i$ seems more difficult.

\begin{lemma}\label{lem:conjstationary}
For any $2g-2+n > 0$,
\[
\CC_g(\pt^{\times n}) \cdot \lambda_{g-1} \in H^{\ast}(\Mbar_{g,n}) \otimes \QMod^E.
\]
\end{lemma}
\begin{proof}
Following the proof of the quasi-modularity in \cite{HAE}, we see that $\CC_g(\pt^{\times n})$ is written there as a sum over terms supported on stable graphs on $n$ vertices with no cut edges. 
Cupping with $\lambda_{g-1}$ kills all terms supported on a graph with more than one cycle, so only terms supported on the $n$-cycle graph remain. These terms only produce quasi-modular contributions that are derivatives of a single Eisenstein series, i.e. belong to $\QMod^E$.
\end{proof}

\begin{thm}[{\cite[Conjecture 4.4.5]{Pixton_Senior}}]\label{thm:numericalconj}
For any $2g-2+n > 0$, $\gamma_1,\ldots,\gamma_n\in H^\ast(E)$, and tautological class $\alpha\in RH^\ast(\Mbar_{g,n})$,
\[
\int_{\Mbar_{g,n}}\CC_g(\gamma_1,\ldots,\gamma_n)\lambda_{g-1}\alpha\in \QMod^E.
\]
\end{thm}
\begin{proof}
Recall that a general tautological class is a linear combination of basic classes of the form $(\xi_\Gamma)_*(\text{monomial in $\kappa,\psi$ classes})$, where $\Gamma$ is a stable graph and $\xi_\Gamma$ is the corresponding gluing map. We assume $\alpha$ is of this form.

We also assume the $\gamma_i$ are homogeneous. Let $\gamma_i$ have complex codimension $d_i$, i.e. $\gamma_i\in H^{2d_i}(E)$. Then $\CC_g(\gamma_1,\ldots,\gamma_n)$ has codimension $g-1+\sum d_i$, so we can assume $\alpha\in RH^{g-1+n-\sum d_i}(\Mbar_{g,n})$ since otherwise the integral vanishes for dimension reasons.

Next we note that this integral vanishes whenever in the stable graph $\Gamma$, the genus splits across vertices or is reduced by a self-node, since in either case the vanishing result Lemma~\ref{lemma:basic facts} will apply to some vertex (after splitting the $\lambda_{g-1}$ insertion over the vertices). In other words, we can assume the graph $\Gamma$ has a vertex $v_0$ of genus $g$ (i.e. $\Gamma$ is a rational tails graph). Let $n_0$ be the valence of $v_0$ (including legs).

Let $\alpha_0$ be the piece of $\alpha$ coming from $v_0$, i.e. $\alpha_0\in R^*(\Mbar_{g,n_0})$ is a monomial in kappa and psi classes there. Since $R^g(M_{g,n_0}) = 0$ 
 (and this can be achieved by tautological relations defined on $\Mbar_{g,n_0}$), we can assume that $\alpha_0$ has codimension at most $g-1$. The dimension of the class $\alpha$ is then at least $\dim\Mbar_{g,n_0}-(g-1) = 2g-2+n_0$, so the codimension of the class $\alpha$ is at most $g-1+n-n_0$. But we have already assumed this codimension is $g-1+n-\sum d_i$, so we conclude $n_0 \le \sum d_i$.

But all rational components are contracted by any map to $E$, so this means that the $n$ markings are mapped to only at most $n_0\le \sum{d_i}$ distinct points on $E$. This means the integral will vanish unless $n_0 = \sum d_i$ and the graph is such that the markings are partitioned into $n_0$ groups (the connected components if $v_0$ is deleted), each of which has insertion classes $\gamma_i$ with product equal to a multiple of $\pt$.

But then the Gromov-Witten splitting formula reduces the integral to (a multiple of)
\[
\int_{\Mbar_{g,n_0}}\lambda_{g-1}\CC_g(\pt^{\times n})\alpha_0.
\]

By Lemma~\ref{lem:conjstationary}, this belongs to $\QMod^E$ as desired.
\end{proof}

We now compute the product appearing in Lemma~\ref{lem:conjstationary} explicitly.

\begin{prop}\label{prop:lambdaformula}
We have 
\[
\CC_g(\pt^{\times n}) \cdot (-1)^{g-1} \lambda_{g-1}
= \alpha_{g,n} \left( q \frac{d}{dq} \right)^{n-1} G_{2g}(q).
\]
where the class $\alpha_{g,n} \in H^{\ast}(\Mbar_{g,n})$ is given by
\[
\alpha_{g,n} = \frac{(2g-1)!}{(2g-2+n)!}\cdot\frac{4g}{B_{2g}}\lambda_g\lambda_{g-1}\sum_{i=1}^n\psi_1\psi_2\cdots\widehat{\psi}_i\cdots\psi_n,
\]
where the product of psi classes omits $\psi_i$.
\end{prop}
\begin{proof}
By Lemma~\ref{lem:conjstationary} and the fact from \cite{HAE} that $\CC_g(\pt^{\times n})$ is of pure weight $2g-2+2n$, we have that
\[
\CC_g(\pt^{\times n}) \cdot (-1)^{g-1}\lambda_{g-1}
\in H^{\ast}(\Mbar_{g,n}) \otimes \mathrm{Span}\left( D_q^i G_{2g-2+2n-2i} \middle| i \geq 0 \right)
\]
where $D_q = q\frac{d}{dq}$. We hence can write
\begin{equation}\label{eq:alpha}
\CC_g(\pt^{\times n}) \cdot (-1)^{g-1}\lambda_{g-1}
= \alpha_0 G_{2g+2n-2} + \alpha_1 D_q G_{2g+2n-4} + \alpha_2 D_q^2 G_{2g+2n-6} + \ldots
\end{equation}
for classes $\alpha_i \in H^{\ast}(\Mbar_{g,n})$. It remains to compute the $\alpha_i$, which we claim are all zero except for $\alpha_{n-1} = \alpha_{g,n}$.

We now use the holomorphic anomaly equation (Theorem~\ref{thm:HAE}) to evaluate the expression
\[
\left(\frac{d}{dG_2}\right)^i \CC_g(\pt^{\times n}) \cdot (-1)^{g-1}\lambda_{g-1}
\]
for every $i$. The HAE has three terms, but we claim that the first and second will always vanish in this case each time the HAE is applied. Indeed, the first term reduces the genus and hence vanishes immediately by Lemma~\ref{lemma:basic facts}. The second term vanishes for the same reason if it splits the genus properly, so we can assume there that $g_1=0$ or $g_2=0$. The degree of the map on the genus $0$ component must then be $0$ to give a non-zero Gromov-Witten class, and at most one of the $\pt$ insertions can be on that component. For the initial application of the HAE
\[
\frac{d}{dG_2} \CC_g(\pt^{\times n}) \cdot (-1)^{g-1}\lambda_{g-1},
\]
this cannot happen because the genus $0$ component must have at least two of the markings for stability, but every marking has a $\pt$ insertion. For later applications of the HAE, some markings will have $\1$ insertions but they will also have psi classes causing these terms to vanish anyway.

So we are left with the third term of the HAE, which converts a $\pt$ insertion to a $\1$ insertion and multiplies the class by the corresponding psi class. There is also a scalar factor of $(-2)$. The conclusion is that for $0\le i \le n$ we have
\[
\left(\frac{d}{dG_2}\right)^i \CC_g(\pt^{\times n}) \cdot (-1)^{g-1}\lambda_{g-1}
= (-2)^ii!\sum_{\binom{n}{i}}\CC_g(\1^{\times i}\pt^{\times n-i})\psi_1\psi_2\cdots\psi_i \cdot (-1)^{g-1}\lambda_{g-1},
\]
where the sum runs over all ways to distribute the $\1$ and $\pt$ insertions over the $n$ points (and we multiply by the psi classes corresponding to the $\1$ insertions).

Since by Lemma~\ref{lemma:unit vanishing} we have $\CC_g(\1^s) = 0$, we find that
\begin{equation} \label{eq1}
\text{for all } i \geq n: \quad \left( \frac{d}{dG_2} \right)^i \CC_g(\pt^{\times n}) (-1)^{g-1}\lambda_{g-1} = 0.
\end{equation} 
Moreover, since a constant map can not go through two distinct points, we get
\begin{equation} \label{w0-fdsf}
\text{for all } i < n-1: \quad 
\mathrm{Coeff}_{q^0} \left[ \left( \frac{d}{dG_2} \right)^i \CC_g(\pt^{\times n})  (-1)^{g-1}\lambda_{g-1} \right] = 0.
\end{equation}
Finally, since $\mathrm{Coeff}_{q^0}\CC_g(\pt, \1^{n-1}) = (-1)^g\lambda_g$, we have
\begin{equation}\label{eq:nminus1}
\mathrm{Coeff}_{q^0} \left[ \left( \frac{d}{dG_2} \right)^{n-1} \CC_g(\pt^{\times n})  (-1)^{g-1}\lambda_{g-1} \right] = -(n-1)!(-2)^{n-1}\lambda_g\lambda_{g-1}\sum_{i=1}^n\psi_1\psi_2\cdots\widehat{\psi}_i\cdots\psi_n.
\end{equation}

We now compute the $\alpha_i$ appearing in \eqref{eq:alpha} for each $i$. Apply  $(d/dG_2)^{i}$ to both sides of \eqref{eq:alpha} and repeatedly use the commutation relation
$[d/dG_2, D_q]=-2 \wt$. The result is
\begin{multline*}
\left( \frac{d}{dG_2} \right)^i \CC_g(\pt^{\times n}) (-1)^{g-1}\lambda_{g-1} 
=
c_{g,n,i} \alpha_{i} G_{2g-2+2n-2i} \\
+ (\text{sum of terms involving a factor of } D_q^j(G_{m}) \text{ for } j>0),
\end{multline*}
where $c_{g,n,i} \in \BQ$ is a nonzero constant. Taking the $q^0$-coefficient and using \eqref{eq1} or \eqref{w0-fdsf}, we find $\alpha_i=0$ for $i\ne n-1$.

For $i=n-1$, this same approach lets us compute $\alpha_{n-1}$ by dividing the right side of \eqref{eq:nminus1} by the constant term $-\frac{B_{2g}}{4g}$ of the Eisenstein series $G_{2g}$ as well as the commutation constant
\[
c_{g,n,n-1} = (-2)^{n-1}(2g)(2g+(2g+2))\cdots(2g+\cdots+(2g+2n-4)) = (-2)^{n-1}\frac{(n-1)!(2g+n-2)!}{(2g-1)!}.
\]
The result is the class $\alpha_{g,n}$ in the statement of the proposition.
\end{proof}

\begin{example} \label{example:C_g lambda g-1}
For $n=1$ and $g \geq 1$ the previous proposition says
\[
\CC_g(\pt) (-1)^{g-1} \lambda_{g-1} = - \lambda_g \lambda_{g-1} E_{2g}(q).
\]
where $E_{2g} := -4g/B_{2g} G_{2g} = 1 + O(q)$ is a renormalized Eisenstein series.
\end{example}

\subsection{The double ramification cycle}
Let
\[ A = (a_1, a_2, \ldots, a_n), \ a_i \in \BZ \]
be a vector satisfying $\sum_{i=1}^{n} a_i = 0$.
The $a_i$ are the parts of $A$.
Let $\mu$ be the partition
defined by the positive parts of $A$,
and let $\nu$ be the partition defined by the negatives
of the negative parts of $A$.
Let $I$ be the set of markings corresponding
to the $0$ parts of $A$.
Let $\Mbar_{g,I}(\p^1,\mu,\nu)^{\sim}$
be the moduli space of stable relative maps of
connected curves of genus $g$ to rubber with ramification
profiles $\mu,\nu$ over the points $0,\infty \in \p^1$ respectively.
The moduli space admits a forgetful morphism
\[ \pi : \Mbar_{g,I}(\p^1,\mu,\nu)^{\sim}
\to \Mbar_{g,n}. \]
The double ramification cycle is the pushforward
\[
\DR_g(A)
= \pi_{\ast} 
\left[ \Mbar_{g,I}(\p^1,\mu,\nu)^{\sim} \right]^{\text{vir}}
\ \in A^g(\Mbar_{g,n}).
\]

Hain's formula \cite{Hain} is an explicit formula for the double ramification cycle after restriction to compact type curves $M^{ct}_{g,n}\subset \Mbar_{g,n}$:
\begin{equation}\label{eq:hain}
\DR_g^{\text{ct}}(A) = \frac{1}{g!}\Bigg(\sum_{i=1}^n\frac{a_i^2}{2}\psi_i-\sum_{h,S}\frac{a_S^2}{4}\delta_{h,S}\Bigg)^g,
\end{equation}
where the second sum runs over all $0\le h\le g$ and $S\subset\{1,2,\ldots,n\}$ defining a separating boundary divisor with class $\delta_{h,S} = [\Delta_{h,S}]$, and $a_S := \sum_{i\in S}a_i$.

\subsection{Evaluation}
The main result of this section is the following pair of integrals:

\begin{thm} \label{thm:evaluation}
For $g \geq 1$, $n \geq 1$ and $a_1, \ldots, a_n \neq 0$ we have:
\begin{multline*}
\int_{\Mbar_{g,n}} (-1)^{g-1} \lambda_{g-1} \CC_g(1, \pt^{n-1}) \DR_g(a_1, \ldots , a_{n}) \\
= \frac{a_1^2}{a_1 a_2 \ldots a_{n}} \sum_{\substack{S \subset \{1, ... , n\} 
}} (-1)^{|S|} a_S^{2g-2+n}
\frac{(-1)^{g-1+n}}{(2g - 2 + n)!} \left( q \frac{d}{dq} \right)^{n-2} G_{2g}(q)
\end{multline*}
and
\begin{multline*}
\int_{\Mbar_{g,n}} (-1)^{g-1} \lambda_{g-1} \CC_g(\alpha, \beta, \pt^{n-2}) \DR_g(a_1, \ldots , a_{n}) \\
= \frac{-a_1 a_2}{a_1 a_2 \ldots a_{n}} \sum_{\substack{S \subset \{1, ... , n\} 
}} (-1)^{|S|} a_S^{2g-2+n}
\frac{(-1)^{g-1+n}}{(2g - 2 + n)!} \left( q \frac{d}{dq} \right)^{n-2} G_{2g}(q),
\end{multline*}
where $a_S = \sum_{i \in S} a_i$.
\end{thm}


\subsection{Proof}
We begin with the first integral. For convenience, we replace $n$ with $n+1$ and label the marked points starting at $0$ instead of $1$, so we wish to compute the pairing of $(-1)^{g-1}\lambda_{g-1}\CC_g(1, \pt^n)$ with the DR cycle $\DR_g(a_0,\ldots,a_n)$. The first of these is simply the pullback along the map $\pi_0$ forgetting the $0$th marking of the cycle computed in Proposition~\ref{prop:lambdaformula}, so by the projection formula we can instead take the pushforward
\[
(\pi_0)_*\DR_g(a_0,\ldots,a_n)
\]
and pair it against the formula given in Proposition~\ref{prop:lambdaformula}.

Since $\lambda_g\lambda_{g-1}\psi_1\psi_2\ldots\psi_{n-1}$ vanishes on all boundary divisors in $\Mbar_{g,n}$, we can throw away almost all terms appearing in the pushforward $(\pi_0)_*\DR_g(a_0,\ldots,a_n)$. The surviving terms can be computed easily using Hain's formula \eqref{eq:hain} and can be grouped into three types:
\begin{enumerate}[(A)]
\item pushforwards of those terms appearing in
\[
\frac{1}{2^gg!}(a_0^2\psi_0+\cdots+a_n^2\psi_n)^g
\]
with at least one factor of $\psi_0$ (which is then converted to a kappa class by the pushforward);
\item pushforwards of
\[
\frac{1}{2^gg!}(a_1^2\psi_1+\cdots+a_n^2\psi_n)^g;
\]
\item pushforwards of those terms appearing in
\[
\frac{1}{2^gg!}(a_1^2\psi_1+\cdots+a_{j-1}^2\psi_{j-1}-(a_j+a_0)^2\delta_{0,\{0,j\}}+a_{j+1}^2\psi_{j+1}+\cdots+a_n^2\psi_n)^g
\]
for some $1\le j\le n$ with at least one factor of the rational tails boundary divisor $\Delta_{0,\{0,j\}}$ (which is then contracted by the pushforward).
\end{enumerate}

In each of these cases, we want to compute the pairing with the formula in Proposition~\ref{prop:lambdaformula}. This requires the standard socle evaluation formulas \cite{Getzler-Pandharipande}
\[
\int_{\Mbar_{g,n}}\lambda_g\lambda_{g-1}\kappa_{b_0}\psi_1^{b_1+1}\cdots\psi_n^{b_n+1} = \frac{(-1)^{g+1}(2g-2+n)!\cdot B_{2g}}{2^{2g-1}(2g)!\prod_{i=0}^n(2b_i+1)!!}
\]
and
\[
\int_{\Mbar_{g,n}}\lambda_g\lambda_{g-1}\psi_1^{b_1+1}\cdots\psi_n^{b_n+1} = \frac{(-1)^{g+1}(2g-3+n)!\cdot B_{2g}}{2^{2g-1}(2g)!\prod_{i=1}^n(2b_i+1)!!}
\]
(both valid when the integrand has the correct codimension and either all the $b_i$ are nonnegative or $b_j = -1$ for a single $j>0$ and all other $b_i$ are nonnegative).

Removing for now the common factors (from Proposition~\ref{prop:lambdaformula} and the socle evaluation formulas)
\begin{equation}\label{eq:factorstorage}
\left(\frac{(2g-1)!}{(2g-2+n)!}\cdot\frac{4g}{B_{2g}}\left( q \frac{d}{dq} \right)^{n-1} G_{2g}\right)\left(\frac{(-1)^{g+1}(2g-2+n)!\cdot B_{2g}}{2^{2g-1}(2g)!}\right) = \frac{(-1)^{g+1}}{2^{2g-2}}\left( q \frac{d}{dq} \right)^{n-1} G_{2g},
\end{equation}
we obtain the three terms
\begin{equation}
\frac{1}{2^g}\sum_{i=1}^n\sum_{\substack{b_0+\cdots+b_n=g\\ b_0 > 0}}\frac{a_0^{2b_0}\cdots a_n^{2b_n}}{b_0!\cdots b_n!}\cdot\frac{1}{(2b_0-1)!!(2b_1+1)!!\cdots(2b_n+1)!!}(2b_i+1); \tag{A}
\end{equation}
\begin{equation}
\frac{1}{2^g}\sum_{i=1}^n\sum_{\substack{b_1+\cdots+b_n=g\\ b_i > 0}}\frac{a_1^{2b_1}\cdots a_n^{2b_n}}{b_1!\cdots b_n!}\cdot\frac{1/(2g-2+n)}{(2b_1+1)!!\cdots(2b_n+1)!!}(2b_i+1)\left(-2+\sum_{j=1}^n(2b_j+1)\right); \tag{B}
\end{equation}
\begin{equation}
-\frac{1}{2^g}\sum_{i=1}^n\sum_{\substack{b_1+\cdots+b_n=g\\ b_i > 0}}\frac{a_1^{2b_1}\cdots (a_0+a_i)^{2b_i} \cdots a_n^{2b_n}}{b_1!\cdots b_n!}\cdot\frac{1/(2g-2+n)}{(2b_1+1)!!\cdots(2b_n+1)!!}(2b_i+1)\left(-2+\sum_{j=1}^n(2b_j+1)\right). \tag{C}
\end{equation}

We can simplify these terms by combining the factorials and double factorials in the denominators, adding terms B and C, and evaluating the sum over $j$ at the end of terms B and C to get $2g+n$. The resulting simplified terms are
\begin{equation}
\sum_{i=1}^n\sum_{\substack{b_0+\cdots+b_n=g\\ b_0 > 0}}\frac{a_0^{2b_0}\cdots a_n^{2b_n}}{(2b_0)!(2b_1+1)!\cdots (2b_n+1)!}(2b_i+1); \tag{A}
\end{equation}
\begin{equation}
\sum_{i=1}^n\sum_{b_1+\cdots+b_n=g}\frac{a_1^{2b_1}\cdots\left(a_i^{2b_i}-(a_0+a_i)^{2b_i}\right) \cdots a_n^{2b_n}}{(2b_1+1)!\cdots (2b_n+1)!}(2b_i+1). \tag{B+C}
\end{equation}

We now expand $(a_0+a_i)^{2b_i}$. When $a_0$ and $a_i$ have even exponent, this precisely cancels the contribution of the first term, so we are left with only terms with odd exponents, say $a_0^{2c_0+1}a_i^{2c_i+1}$, and we get that the sum of all three terms is
\begin{equation}
-\sum_{i=1}^n\sum_{c_0+b_1+\cdots+c_i+\cdots+b_n=g}\frac{a_0^{2c_0+1}a_1^{2b_1}\cdots a_i^{2c_i+1} \cdots a_n^{2b_n}}{(2c_0+1)!(2b_1+1)!\cdots (2c_i+1)! \cdots (2b_n+1)!} \tag{A+B+C}
\end{equation}

Setting $e_j = 2b_j+1$ (or $2c_j+1$ if $j = 0,i$) lets us rewrite this as
\begin{equation}
-\left(\frac{\sum_{i=1}^na_0a_i}{a_0a_1\cdots a_n}\right)\sum_{\substack{e_0+\cdots+e_n = 2g-2+n+1\\\text{$e_j>0$ all odd}}}\frac{a_0^{e_0}\cdots a_n^{e_n}}{e_0!\cdots e_n!}. \tag{A+B+C}
\end{equation}

This sum over $e_i$ is the odd part of the formal polynomial $(a_0+\cdots+a_n)^{2g-2+n+1}/(2g-2+n+1)!$, so it can be interpreted as an average over all ways to flip signs on the $a_i$ and we get
\begin{equation}
-\left(\frac{\sum_{i=1}^na_0a_i}{a_0a_1\cdots a_n}\right)\frac{1}{2^{n+1}}\sum_{S\subset\{0,1,\ldots,n\}}(-1)^{|S|}\frac{(a_0+\cdots+a_n-2a_S)^{2g-2+n+1}}{(2g-2+n+1)!}. \tag{A+B+C}
\end{equation}
Multiplying by the factors in \eqref{eq:factorstorage} and using $a_0+\cdots+a_n = 0$ to simplify yields the formula in the theorem statement (remembering that we replaced $n-1$ by $n$ there and reindexed the $a_i$).

This proves the first part. The second part follows from Proposition~\ref{prop:Ratio even odd} below\footnote{We could alternatively follow the same method as for the first integral, using Theorem~\ref{thm:numericalconj} in place of Lemma~\ref{lem:conjstationary} inside the proof of Proposition~\ref{prop:lambdaformula}.}. \qed

\begin{prop} \label{prop:Ratio even odd}
For $a_1, \ldots, a_n \neq 0$ with $a_1 \neq 0$. We have
\[ -\frac{a_2}{a_1} \int_{\Mbar_{g,n}} \lambda_{g-1} \CC_g(1, \pt^{n-1}) \DR_g(a_1, \ldots , a_{n}) =
\int_{\Mbar_{g,n}} \lambda_{g-1} \CC_g(\alpha, \beta, \pt^{n-2}) \DR_g(a_1, \ldots , a_{n}) \]
\end{prop}

For the proof we need the following two lemmas:
\begin{lemma} \label{lemma:study of N}
Let $a_1, \ldots, a_n \in \BZ$ be not all zero. The closed subscheme of $E^n$ defined by
\[ N_{a_1, \ldots, a_n} = \{ (x_1,..., x_n) \in E^n | a_1 x_1+ ... + a_n x_n=0 \in E \} \]
is smooth, pure of dimension $n-1$, has $\gcd(a_1,..,a_n)^2$ connected components $N_i$, and the classes $[N_i] \in H^{\ast}(E^n)$ are independent of $i$.
\end{lemma}
\begin{proof}
First, observe that for any $\varphi \in \mathrm{GL}_n(\BZ)$, the automorphism $\varphi : E^n \to E^n$ given by $\varphi(x_1, \ldots, x_n) = ( \sum_{j} \varphi_{ij} x_j )_{i=1}^{n}$ sends $N_{(a_1, \ldots, a_n)}$ to $N_{(a_1, \ldots, a_n) \varphi^{-1}}$.

Now let $g = \gcd(a_1,..,a_n)^2$. Any two vectors in $\BZ^n$ of the same divisibility are in the same $\mathrm{GL}_n(\BZ)$-orbit, so there exists $\varphi\in\mathrm{GL}_n(\BZ)$ with $(a_1, \ldots, a_n)\varphi^{-1} = (g,0,0,\ldots,0)$.

Thus it suffices to prove the lemma with $a_1 = g, a_2=a_3=\cdots=a_n=0$. But then $N_{(g,0,\ldots,0)}$ is just the union of slices $\{x\}\times E^{n-1}$, where $x$ is one of the $g^2$ $g$-torsion points of $E$, and the statement of the lemma is obvious.
\end{proof}
\begin{lemma} \label{lemma:N evaluation}
Let $\pr_i : E^n \to E$ be the projections. We have
\begin{gather*}
\int_{E^n} \pr_1^{\ast}(1) \pr_2^{\ast}(\pt) \cdots \pr_n^{\ast}(\pt) \cdot [N_{a_1,\ldots, a_n}] = a_1^2 \\
\int_{E^n} \pr_1^{\ast}(\alpha) \pr_2^{\ast}(\beta) \pr_3^{\ast}(\pt) \cdots \pr_n^{\ast}(\pt) \cdot [N_{a_1,\ldots, a_n}] = -a_1 a_2.
\end{gather*}
\end{lemma}
\begin{proof}
Let $\sigma : E^n \to E$ be given by $\sigma(x_1, \ldots, x_n)=\sum_i a_i x_i$, let $p_2, \ldots, p_n \in E$ be fixed points and let $\iota : E \to E^n$ be the inclusion $\iota(x) = (x,p_2, \ldots, p_n)$. We then have
\[
\int_{E^n} \pr_1^{\ast}(1) \pr_2^{\ast}(\pt) \cdots \pr_n^{\ast}(\pt) \cdot [N_{a_1,\ldots, a_n}] 
= \int_{E^n} \iota_{\ast}(1) \sigma^{\ast}( \pt )
= \int_{E} (\sigma \circ \iota)^{\ast}(\pt) = a_1^2
\]
where the last line follows since the map $\sigma \circ \iota(x) = a_1 x + \sum_i a_i p_i$ is \'etale of degree $a_1^2$.

For the second formula consider $\iota : E^2 \to E^n$ given by $\iota(x_1,x_2) = (x_1,x_2,p_3,\ldots, p_n)$ and $\sigma$ as before. Then
\begin{multline*}
\int_{E^n} \pr_1^{\ast}(\alpha) \pr_2^{\ast}(\beta) \pr_3^{\ast}(\pt) \cdots \pr_n^{\ast}(\pt) \cdot [N_{a_1,\ldots, a_n}]
=
\int_{E^n} \iota_{\ast}(\alpha \otimes \beta) \sigma^{\ast}(\pt)
=
\int_{E^2} (\alpha \otimes \beta) \cdot (\sigma \circ \iota)^{\ast}(\pt) \\
\overset{(*)}{=}
\int_{E^2} (m_{a_1} \times m_{-a_2})_{\ast}(\alpha \otimes \beta)
\cdot \sigma_0^{\ast}(\pt) = \int_{E^2} -a_1 a_2 (\alpha \otimes \beta) \cdot \Delta_{E} = -a_1 a_2 \int_E \alpha \cdot \beta = -a_1 a_2,
\end{multline*}
where in (*) we factor $\sigma \circ \iota = \sigma_0 \circ (m_{a_1}, m_{-a_2})$, with $m_a$ the multiplication by $a$ map and $\sigma_0:E^2 \to E$ given by $\sigma_0(x,y) = x-y + \sum_i a_i p_i$. 
\end{proof}

\begin{proof}[Proof of Proposition~\ref{prop:Ratio even odd}]
Let $\mu$ be the partition
defined by the positive parts of $A = (a_1, \ldots, a_n)$,
and let $\nu$ be the partition defined by the negative parts of $A$.
Let $m=|\mu|$ and let
\[ \Mbar_{g,(d,m), (\mu, \nu)}(E \times \p^1, E_{0,\infty})^{\sim} \]
be the moduli space of rubber relative stable maps to $(E \times \p^1, E_{0,\infty})$ of degree $(d,m)$
with prescribed ordered ramification profile $\mu$ over $0$ and $\nu$ over $\infty$.
Let $\tau$ denote the morphism to the moduli space of curves, and let $\ev_i$ denote the relative evaluation maps indexed by the parts of $A$.
Let $\gamma_i \in H^{\ast}(E)$.
By the product formula in relative Gromov-Witten theory \cite{LQ} we have
\begin{equation} \label{product formula ExP1}
\tau_{\ast}\left( [ \Mbar_{g,(d,m), (\mu, \nu)}(E \times \p^1, E_{0,\infty})^{\sim}  ]^{\vir} \prod_{i=1}^{n} \ev_i^{\ast}(\gamma_i) \right)
=
[\CC_g(\gamma_1, \ldots, \gamma_n)]_{q^d} \DR_g(a_1, \ldots, a_n).
\end{equation}

The key geometric fact that we use is that
the joint evaluation map
$\ev = \prod_{i=1}^{n} \ev_i$ factors as
\[
\ev : \Mbar_{g,(d,m), (\mu, \nu)}(E \times \p^1, E_{0,\infty})^{\sim} \to N_{a_1, \ldots, a_{n}} \subset E^n.
\]
Indeed, since we are over $\p^1$, for any relative stable map to an expansion $f : C \to (E \times \p^1)[\ell]$ we have $\CO(f^{-1}(E_0)) = \CO(f^{-1}(E_{\infty})) \in \Pic^{m}(E)$, and $f^{-1}(E_0) - f^{-1}(E_{\infty}) = \sum_i a_i p_i$.

Let $N_1, \ldots, N_{r}$ be the components of $N_{a_1, \ldots, a_{n}}$.
By dimension reasons, it follows that there are coefficients $c_1, \ldots, c_r \in \BQ$ such that
\[
\ev_{\ast}\left( [ \Mbar_{g,(d,m), (\mu, \nu)}(E \times \p^1, E_{0,\infty})^{\sim} ]^{\vir} \lambda_{g-1} \right) = \sum_{i=1}^{r} c_i [N_i].
\]
By Lemma~\ref{lemma:study of N} all $[N_i]$ have the same class in $E^n$, so with
$c=(c_1+ \ldots+c_r)/r$ we find
\begin{equation} \label{ev equation}
\ev_{\ast}\left( [ \Mbar_{g,(d,m), (\mu, \nu)}(E \times \p^1, E_{0,\infty})^{\sim} ]^{\vir} \lambda_{g-1} \right) = c [N] \ \in H^{\ast}(E^n).
\end{equation}

We obtain
\begin{multline*}
\int_{\Mbar_{g,n}} \lambda_{g-1} [\CC_g(1, \pt^{n-1})]_{q^d} \DR_g(a_1, \ldots , a_{n})
\,\overset{\eqref{product formula ExP1}}{=}\, \int_{[ \Mbar_{g,(d,m), (\mu, \nu)}(E \times \p^1, E_{0,\infty})^{\sim}  ]^{\vir}} \lambda_{g-1} \prod_{i=2}^{n} \ev_i^{\ast}(\pt) \\
\,\overset{\eqref{ev equation}}{=}\, c[N] \cdot \prod_{i=2}^{n} \pr_i^{\ast}(\pt) 
\,=\, a_1^2 c,
\end{multline*}
where the last equality follows by Lemma~\ref{lemma:N evaluation}, and similarly,
\[
\int_{\Mbar_{g,n}} \lambda_{g-1} [\CC_g(\alpha, \beta, \pt^{n-2})]_{q^d} \DR_g(a_1, \ldots , a_{n})
= c[N] \cdot \pr_1^{\ast}(\alpha) \pr_2^{\ast}(\beta) \prod_{i =3}^{n} \pr_i^{\ast}(\pt) = -a_1 a_2 c. \qedhere
\]
\end{proof}

\section{Elliptic surfaces: Theory}\label{sec:elliptic surfaces1}
\subsection{Invariants} \label{sec:invariants of elliptic surfaces}
Let $\pi: S \to \Sigma$ be an elliptic surface. We will assume throughout that there is a distinguished section $\iota : \Sigma \to S$. 
We refer to \cite{Miranda_Lectures} and \cite{FM} for introductions.

The basic invariants of the elliptic surface are
the genus $g:= g(\Sigma)$ of the base curve and the degree 
\[ d := d(S) := \frac{1}{12} \chi_{\mathrm{top}}(S). \]
The number of singular fibers of $\pi$, counted with multiplicity, is $12 d$.
The canonical bundle is the pullback of a degree $2g-2+d$ line bundle from the base. In particular,
\[ d_{\Sigma} = \int_{\Sigma} c_1(S) = 2-2g-d. \]
We have
\[
p_g(S) = h^0(K_S) =
\begin{cases}
g-1+d & \text{ if } S \text{ is not a product} \\
g + d & \text{ if } S \text{ is a product }.
\end{cases}
\]
and
\[
q(S) = h^1(\CO_S) =
\begin{cases}
g & \text{ if } S \text{ is not a product} \\
g + 1 & \text{ if } S \text{ is a product }.
\end{cases}
\]
\begin{example}
If $g(\Sigma) = 0$, then:
\begin{itemize}
\item $S=\p^1 \times E$ if $d=0$,
\item $S$ is a rational elliptic surface if $d=1$,
\item $S$ is a K3 surface if $d=2$.
\end{itemize}
If $g(\Sigma) = 1$, then
\begin{itemize}
\item $S$ is the product of two elliptic curves if $d=0$ and $\omega_S \cong \CO_S$,
\item $S$ is a bielliptic surface if $d=0$ and $\omega_S$ is torsion.
\end{itemize}
\end{example}

For any elliptic surface $S \to \Sigma$,
taking the fiberwise $j$-invariant defines a morphism $\Sigma \to \p^1$.

\begin{example} \label{example:elliptic surface with nonconstant j irreducible fibers}
Assume that $\Sigma$ admits a finite degree $2$ morphism to $\p^1$ (equivalently, $\Sigma$ has genus $g \leq 1$ or is hyperelliptic). Then
for any $d>0$ there exists an elliptic surface over $\Sigma$ of degree $d$ such that $S$ has a non-constant $j$-invariant and all of its fibers are irreducible, see \cite[p.62]{FM}.

Indeed, let us argue the case $d=1$, the other cases are similar. Consider the morphism $f : \Sigma \to \p^1$ of degree $2$. Let $\CL \in \Pic(\Sigma)$ be some line bundle such that $\CL^{\otimes 2} \cong f^{\ast} \CO(1)$. Then $\Gamma(\Sigma, \CL^4)$ contains $H^0(\p^1, \CO(2))$ and $\Gamma(\Sigma, \CL^6)$ 
contains $H^0(\p^1, \CO(3))$. Pick generic elements $g_2 \in H^0(\p^1, \CO(2))$ and $g_3 \in \Gamma(\p^1, \CO(3))$. Then $g_2, g_3$ are non-zero and have simple distinct zeros on $\Sigma$, so the {Weierstra\ss} model defined by
$(\Sigma,\CL, g_2, g_3)$ (see \cite{Miranda_Lectures}) is of the desired form.
\end{example}

\subsection{Monodromy and degenerations}
We recall basic results on the monodromy and degenerations of elliptic surfaces.

\begin{prop}[{L\"onne \cite{Lo1, Lo2}}] \label{prop:monodromy}
	Let $S \to \Sigma$ be an elliptic surface whose genus and degree satisfy $g=0$ and $d \geq 2$, or $g > 0$ and $d>0$. Let $L = H^2(S,\BZ) / \mathrm{Torsion}$. The subgroup of $O(L)$ generated by all monodromy operators of $S$ is 
	\[ O^{+}_{K_S}(L) = \{ g \in O(L) | g K_S = K_S \text{ and } g \text{ has real spinor norm } 1 \}. \]
\end{prop}

\begin{prop}[{Seiler, \cite[Thm.4.10]{FM}}]
\label{prop:Seiler}
Any two elliptic surfaces of the same genus and the same degree $d>0$ are deformation equivalent. 
\end{prop}

\begin{cor}\label{cor:integral fibers}
Any elliptic surface is deformation equivalent to one with integral fibers.
\end{cor}
\begin{proof}
If $d = 0$, the surface already has integral fibers (indeed smooth fibers). If $d > 0$, Proposition~\ref{prop:Seiler} says that the surface is deformation equivalent to one of those constructed in Example~\ref{example:elliptic surface with nonconstant j irreducible fibers} (since there exist hyperelliptic curves of every genus $g$).
\end{proof}

We give two examples of degenerations of elliptic surfaces that we need later on:
\begin{example} \label{example: higher genus higher degree degeneration}
Let $\Sigma$ be a curve with a degree $2$ morphism to $\p^1$, and let $S \to \Sigma$ be one of the elliptic surfaces of degree $d > 0$ constructed in Example~\ref{example:elliptic surface with nonconstant j irreducible fibers}.
Then there exists a deformation
\begin{equation} \label{degeneration to constant component} S \rightsquigarrow S_1 \cup_E S_2 \end{equation}
where $S_1 = \Sigma \times E \to \Sigma$ is a trivial elliptic surface,
$S_2 \to \p^1$ is an elliptic surface of degree $d$, and $E$ is a joint smooth fiber. 

Indeed, let $x \in \Sigma$ be a Weierstra{\ss} point, let $\CL = \CO(dx)$ and let $S \to \Sigma$ be an elliptic surface contructed as a Weierstra{\ss} model for $(\Sigma,\CL, g_2, g_3)$ for suitable $g_2, g_3$. Let $\CC = \mathrm{Bl}_{(x,0)}(\Sigma \times \BA^1) \to \BA^1$ be the degeneration to the normal cone of $x \in C$. Then the line bundle associated to $d$ times the proper transform of $\{ x \} \times \BA^1 \subset \Sigma \times \BA^1$ gives a degeneration of $\CL$ which specializes to $\CC_0 = \Sigma \cup \p^1$ as a degree $d$ line bundle over $\p^1$ and as a constant line bundle over $\Sigma$. Taking an associated Weierstra{\ss} model for the pair $(\CC, \CL)$ then yields the claim.

\begin{lemma} The degeneration \eqref{degeneration to constant component} constructed above has no vanishing cohomology. \end{lemma}
\begin{proof} 
Let $\CS \to \CC \to \BA^1$ be the total space of the degeneration.
We restrict the degeneration to a small disk $D \subset \BA^1$ around the origin such that $H^k(\CS) \cong H^k(S_0)$,
where $S_0 = S_1 \cup_E S_2$ is the special fiber.
We need to  show that $H^{\ast}(S_0) \cong H^{\ast}(\CS) \to H^{\ast}(S)$ is surjective, where the last morphism is induced by the restriction along a smooth fiber $S \hookrightarrow \CS$.
By the Mayer-Vietoris sequence one proves that $H^0(S_0) \cong \BQ$, $H^1(S_0) \cong H^1(\Sigma) = \BQ^{2g}$,
$H^2(S_0) \cong \BQ^{4g+12d-1}$ and $H^3(S_0) \cong \BQ^{2g+2}$.
Then the Clemens-Schmidt spectral sequence \cite{Morrison} gives
\[ 0 \to H^0(S_0) \to H_4(S_0) \to H^2(S_0) \cong H^2(\CS) \xrightarrow{\iota^{\ast}} H^2(S) \]
so the surjectivity of the last map follows by a simple dimension count.
For the other cases one can argue similarly, or alternatively observe that $H^1(S)$ is the pullback of $H^1(\Sigma)$ along $\pi$, and $H^3(S)$ is the pushforward of $H^1(\Sigma)$ along the section, which implies the claim immediately.
\end{proof}
\end{example}

\begin{example} \label{example:torsion L}
Let $\pi : S \to \Sigma$ be an elliptic surface of degree $0$, such that $\pi$ is not the trivial fibration.
We then claim that there exists an elliptic surface $S'\to \Sigma$ of some degree $d>0$, and a degeneration of elliptic surfaces
\[ S'\rightsquigarrow S \cup_E S_2 \]
where $S_2 \to \p^1$ is an elliptic surface of degree $d$.

Indeed, define the line bundle $\CL=(R^1 \pi_{\ast} \CO_S)^{\vee}$. By assumption $\CL$ is a torsion line bundle of order $m \in \{ 2, 3, 4, 6 \}$, see \cite[Prop.4.5]{FM}. Let $x \in \Sigma$ be a point, let $\CS = \mathrm{Bl}_{(x,0)}(\Sigma \times \BA^1) \to \BA^1$ be the degeneration to the normal cone of $x\in \Sigma$, let $p : \CS\to \Sigma$ be the projection, let $D$ be the proper transform of $\{ x \} \times \BA^1 \subset \Sigma \times \BA^1$ to $\CS$, and let $\widetilde{\CL} = p^{\ast}(\CL) \otimes \CO_{\CS}(d D)$. Then for $d \gg 0$, there exists sections $g_2 \in \Gamma( \CS, \widetilde{\CL}^{\otimes 4})$ and $g_3 \in \Gamma(\CS, \widetilde{\CL}^{\otimes 6})$ such that the Weierstra{\ss} model associated to $(\CS, \overline{\CL}, g_2, g_3)$ gives the desired degeneration.\footnote{We thank Nikolas Kuhn for pointing out this fact.}

Let $S_0 = S \cup_E S_2$ be the central fiber of the degeneration above.
After shrinking the base $\BA^1$,
by the Mayer-Vietoris sequence one sees that $H^k(\CS) \cong H^k(S_0) \to H^k(S)$ is surjective for all $k$.
\end{example}

\subsection{Gromov-Witten theory} \label{subsec:GW theory for elliptic surfaces}
Let $f \in H_2(S,\BZ)$ be the class of a fiber.
The moduli space $\Mbar_{g,n}(S,df)$ is of virtual dimension $g-1+n$.
Let $E \subset S$ be a smooth fiber of $\pi$, and let
\[ \iota : \Mbar_{g,n}(E,d) \to \Mbar_{g,n}(S,df) \]
be the induced inclusion.
We have the following description of the virtual class of $\Mbar_{g,n}(S,df)$:

\begin{prop} 
\label{prop:GW class elliptic surface}
For any $d \geq 1$ we have
in $H_{\ast}(\Mbar_{g,n}(S,df))$:
\[
[ \Mbar_{g,n}(S,df) ]^{\vir} = d_{\Sigma} \cdot \iota_{\ast}\left( [ \Mbar_{g,n}(E,d) ]^{\vir} (-1)^{g-1} \lambda_{g-1}\right).
\]
\end{prop}

\begin{proof}
For $g=0$, the virtual dimension of $\Mbar_{0}(S,df)$ is negative, so the virtual class vanishes. But $[\Mbar_{0,n}(S,df)]^{\vir}$ is the pullback of the virtual class of $\Mbar_{0}(S,df)$ by the morphism forgetting the points, so the latter one also vanishes. Hence the claim holds for $g=0$ and we may assume $g \geq 1$.

Consider first the case $S = \p^1 \times E$ which is elliptically fibered by the projection $\pi : S \to \p^1$.
Let $\BC^{\ast}_t$ be the torus which acts on $\p^1$ with $\BC^{\ast}$-weight $t$ at the tangent space of $0 \in \p^1$.
The localization formula expresses the equivariant virtual class as follows:
\begin{align*}
[ \Mbar_{g,n}(S,df) ]^{\vir}_{\BC^{\ast}}
& = \iota_{E_0, \ast} [ \Mbar_{g,n}(E,d) ]^{\vir} \frac{ (-1)^g \lambda_g + \ldots + t^g }{t} \\
& + \iota_{E_{\infty}, \ast} [ \Mbar_{g,n}(E,d) ]^{\vir} \frac{ (-1)^g \lambda_g + \ldots + (-t)^g }{-t}.
\end{align*}
Taking the non-equivariant limit $t \to 0$ implies our claim.

In the general case we argue by cosection localization.
Let $\theta$ be a rational section of $\omega_S$ such that
$\theta$ has zeros and poles of order $1$ along smooth distinct fibers of $\pi$.
The construction of such a section is simple: If $\omega_S = \pi^{\ast}(\omega_C \otimes \CL)$, then take a rational section of $\omega_C \otimes \CL$ with zeros and poles of order $1$ at points where $\pi$ is smooth, and then pull it back to $S$.
Let $p_1, \ldots, p_r \in \Sigma$ be the base points of the smooth fibers along which $\theta$ has zeros and poles. Write $E_p:= \pi^{-1}(p)$ for $p \in \Sigma$.

As explaind in \cite[Sec.6]{KL} (but see also \cite[Sec.3]{LiLi} where only a rational $2$-form is used) $\theta$ induces a rational cosection
of the obstruction sheaf on $\Mbar_{g,n}(S,df)$,
\[ \eta_{\theta} : \mathrm{Ob}_{\Mbar_{g,n}(S,df)} \to \CO_{\Mbar_{g,n}(S,df)} \]
The degeneracy locus of $\eta_{\theta}$ is by definition
\[ \Deg(\eta_{\theta}) = \{ (f,p_1,\ldots,p_n) \in \Mbar_{g,n}(S,df) \ |\ \text{either } \eta_{\theta} \text{ is not defined or is not surjective at } f \}. \]
Let $\pi_M : \Mbar_{g,n}(S,df) \to \Sigma$ be the map induced by the elliptic fibration $S \to \Sigma$. By \cite[Prop.6.4]{KL} 
the degeneracy locus $\Deg(\eta_{\theta})$ is contained in the closed subset
\[ \bigsqcup_{i=1}^{r} \pi_M^{-1}(p_i)
= \bigsqcup_{i=1}^{r} \iota_{E_{p_i}}( \Mbar_{g,n}(E_{p_i}, d) ). \]

Hence by cosection localization \cite[Theorem 1.1]{KL} there exists $\alpha_{p_i} \in \mathrm{H}_{\ast}(\Mbar_{g,n}(E_{p_i}, d))$ such that
\begin{equation} \label{22442} 
[ \Mbar_{g,n}(S,df) ]^{\vir} = 
\sum_{i=1}^{r} \iota_{E_{p_i} \ast}(\alpha_{p_i}).
\end{equation}

\vspace{7pt}
\noindent \textbf{Claim:}
\[ \alpha_{p_i} =
\begin{cases}
+ [ ( \Mbar_{g,n}(E_{p_i}, d) ]^{\vir} (-1)^{g-1} \lambda_{g-1} & \text{ if } \theta \text{ has a pole along } E_{p_i} \\
-[ ( \Mbar_{g,n}(E_{p_i}, d) ]^{\vir} (-1)^{g-1} \lambda_{g-1} & \text{ if } \theta \text{ has a zero along } E_{p_i}.
\end{cases} \]

\vspace{7pt}
Before proving the claim, let us show that it implies Proposition~\ref{prop:GW class elliptic surface}.
Namely for any $p \in \Sigma$ with $E_p$ smooth the class
\[ \iota_{E_p \ast}[ \Mbar_{g,n}(E_p,d) ]^{\vir} \in H_{\ast}(\Mbar_{g,n}(S,df)) \]
is independent of $p$ (for any two such $p,q \in \Sigma$, the $E_p, E_q$ lie in a smooth proper family and hence this follows from the properties of the virtual class).
Inserting the claim into \eqref{22442} gives thus
\[
[ \Mbar_{g,n}(S,df) ]^{\vir} = 
(b-a) \iota_{E \ast} [ \Mbar_{g,n}(E,d)]^{\vir} (-1)^{g-1} \lambda_{g-1},
\]
where $a,b$ are the number of fibers $E_{p_i}$ which are zeros and poles of $\theta$.
Since $a-b=K_S \cdot \Sigma=-d_{\Sigma}$ this gives the proposition.

We now prove the claim.
For the case where $\theta$ has a zero, this follows immediately from \cite[Theorem 1.1]{KL_low_degree}. However, we may argue both cases directly as follows.
By construction (compare \cite[Proof of Lemma 4.3]{LiLi})
the localized virtual cycle $\alpha_{p}$ is universal, in the sense 
that it depends on $S$ and $\theta$ only through whether $E_{p}$ is a zero or pole of $\theta$. 
We write $\alpha_{p}^{+}$ in the case of a pole, and $\alpha_{p}^{-}$ in the case of a zero. 
In particular, we can choose any surface to evaluate $\alpha_p^{\pm}$.
Denote $E := E_p$ and $\alpha^{\pm} := \alpha_p^{\pm}$. Let us consider $S = \p^1 \times E$. Since $\Mbar_{g,n}(\p^1 \times E, df) = \Mbar_{g,n}(E,d) \times \p^1$, we have the inclusion
\[ H_{\ast}(\Mbar_{g,n}(E,d)) \subset H_{\ast}( \p^1 \times E, df). \]
Thus inside $H_{\ast}(\Mbar_{g,n}(S,df))$ the equality \eqref{22442} together with the universality becomes
\[ [ \Mbar_{g,n}(S,df) ]^{\vir} = a \alpha^{-} + b \alpha^{+} \]
for any rational cosection $\theta$ with simple poles. 
By multiplying $\theta$ by a rational function with a simple zero and pole, and using that the result has to stay invariant, we get $\alpha^{-} = - \alpha^{+}$. Using $a-b=-2$ and our computation using the localization formula then proves the claim by:
\[ -2 \alpha^{-} = (a-b) \alpha^{-} = [ \Mbar_{g,n}(S, df) ]^{\vir} = 2 [ \Mbar_{g,n}(E, df) ]^{\vir} (-1)^{g-1} \lambda_{g-1}. \qedhere \]
\end{proof}

\begin{rmk}
On the level of Gromov-Witten invariants (i.e. numerically),
Proposition~\ref{prop:GW class elliptic surface} can be proven also by a degeneration argument parallel to the proof of Theorem~\ref{thm:GWPT correspondence} below.
\end{rmk}


\begin{prop} \label{prop:elliptic surface genus 0} For $2g-2+n > 0$ we have
\[
[ \Mbar_{g,n}(S,0)]^{\vir} =
\begin{cases}
[ \Mbar_{0,n} \times S] & \text{ if } g=0 \\
[ \Mbar_{1,n} \times S ] (c_2(S) - \lambda_1 c_1(S)) & \text{ if } g=1 \\
[ \Mbar_{g,n} \times S ] (- c_1(S) \lambda_g \lambda_{g-1} ) & \text{ if } g \geq 2.
\end{cases}
\]
\end{prop}
\begin{proof}
Immediate, see for example \cite{Getzler-Pandharipande}.
\end{proof}
\begin{prop} \label{prop:lambda g-1 invariants}
Let $\sigma(d) = \sum_{k|d} k$.
For all $d>0$ or $g \geq 2$ we have
\[
\int_{[ \Mbar_{g}(S,df) ]^{\vir}} \lambda_{g-1} = \delta_{g1} d_{\Sigma} \frac{\sigma(d)}{d}.
\]
\end{prop}
\begin{proof}
We apply Proposition~\ref{prop:GW class elliptic surface}.
In case $g>1$ we use the tautological relation $\lambda_{g-1}^2 = 2 \lambda_{g-2} \lambda_g$ (which is a special case of Mumford's relation)
and conclude the vanishing from Lemma~\ref{lemma:basic facts}.
In case $g=1$, we use the well-known computation $\langle 1 \rangle^E_{1,d} = \sigma(d)/d$, see \cite{Pixton_Senior}.
\end{proof}

\subsection{Vanishing results}
\begin{defn}
	A class $\beta \in H_2(S,\BZ)$ is called {\it vertical} if $\pi_{\ast} \beta = 0$.
\end{defn}
We collect two results on the vanishing of curve-counting invariants in vertical curve classes that are not multiples of the fiber class $f$.

Let $C$ be a smooth curve, or alternatively let $\CC_{h,N} \to \CM_{h,N}$ be the universal curve over the moduli space of semistable curves. Let $z=(z_1, \ldots, z_N) \in C^N$ be a tuple of distinct points or the tuple of sections, respectively.

\begin{prop} \label{prop:PT vanishing vertical classes not fiber}
	If $\beta \in H_2(S,\BZ)$ is vertical but not a multiple of $f$, then
	all GW or PT invariants of the relative geometries
	$(S \times C, S_{z})$ and $(S \times \CC_{h,N}, S_z)$ 
	in the curve classes $(\beta,n)$ vanish.
\end{prop}
\begin{proof}
By Corollary~\ref{cor:integral fibers} $S$ is deformation equivalent to an elliptic surface with only integral fibers. Hence after deformation $\beta$ is not effective and the invariants vanish by deformation invariance.
\end{proof}

Next we consider Gromov-Witten invariants of the Hilbert scheme of points on $S$.

\begin{prop} \label{prop:vanishing Hilbert scheme}
If $\beta \in H_2(S,\BZ)$ is vertical but not a multiple of $f$,
then all Gromov-Witten invariants $\langle \lambda_1, \ldots, \lambda_N\rangle^{S^{[n]}}_{h,\beta+kA}$ vanish.
\end{prop}
\begin{proof}
Let $\CZ \subset S^{[n]} \times S \to S^{[n]}$ be the universal subscheme.
Let $f : C \to S^{[n]}$ be a stable map on class $\beta + kA$.
The fiber product $\CZ \times_{S^{[n]}} C$ is a closed subscheme of $C \times S$
whose projection to $S$ is a curve of class $\beta$, see \cite[Lemma 1]{HilbK3}.
Hence if $\beta+kA$ is effective, then so is $\beta$.
However, arguing as in Proposition~\ref{prop:PT vanishing vertical classes not fiber} we see that $\beta$ is not effective after a deformation of $S$.
\end{proof}

\subsection{GW/PT correspondence}

Let $\lambda_1, \ldots, \lambda_N$ be $H^{\ast}(S)$-weighted partitions,
let $\gamma_i \in H^{\ast}(S)$ and let $k_i \geq 0$.
Let $\omega \in H^2(C)$ denote the class of a point.

\begin{thm} \label{thm:GWPT correspondence}
Let $d \geq 0$. 
The series
	$Z^{(S \times C, S_z)}_{\PT, (df,n)}\left( \lambda_1, \ldots, \lambda_N | {\textstyle \prod_{i=1}^{r} \ch_{k_i}(\omega \gamma_i)} \right)$
	is the expansion of a rational function in $p$ and under the variable change $p=e^{z}$ we have
	\[
	Z^{(S \times C, S_z)}_{\PT, (df,n)}\left( \lambda_1, \ldots, \lambda_N \middle| {\textstyle \prod_{i=1}^{r} \ch_{k_i}(\omega \gamma_i)} \right) 
	=
	Z^{(S \times C, S_z)}_{\GW, (df,n)}\left( \lambda_1, \ldots, \lambda_N \middle| \overline{ {\textstyle \prod_{i=1}^{r} \ch_{k_i}(\omega \gamma_i)}} \right),
	\]
\end{thm}

\vspace{7pt}
The proof of Theorem~\ref{thm:GWPT correspondence} will take the remainder of the section.
One strategy of proof would be to use cosection localization by a rational cosection parallel to the proof of Proposition~\ref{prop:GW class elliptic surface}. However, stable pairs can have disconnected support, so the usual degeneracy locus of the cosection is too large. A refined degeneracy locus as in \cite{LiLi} can be used, but the result is still messy.
Instead, we will use a degeneration argument.

Let $E \subset S$ be a smooth fiber of $\pi : S \to \Sigma$.
Consider the bi-relative geometry
\begin{equation} \label{birelative geometry} (S \times \p^1, S_{\infty} \cup (E \times \p^1)). \end{equation}
Since $S_{\infty}$ intersects $E \times \p^1$ non-trivially, the full GW/PT theory of this pair does not fit the classical framework of \cite{Li2, LiWu}
and instead requires logarithmic methods \cite{MR,MR2}.
However, here we only need to deal with curve classes $(\beta, n) \in H_2(S \times \p^1)$ where $\beta$ is vertical, so that
\[ (\beta,n) \cdot (E \times \p^1) = 0. \]
The moduli spaces of relative stable pairs and maps to \eqref{birelative geometry} have then been constructed in \cite[Sec.6.5]{PP_GWPT_Toric3folds}
by classical means.
As discussed there
moduli spaces, invariants, partition functions, and degeneration formulas work then parallel to \cite{Li2, LiWu}.
Also a GW/PT correspondence for \eqref{birelative geometry} for $\beta$ vertical can be formulated, see \cite{PP_GWPT_Toric3folds,Marked}.
Thus log-geometric constructions can be avoided.

By degenerating the curve $C$ and by using a basic invertibility result for the cap geometry $(S \times \p^1, S_{\infty})$ explained in \cite[Sec.6.2]{Marked} (which is based on the works \cite{PaPix_Japan},\cite{PP_GWPT_Toric3folds}) it suffices to prove Theorem~\ref{thm:GWPT correspondence} for the case $(S \times \p^1, S_{\infty})$.
We leave the details for this reduction step to the reader, but instead refer
to \cite[Sec.6]{PP_GWPT_Toric3folds} or \cite[Sec.9.2]{Marked} where identical arguments have been made for the $\CA_n$-surface and the K3 surface.

We hence from now on assume that $(S \times C,S_z) = (S \times \p^1, S_{\infty})$.

The plan of the proof of  Theorem~\ref{thm:GWPT correspondence} is as follows:
\begin{itemize}
\item the correspondence holds for the two cases $S = \p^1\times E$ and $S$ a rational elliptic surface (Proposition~\ref{prop:case d=0 and 1}, \cite{PaPix_GWPT});
\item the correspondence holds when $g(S) = 0$ by degeneration to copies of $R$ and a monodromy argument (Proposition~\ref{prop:GWPT for elliptic surface over P1});
\item the correspondence holds for trivial fibrations $S = \Sigma \times E$ by degeneration to copies of $\p^1\times E$ and a monodromy argument (Proposition~\ref{prop:GWPT trivial fibration});
\item at the end of the section, we prove the correspondence holds for an arbitrary elliptic fibration by degeneration to a fibration over $\p^1$ and a trivial fibration, as in Example~\ref{example: higher genus higher degree degeneration}.
\end{itemize}

\begin{prop} \label{prop:case d=0 and 1}
Theorem~\ref{thm:GWPT correspondence} holds for $S = \p^1 \times E$ and for $S$ a rational elliptic surface.
\end{prop}
\begin{proof}
The case $S = \p^1 \times E$ is a special case of \cite[Theorem 7]{PaPix_GWPT}.
The case of a rational elliptic surfaces follows from \cite[Theorem 2]{PaPix_GWPT}, since a rational elliptic surface is the blow-up of $\p^2$ along 9 points and hence deformation equivalent to a toric surface.
(The above references prove these claim equivariantly in the $\p^1$ direction, but as explained in \cite{PP_GWPT_Toric3folds}, the equivariant correspondence is always stronger then the non-equivariant limit.)
\end{proof}

\begin{prop} \label{prop:absolte relative}
Theorem~\ref{thm:GWPT correspondence} holds for an elliptic surface $S$
if and only if it holds (in the corresponding bi-relative setting) for $(S,E)$.
In this case, the correspondence also holds for $(S, E_{1} \cup \ldots \cup E_{\ell})$, where $E_1, \ldots, E_{\ell} \subset S$ are smooth fibers.
\end{prop}
\begin{proof}
Let $Y = E \times \p^1$ and consider the degeneration of $S$ to the normal cone of $E$, denoted by $S \rightsquigarrow S \cup_{E} Y$. 
By taking the product with $\p^1$, we obtain a 
degeneration of $S \times \p^1$. 
We lift all cohomology insertions from $S$ to the normal cone degeneration  by pullback (in particular, the degeneration has no vanishing cohomology).
The degeneration formula yields:
\begin{multline} \label{deg formula}
Z^{(S \times \p^1, S_\infty)}_{\PT, (df,n)}\left( \lambda | {\textstyle \prod_{i=1}^{r} \ch_{k_i}(\omega \gamma_i)} \right) 
= 
\sum_{d = d_1+d_2}
\sum_{\substack{\lambda = \lambda_1 \sqcup \lambda_2 \\
	\{ 1, \ldots, r \} = S_1 \sqcup S_2 }}  \\
Z^{(S \times \p^1, S_\infty \cup E \times \p^1)}_{\PT, (d_1 f,|\lambda_1|)}\left( \lambda_1 | \varnothing | {\textstyle \prod_{i\in S_1} \ch_{k_i}(\omega \gamma_i)} \right)
\cdot 
Z^{(Y \times \p^1, Y_\infty \cup E \times \p^1)}_{\PT, (d_2 f,|\lambda_2|)}\left( \lambda_2 | \varnothing | {\textstyle \prod_{i\in S_2} \ch_{k_i}(\omega \gamma_i|_{E})} \right),
\end{multline}
where $\lambda_1, \varnothing$ in the first factor on the right stands for the weighted partitions
inserted along the relative divisors $S_{\infty}$ and $E \times \p^1$ respectively, and similar for the second factor. Here the class associated to a partition are defined by the relaive Nakajima cycles, see \cite{PaPix_GWPT}.

Since there are no curves in class $(0,0) \in H_2(Y \times \p^1,\BZ)$,
one has
\begin{equation} \label{initial term 1}
Z^{(Y \times \p^1, Y_\infty \cup E \times \p^1),T}_{\PT, (0,0)}\left( \varnothing | \varnothing | 1 \right) = 1,
\end{equation}

The above formulas hold identically also on the Gromov-Witten side.

If we specialize to $S=E \times \p^1$ we have that both factors in \eqref{deg formula} are invariants of the same space.
Hence, if all cohomology insertions $\gamma_i$ are pulled back from $E$ we find:
\[
Z^{(S \times \p^1, S_\infty)}_{\PT, (df,n)}\left( \lambda | {\textstyle \prod_{i=1}^{r} \ch_{k_i}(\omega \gamma_i)} \right) 
=
2 \cdot Z^{(S \times \p^1, S_\infty \cup E \times \p^1)}_{\PT, (df,n)}\left( \lambda | \varnothing | {\textstyle \prod_{i=1}^{r} \ch_{k_i}(\omega \gamma_i)} \right) + ( \ldots ) \] 
where $( \cdots )$ stands for the terms in \eqref{deg formula} where $(\beta_i, n_i) > 0$ for both $i=1,2$, which therefore involves invariants only for curve classes of strictly lower degree.

If there is some $\gamma_i \in H^{\ast}(S)$ such that $\gamma_i|_{E} = 0$, then the full curve degree cannot be distributed to the $Y \times \p^1$ factor. Hence in this case:
\[
Z^{(S \times \p^1, S_\infty)}_{\PT, (df,n)}\left( \lambda | {\textstyle \prod_{i=1}^{r} \ch_{k_i}(\omega \gamma_i)} \right) 
=
Z^{(S \times \p^1, S_\infty \cup E \times \p^1)}_{\PT, (df,n)}\left( \lambda | \varnothing | {\textstyle \prod_{i=1}^{r} \ch_{k_i}(\omega \gamma_i)} \right) + ( \ldots ). \] 

Theorem~\ref{thm:GWPT correspondence} holds for $S = E \times \p^1$.
Moreover, as shown in \cite{PaPix_GWPT} the GW/PT correspondence is compatible with the degeneration formula.
By an induction on the degree of the curve classes,
it follows that Theorem~\ref{thm:GWPT correspondence} 
holds for the invariants
$Z^{(S \times \p^1, S_\infty \cup E \times \p^1)}_{\PT, (df,n)}\left( \lambda | \varnothing | {\textstyle \prod_{i=1}^{r} \ch_{k_i}(\omega \gamma_i)} \right)$.

Let $S$ be now a general elliptic surface with section.
The degeneration formula \eqref{deg formula} together with \eqref{initial term 1} reads
\[
Z^{(S \times \p^1, S_\infty)}_{\PT, (df,n)}\left( \lambda | {\textstyle \prod_{i=1}^{r} \ch_{k_i}(\omega \gamma_i)} \right) 
=
Z^{(S \times \p^1, S_\infty \cup E \times \p^1)}_{\PT, (df,n)}\left( \lambda | \varnothing | {\textstyle \prod_{i=1}^{r} \ch_{k_i}(\omega \gamma_i)} \right) + ( \ldots ) \] 
where $(\ldots)$ stands for invariants of 
$(S \times \p^1, S_\infty \cup E \times \p^1)$ in curve classes $(d'f, n') < (d,f)$ (with respect to a lexicographic order) times invariants
of $(Y \times \p^1, Y_{\infty} \cup E \times \p^1)$. 
We hence see that the invariants of
$(S \times \p^1, S_{\infty})$ and 
$(S \times \p^1, S_\infty \cup E \times \p^1)$
are related by an invertible upper triangular relation.
Since the GW/PT correspondence is compatible with the degeneration formula and is known for the coefficients in this relation by the first step,
we see that the GW/PT for one geometry implies the claim for the other geometry.
This completes the proof of the first claim.

The second claim follows likewise by the degeneration formula for the degeneration of $S$ to the normal cone of the fibers $E_{1}, \ldots, E_{\ell}$.
\end{proof}

\begin{prop} \label{prop:GWPT for elliptic surface over P1}
	Theorem~\ref{thm:GWPT correspondence} holds for an elliptic surface $S$ which is fibered over $\p^1$.
\end{prop}
\begin{proof}
Let $S \to \p^1$ be an elliptic surface of degree $d>1$ (the cases $d \in \{ 0,1 \}$ are covered by Proposition~\ref{prop:case d=0 and 1}).
Let $V \subset H^2(S,\BZ)$ be the orthogonal complement to the vectors $[\Sigma], f$. 
The monodromy group of $S$ contains the subgroup $O^+(V)$ of orthogonal transformations of $V$ of real spinor norm $1$.
The group $O^{+}(V)$ is Zariski dense in the complex orthogonal group $O(V_{\BC})$.
By the same proof as in \cite[Prop.9.3]{Marked}
it hence suffices to prove the GW/PT correspondence for the invariants
\begin{equation} \label{fsdfds} Z^{(S \times \p^1, S_{\infty})}_{\PT, (df,n)}\Big( ({\textstyle \prod_{i=1}^{r} \Fq_{\ell_i} })(\delta) \Big| ({\textstyle \prod_{i=1}^{s} \ch_{k_i}})(\Gamma) \Big) \end{equation}
where the cohomology insertions are
\begin{equation} \label{delta Gamma}
\begin{gathered}
\delta = \pr_{12}^{\ast}(\Delta_{S}) \cdots \pr_{2a-1,2a}^{\ast}(\Delta_S) \prod_{i=2a+1}^{r} \pr_i^{\ast}(\delta_i)  \\
\Gamma = \pr_{12}^{\ast}(\Delta_{S}) \cdots \pr_{2b-1,2b}^{\ast}(\Delta_S) \prod_{i=2b+1}^{s} \pr_i^{\ast}(\gamma_i) \prod_{i=1}^{s} \pr_i^{\ast}(\omega)
\end{gathered}
\end{equation}
with $\delta_i, \gamma_i \in \{ 1, \pt, [\Sigma], f \}$,
$\Delta_S \in H^{\ast}(S \times S)$ the class of the diagonal,
and $\pr_i, \pr_{ij}$ the projetions.

Consider a degeneration of $S$ to a union
\[
S \rightsquigarrow S' \cup_E R.
\]
where $S' \to \p^1$ is an elliptic surface of degree $d-1$ and $R$ is a rational elliptic suface, and the surfaces are glued along a joint smooth elliptic fiber $E$.
The vanishing cohomology of this degeneration is contained in $V$.
Hence by applying the degeneration formula to \eqref{fsdfds} and arguing as in \cite[9.5]{Marked}
then reduces the claimed correspondence to the correspondence for 
$(R \times \p^1, R_{\infty} \cup (E \times \p^1))$ and
$(S' \times \p^1, S_{\infty} \cup (E \times \p^1))$.
By an induction on $d$ Theorem~\ref{thm:GWPT correspondence} is known for $S'$, and by Proposition~\ref{prop:case d=0 and 1} for $R$, so by 
Proposition~\ref{prop:absolte relative}
it holds for both of these relative geometries.
\end{proof}

\begin{prop} \label{prop:GWPT trivial fibration}
Theorem~\ref{thm:GWPT correspondence} holds for
the trivial elliptic fibration $\Sigma \times E \to \Sigma$.
\end{prop}
\begin{proof}
Let $g=g(\Sigma)$.
Consider a degeneration of $\Sigma$ to a curve $\Sigma_0$ which is isomorphic to $\p^1$ with $2g$ distinct points $x_1,\ldots, x_{2g}$ pairwise glued.
Let $S \rightsquigarrow S_0$ be the induced degeneration obtained by taking the product with $E$.
The degeneration has vanishing cohomology $H^1(\Sigma) \otimes H^{\ast}(E)$. The monodromy group of $\Sigma$ acts as the full symplectic group on $H^1(\Sigma,\BZ)=\BZ^{2g}$.
Hence arguing as in the proof of Proposition~\ref{prop:GWPT for elliptic surface over P1} (see \cite{ABPZ} for the use of the symplectic instead of the orthogonal group) it suffices to prove the GW/PT correspondence 
where (parallel to \eqref{delta Gamma}) the insertions which are vanishing can be grouped into pairs given by a product of $\Delta_{\Sigma}$ times classes from $H^{\ast}(E \times E)$.
These can then be degenerated by the methods of \cite{ABPZ,Marked}.
After resolving $\Sigma_0$, we are reduced to proving Theorem~\ref{thm:GWPT correspondence} for 
$(E \times \p^1, E_{x_1, \ldots, x_{2g}})$.
This follows by Proposition~\ref{prop:absolte relative}
and Proposition~\ref{prop:case d=0 and 1}.
\end{proof}

\begin{proof}[Proof of Theorem~\ref{thm:GWPT correspondence}]
We know the statement whenever $g(\Sigma)=0$ by
Proposition~\ref{prop:GWPT for elliptic surface over P1}. Hence assume $g(\Sigma)>0$.

If $d>0$ we use Proposition~\ref{prop:Seiler} and the degeneration of Example~\ref{example: higher genus higher degree degeneration}
which has no vanishing cohomology.
By the compatibility of the statement with the degeneration formula,
the claim then follows from Proposition~\ref{prop:GWPT trivial fibration}
and the $g(\Sigma)=0$ case (use Proposition~\ref{prop:absolte relative} again).

If $d=0$ but we are non-constant, we use Example~\ref{example:torsion L}
and the earlier cases to reduce to $d>0$.
\end{proof}

\section{Elliptic surfaces: Computations}\label{sec:elliptic surfaces2}
The goal of this section is to compute the $2$-point operator $\QHilb$ for elliptic surfaces (Theorem~\ref{thm:2point operator}). We start with computations in GW theory, then move them to PT theory by the correspondence to obtain a formula for the $2$-point operator in PT theory $\QPT$ (Theorem~\ref{thm:QPT}),
and finally use Nesterov's wall-crossing in Corollary~\ref{cor:Q Hilb-PT comparision} to deduce the Hilbert scheme operator.

\subsection{Fiber evaluations}
We start with some basic evaluations.
Consider the partition functions
of disconnected Gromov-Witten invariants of $(S \times \CC_{h,N}, S_z)$,
\begin{multline*}
Z^{(S \times \CC_{h,N}, S_z)}_{\GW, (\beta,n)}\left( \lambda_1, \ldots, \lambda_N \middle| \tau_{k_1}(\gamma_1) \cdots \tau_{k_r}(\gamma_r) \right) 
= 
(-1)^{(1-h-N)n + \sum_i \ell(\lambda_i)}
z^{(2-2h-N)n + \sum_i \ell(\lambda_i)} \\
\cdot \sum_{g \in \BZ} (-1)^{g-1} z^{2g-2}
\left\langle \, \lambda_1, \ldots, \lambda_N \, \middle| \, \tau_{k_1}(\gamma_1) \cdots \tau_{k_r}(\gamma_r) \right\rangle^{\GW, (S \times \CC_{g,N}, S_z), \bullet}_{g, (\beta,n)}.
\end{multline*}

As before, let us write $\varnothing$ for the empty partition (with no parts).

\begin{lemma} \label{lemma:some eval} We have:
\begin{enumerate}
\item[(a)]
$\left\langle \, \varnothing, \varnothing, \varnothing \right\rangle^{\GW, (S \times \p^1, S_{0,1,\infty})}_{g, (df,0)}
=
- \delta_{g1} d_{\Sigma} \sigma(d)/d.$
\item[(b)] 
$\left\langle \, \varnothing, \varnothing \right\rangle^{\GW, (S \times \p^1, S_{0,\infty})}_{g, (df,0)}
=
0$
\end{enumerate}
\end{lemma}
\begin{proof}
(a) First we consider the case of genus $g=0$. If $d>0$, the invariants vanish by the product formula for relative Gromov-Witten invariants \cite{LQ}
and the results of Section~\ref{subsec:GW theory for elliptic surfaces}. The case $g=d=0$ is excluded by stability. 

In positive genus, 
we can compute the absolute invariants of $S \times \p^1$
by applying the localization formula and Proposition~\ref{prop:lambda g-1 invariants}: 
\[ \left\langle 1 \right\rangle^{\GW, S \times \p^1}_{g, (df,0)}
=
2 \int_{[ \Mbar_{g}(E,df) ]^{\vir}} (-1)^{g-1} \lambda_{g-1}
=
2 \delta_{g1} d_{\Sigma} \frac{\sigma(d)}{d}.
\]
The degeneration formula for degenerating $\p^1$ to $\p^1 \cup_{x} \p^1$ gives \cite{Li2}:
\[
\left\langle 1 \right\rangle^{\GW, S \times \p^1}_{g, (df,0)}
=
\left\langle  \varnothing | 1 \right\rangle^{\GW, (S \times \p^1, S_{0})}_{g, (df,0)}
+ \left\langle  \varnothing | 1 \right\rangle^{\GW, (S \times \p^1, S_{\infty})}_{g, (df,0)}.
\]
We find that
\[ \left\langle \varnothing | 1 \right\rangle^{\GW, (S \times \p^1, S_{0})}_{g, (df,0)}
= 
\delta_{g1} d_{\Sigma} \frac{\sigma(d)}{d}.
\]
Hence by the degeneration formula again (now degenerating off $3$ rational tails):
\[
\left\langle 1 \right\rangle^{\GW, S \times \p^1}_{g, (df,0)}
=
\left\langle \varnothing, \varnothing, \varnothing | 1 \right\rangle^{\GW, (S \times \p^1, S_{0,1,\infty})}_{g, (df,0)}
+ 3 \left\langle \varnothing | 1 \right\rangle^{\GW, (S \times \p^1, S_{0})}_{g, (df,0)}.
\]
Hence
\[
\left\langle \varnothing, \varnothing, \varnothing | 1 \right\rangle^{\GW, (S \times \p^1, S_{0,1,\infty})}_{g, (df,0)}
=
- \delta_{g1} d_{\Sigma} \frac{\sigma(d)}{d}.
\]
The case (b) is similar.
\end{proof}

Let
$\1, \pt_n \in H^{\ast}(S^{[n]})$
denote the unit and the class of a point on the Hilbert scheme respectively,
viewed here as the weighted partitions 
$\frac{1}{n!} (1,1)^n$ and $(1,\pt)^n$ respectively.
\begin{prop} \label{prop:p11 evaluation}
$\sum_{d \geq 0} Z^{(S \times \p^1, S_{0,1,\infty})}_{\GW, (df,n)}\left( \pt_n, 1, 1 \right) q^d
\ =\  \prod_{d \geq 1} (1-q^d)^{d_{\Sigma}}.$
\end{prop}
\begin{proof}
We have to evaluate
\[ \sum_{d \geq 0}
Z^{(S \times \p^1, S_{0,1,\infty})}_{\GW, (df,n)}\left( \pt_n, 1, 1 \right) q^d
=
\sum_{d \geq 0} \sum_{g \in \BZ} (-1)^{g-1+n} z^{2g-2+2n} q^d
\left\langle \, \pt_n, 1, 1 \right\rangle^{\GW, (S \times \p^1, S_{0,1,\infty}), \bullet}_{g, (df,n)}. \]

Assume $n \geq 1$ first.
If $g>0$ or $d>0$,
by Proposition~\ref{prop:GW class elliptic surface} and Proposition~\ref{prop:elliptic surface genus 0} we have the vanishing of the following {\em connected} Gromov-Witten invariants:
\[ \left\langle \, \pt_n, 1, 1 \right\rangle^{\GW, (S \times \p^1, S_{0,1,\infty})}_{g, (df,n)} = 0.\]
In case $g=d=0$, the invariant vanishes unless $n=1$,
where the moduli space parametrizes a tube and hence is isomorphic to $S$. Then one finds
\[ \left\langle \, \pt_n, 1, 1 \right\rangle^{\GW, (S \times \p^1, S_{0,1,\infty})}_{0, (0,1)} = 1. \]

If $n=0$ by Lemma~\ref{lemma:some eval} we have
\[ \left\langle \, \varnothing, \varnothing, \varnothing \right\rangle^{\GW, (S \times \p^1, S_{0,1,\infty})}_{g, (df,0)}
=
- \delta_{g1} d_{\Sigma} \frac{\sigma(d)}{d}.
\]

The only maps that contribute to the disconnected invariant are hence given by $n$ tubes, and genus $1$ vertical components. The overall genus is $1-n$.
The contribution from the tubes is $1$. So:
\[
\sum_{d \geq 0}
Z^{(S \times \p^1, S_{0,1,\infty})}_{\GW, (df,n)}\left( \pt_n, 1, 1 \right) q^d
=
\exp\left( - d_{\Sigma} \sum_{d \geq 1} \frac{\sigma(d)}{d} q^d \right)
=
\prod_{m \geq 1} (1-q^m)^{d_{\Sigma}}. \qedhere
\]
\end{proof}

\subsection{Gromov-Witten evaluations}
Consider $H^{\ast}(S)$-weighted partitions of $n$,
\begin{gather*}
\lambda = \big( (\lambda_1, \delta_1) , \ldots, (\lambda_{\ell(\lambda)}, \delta_{\ell(\lambda)} ) \big) \\
\mu = \big( (\mu_1, \epsilon_1) , \ldots, (\mu_{\ell(\mu)}, \epsilon_{\ell(\mu)} ) \big).
\end{gather*}
We assume that $\deg(\lambda) + \deg(\mu) = 2n-1$.
Define also the tuples
\begin{gather*}
a=(a_1,\ldots, a_r) := (\lambda_1,\ldots, \lambda_{\ell(\lambda)}, -\mu_1, \ldots, -\mu_{\ell(\mu)}) \\
\gamma=(\gamma_1,\ldots, \gamma_r) := (\delta_1,\ldots, \delta_{\ell(\lambda)}, \epsilon_1, \ldots, \epsilon_{\ell(\mu)}.
\end{gather*}
The four cases in the following proposition cover (using linearity) {\it all} possible weighted partitions $\lambda,\mu$ with $\deg(\lambda) + \deg(\mu) = 2n-1$.

\begin{prop} \label{prop:DR eval connected}
Let $E \subset S$ be a fixed smooth fiber. The following holds:
\begin{enumerate}
\item[(a)]
If $\gamma_i|_{E} = 0$ for some $i$, then
$Z^{(S \times \p^1, S_{0,\infty}), \textup{connected}}_{\GW, (df,n)}\left( \lambda, \mu | \tau_{0}(\omega \Sigma) \right) = 0$
 for all $d>0$.
\item[(b)] If $\gamma_i|_{E} = \gamma_j|_{E} = \alpha$ or $\gamma_i|_{E} = \gamma_j|_{E} = \beta$ for some $i\ne j$ and $\gamma_{\ell}=[\Sigma]$ otherwise, then $Z^{(S \times \p^1, S_{0,\infty}), \textup{connected}}_{\GW, (df,n)}\left( \lambda, \mu | \tau_{0}(\omega \Sigma) \right) = 0$
 for all $d$.
\item[(c)]
If $\gamma_i=1$ and $\gamma_j=[\Sigma]$ for $j \neq i$, then
\begin{multline*}
\sum_{d \geq 1} Z^{(S \times \p^1, S_{0,\infty}), \textup{connected}}_{\GW, (df,n)}\left( \lambda, \mu | \tau_{0}(\omega \Sigma) \right) q^d \\
=
(-1)^n d_{\Sigma}
\frac{a_i^2}{a_1 a_2 \ldots a_{r}} \sum_{S \subset \{1, ... , r\}} (-1)^{|S|}
\sum_{g \geq 1} \frac{(a_S z)^{2g-2+r}}{(2g - 2 + r)!} \left( q \frac{d}{dq} \right)^{r-1} G_{2g}(q).
\end{multline*}
\item[(d)] If $\gamma_i|_{E} = \alpha$, $\gamma_j|_{E} = \beta$ with $i<j$ and $\gamma_{\ell}=[\Sigma]$ otherwise, then
\begin{multline*}
\sum_{d \geq 1} Z^{(S \times \p^1, S_{0,\infty}), \textup{connected}}_{\GW, (df,n)}\left( \lambda, \mu | \tau_{0}(\omega \Sigma) \right) q^d \\
= 
(-1)^n d_{\Sigma}
\frac{-a_i a_j}{a_1 a_2 \ldots a_{r}} \sum_{S \subset \{1, ... , r\}} (-1)^{|S|}
\sum_{g \geq 1} \frac{(a_S z)^{2g-2+r}}{(2g - 2 + r)!} \left( q \frac{d}{dq} \right)^{r-1} G_{2g}(q).
\end{multline*}
\end{enumerate}
\end{prop}
Here, the superscript \emph{connected} stands for the partition function of connected invariants.

\begin{proof}
We have
\[ Z^{(S \times \p^1, S_{0,\infty}), \textup{connected}}_{\GW, (df,n)}\left( \lambda, \mu | \tau_{0}(\omega \Sigma) \right) 
=
\sum_{g} (-1)^{-n+r+g-1} z^{2g-2+r} \langle \lambda, \mu | \tau_0(\omega \Sigma) \rangle^{(S \times \p^1, S_{0, \infty}), \GW}_{g, (df, n)}.
\]
Since we assumed $d \geq 1$, the terms of genus zero vanish. Hence assume $g \geq 1$. 
By derigidification (Proposition~\ref{prop:rigidification}), the divisor equation, the product formula \cite{LQ}, and Proposition~\ref{prop:GW class elliptic surface} we get
\begin{align*}
& \langle \lambda, \mu | \tau_0(\omega \Sigma) \rangle^{(S \times \p^1, S_{0, \infty}), \GW}_{g, (df, n)} \\
& = 
\langle \lambda, \mu | \tau_0(\Sigma) \rangle^{(S \times \CC_{0,2}, S_{0, \infty}), \GW}_{g, (df, n)} \\
& = 
d \langle \lambda, \mu \rangle^{(S \times \CC_{0,2}, S_{0, \infty}), \GW}_{g, (df, n)} \\
& = 
d \int_{[ \Mbar_{g,r}(S,df)]^{\vir}} \DR_{g}(a_1,\ldots, a_r) \prod_{i=1}^{r} \ev_i^{\ast}(\gamma_i) \\
& = 
d \cdot d_{\Sigma} \int_{[ \Mbar_{g,r}(E,d)]^{\vir}} (-1)^{g-1} \lambda_{g-1} \DR_{g}(a_1,\ldots, a_r) \prod_{i=1}^{r} \ev_i^{\ast}(\gamma_i|_{E}).
\end{align*}
Part (a) follows. Part (b) also follows because $\CC_g(\alpha^{\times 2},\pt^{\times n}) = 0$ by a basic monodromy relation. Parts (c) and (d) follow from the integrals computed in Theorem~\ref{thm:evaluation}.
\end{proof}

The disconnected series is expressed in terms of the connected one as follows:
\begin{prop} \label{prop:disconnected GW eval}
For any $\gamma \in H^{\ast}(S)$ and $k \geq 0$ we have
\[ Z^{(S \times \p^1, S_{0,\infty})}_{\GW, (df,n)}\left( \lambda, \mu | \tau_{k}(\omega \gamma) \right) =
\sum_{ \substack{ \lambda = \lambda' \cup \rho \\ \mu = \mu' \cup \rho' }}
\left(\int_{S^{[n-n']}} \rho \cup \rho' \right)
Z^{(S \times \p^1, S_{0,\infty}), \textup{connected}}_{\GW, (df,n')}\left( \lambda', \mu' | \tau_{k}(\omega \gamma) \right)
 \]
where $\rho, \rho'$ run over subpartitions with $\ell(\rho)=\ell(\rho')$ and $|\rho|=|\rho'|$,
and $n'=n-|\rho|$.
\end{prop}
\begin{proof}
Consider a stable map to $(S \times \p^1, S_{0,\infty})$
that lies on a component of the moduli space that contributes non-trivially to the invariant.
There is one component of the domain $C$ which carries the marking $\omega \gamma$.
By \cite[Lemma 1]{HAE} all other components of $C$ must parametrize tubes, that is maps
$f : \p^1 \to S \times \p^1$ which are totally ramified over the boundary $S_{0,\infty}$,
of some degree $m$ over $\p^1$, and of degree $0$ over $S$. The contribution of the tube is $1/m$, see \cite{HAE}.
We let $\rho, \rho'$ denote the weighted subpartitions corresponding to the markings  on the tube components at $S_{0}$ and $S_{\infty}$ respectively. Then the contributing factor from the tubes is precisely
\[
(-1)^{\ell(\rho) + |\rho|} \int_{S^{[|\rho|]}} \rho \cup \rho'.
\]
One checks that
$(-1)^{\ell(\rho) + |\rho|}$ is also the sign that 
appears in the difference between the connected and disconnected partition function. Hence the claim follows.
\end{proof}


\subsection{Pandharipande-Thomas evaluations}
Our next goal is to move the above evaluations to the PT side using the correspondence.
For that we rewrite the previous evaluation in the variable $p=e^{z}$.
For $r \geq 1$ recall from Section~\ref{sec:Modular and Jacobi forms} the series
\[
\A_r(p,q) = \frac{B_r}{r} + \delta_{r,1}\frac{1}{2} \frac{p+1}{p - 1} - \sum_{k,\ell \geq 1} \ell^{r-1} (p^{k} + (-1)^r p^{-k}) q^{k\ell} .
\]
Under the variable change $p=e^{z}$ we have the expansion:
\begin{align*}
\A_r(z) = 
\frac{1}{z} \delta_{r,1} & -2 \sum_{\substack{0 \leq m < r-1 \\ m \equiv r (\text{mod }2)}}
\frac{z^m}{m!} \left(q \frac{d}{dq} \right)^m G_{r-m} 
 -2 \sum_{g \geq 1} \frac{z^{2g-2+r}}{(2g-2+r)!} \left(q \frac{d}{dq} \right)^{r-1} G_{2g}
\end{align*}

\begin{lemma}
For any tuple $a=(a_1, \ldots, a_n)$ and $k < n$ we have
\[ \sum_{S \subset \{ 1, \ldots, n \}} (-1)^{|S|} a_S^{k} = 0, \]
where $a_S = \sum_{i\in S} a_i$.
\end{lemma}
\begin{proof}
The expression is a degree $k$ polynomial in the $a_i$. If a single $a_i$ is set to $0$, the terms cancel in pairs $S = T, T\cup\{i\}$, so the polynomial must be divisible by each $a_i$. But then it is divisible by $a_1a_2\cdots a_n$ and $n > k$, so the polynomial must be the zero polynomial.
\end{proof}

With these formulas the nonzero parts of Proposition~\ref{prop:DR eval connected} become:

\begin{prop}\label{prop:abcde}
If $\gamma_i=1$ and $\gamma_j=[\Sigma]$ for all $j \neq i$, then modulo the constant term (in $q$):
\[ \sum_{d \geq 1} Z^{(S \times \p^1, S_{0,\infty}), \textup{connected}}_{\GW, (df,n)}\left( \lambda, \mu | \tau_{0}(\omega \Sigma) \right) q^d \\
=
\frac{1}{2}(-1)^{n-1} d_{\Sigma}
\frac{a_i^2}{a_1 a_2 \ldots a_{r}} \sum_{S \subset \{1, ... , r\}} (-1)^{|S|}
A_r(a_S z)
\]
Similarly,
if $\gamma_i|_{E} = \alpha$, $\gamma_j|_{E} = \beta$ with $i<j$ and $\gamma_{\ell}=[\Sigma]$ otherwise, then modulo the constant term (in $q$):
\[
\sum_{d \geq 1} Z^{(S \times \p^1, S_{0,\infty}), \textup{connected}}_{\GW, (df,n)}\left( \lambda, \mu | \tau_{0}(\omega \Sigma) \right) q^d \\
=
\frac{1}{2}(-1)^{n-1} d_{\Sigma}
\frac{-a_i a_j}{a_1 a_2 \ldots a_{r}} \sum_{S \subset \{1, ... , r\}} (-1)^{|S|}
A_r(a_S z)
\]
\end{prop}

We want to apply now the GW/PT correspondence for the disconnected series in Proposition~\ref{prop:disconnected GW eval}.
To formulate the invariants on the PT  side concisely
we use Nakajima operators. 

For a class $\gamma \in H^2(S)$ and integers $b_1, \ldots, b_r \in \BZ$ recall the class
$\star^{b_1, \ldots, b_r}(\gamma) \in H^{\ast}(S^r)$ from \eqref{star class}.

Define an operator on Fock space by
\[ T =
-\frac{1}{2}\sum_{\substack{r \geq 2 \\ b_1, \ldots, b_r \in \BZ_{\neq 0} \\ \sum_i b_i = 0}}
\frac{1}{r!}
\frac{(-1)^{r} }{b_1 \cdots b_r}
\left( \sum_{S \subset \{ 1, \ldots, r \}} (-1)^{|S|} A_r(p^{b_S}) \right)
: \Fq_{b_1} \cdots \Fq_{b_r}( \star^{b_1,\ldots, b_d}(f) ) : \]
where $b_S = \sum_{i \in S} b_i$.


\begin{prop} \label{prop:rigidified eval}
For any $\lambda, \mu \in H^{\ast}(S^{[n]})$ we have:
\[
\sum_{d \geq 0} Z^{(S \times \p^1, S_{0,\infty})}_{\PT, (df,n)}\left( \lambda, \mu | \ch_2( \omega \Sigma ) \right) q^d
=
\int_{S^{[n]}} (D(\Sigma) \cdot \lambda + d_{\Sigma} T\lambda) \cdot \mu
\]
\end{prop}
\begin{proof}
We first write out the operator $T$ more concretely as
\begin{multline*}
T = 
-\frac{1}{2} \sum_{\substack{r \geq 2 \\ b_1, \ldots, b_r \in \BZ_{\neq 0} \\ \sum_i b_i = 0}}
\frac{1}{(r-1)!}
\frac{b_1^2 (-1)^{r} }{b_1 \cdots b_r}
\left( 
\sum_{S \subset \{ 1, \ldots, r \}} (-1)^{|S|} A_r(p^{b_S}) \right) : \Fq_{b_1}(\pt) \Fq_{b_2}(f) \cdots \Fq_{b_r}(f) : \\
-\frac{1}{2} \sum_{\substack{r \geq 2 \\ b_1, \ldots, b_r \in \BZ_{\neq 0} \\ \sum_i b_i = 0}}
\frac{1}{2 (r-2)!}
\frac{-b_1 b_2 (-1)^{r} }{b_1 \cdots b_r}
\left( \sum_{S \subset \{ 1, \ldots, r \}} (-1)^{|S|} A_r(p^{b_S}) \right)
: \Fq_{b_1}\Fq_{b_2}(\Delta^{\text{odd}}_{\ast}(f)) \Fq_{b_3}(f) \cdots \Fq_{b_r}(f) :
\end{multline*}

We then consider the case $d>0$ of the equality that we want to prove.
By the GW/PT correspondence (Theorem~\ref{thm:GWPT correspondence})
we can consider GW invariants on the left hand side.
Let us first assume that all the cohomology weights of $\lambda$ and $\mu$ are even
and given by classes $1, \pt, [\Sigma]$ or $\gamma \in H^2(S)$ with $\gamma|_{E} = 0$.
Then by Propositions~\ref{prop:disconnected GW eval} and 
 \ref{prop:DR eval connected}(a)
only the cohomology weights $\alpha \in \{ 1, [\Sigma] \}$ 
can appear in the connected partition function in Proposition~\ref{prop:disconnected GW eval}.
Similarly, by the Nakajima commutation relation only the Nakajima operators $\Fq_k(\alpha)$ with $\alpha \in \{ 1, [\Sigma] \}$ 
interact with the operator $T$, that is have a non-trivial commutator with $T$.
The Gromov-Witten partition function is symmetric in $\lambda,\mu$. Similarly, we have $T^{\ast} = T$ by a direct check.
Taking also into account the dimension/degree constraint,
it hence suffices to consider the case where
\[ \lambda = (\lambda_1, 1) (\lambda_2, [\Sigma]) \cdots (\lambda_{\ell},[\Sigma]), \quad 
\mu=(\mu_1, [\Sigma]) \cdots (\mu_{\ell'},[\Sigma]), \]
Then using the fact that $\lambda$ corresponds to the cohomology class $\frac{1}{\prod_{i} \lambda_i} \Fq_{\lambda_1}(1) \prod_{i \geq 2} \Fq_{\lambda_i}([\Sigma]) \vacuum$, 
the first part of Proposition~\ref{prop:abcde}
implies the claimed equality by a straightforward evaluation.
The odd part is similar (although more tedious).

Now consider the case $d=0$. Here the PT invariants on the left hand side vanish in case $\chi>n$ by derigidification and the divisor equation.
For $\chi=n$ the moduli space of stable pairs is isomorphic to $S^{[n]}$.
Under this isomorphism the class $\tau_0(\omega \Sigma)$ corresponds to $D(\Sigma)$,
and we find
\[ 
Z^{(S \times \p^1, S_{0,\infty})}_{\PT, (0,n)}\left( \lambda, \mu | \ch_2( \omega \Sigma ) \right)
=
\langle \lambda, \mu | \tau_0(\omega \Sigma) \rangle^{\PT, (S \times \p^1, S_{0,\infty})}_{\chi=n, (0,n)} =
\int_{S^{[n]}} \lambda \cdot \mu \cdot D([\Sigma]). \]
\end{proof}

We obtain a similar expression for the insertion $\ch_3(\omega)$.
Define the operator
\[
\widehat{T} =
-\frac{1}{2}\sum_{\substack{r \geq 2 \\ b_1, \ldots, b_r \in \BZ_{\neq 0} \\ \sum_i b_i = 0}}
\frac{1}{r!}
\frac{(-1)^{r} }{b_1 \cdots b_r}
\left( \sum_{S \subset \{ 1, \ldots, r \}} (-1)^{|S|} b_S A_{r-1}(p^{b_S}) \right)
: \Fq_{b_1} \cdots \Fq_{b_r}( \star^{b_1,\ldots, b_d}(f) ) :~, \]
where for $r=2$ the sum runs over all subsets $S$ of $\{ 1, 2 \}$ such that $S \neq \varnothing, \{ 1, 2 \}$.

\begin{prop} \label{prop:rigidified eval2}
For any $\lambda, \mu \in H^{\ast}(S^{[n]})$ we have:
\[
\sum_{d \geq 0} Z^{(S \times \p^1, S_{0,\infty})}_{\PT, (df,n)}\left( \lambda, \mu | \ch_3( \omega  ) \right) q^d
=
\int_{S^{[n]}} \left(\delta - \frac{1}{2} D(c_1(S))\right) \cdot \lambda \cdot \mu + \int_{S^{[n]}} d_{\Sigma} \widehat{T}(\lambda) \cdot \mu
\]
\end{prop}
\begin{proof}
By the GW/PT correspondence (in particular, using the explicit form of the correspondence matrix given in \cite{MOOP})
we have
\[
Z^{(S \times \p^1, S_{0,\infty})}_{\PT, (df,n)}\left( \lambda, \mu | \ch_3( \omega  ) \right)
=
\frac{1}{z}
Z^{(S \times \p^1, S_{0,\infty})}_{\GW, (df,n)}\left( \lambda, \mu | \tau_1( \omega  ) \right)
\]
The claim now follows from Theorem~\ref{thm:evaluation} parallel to the proofs of Propositions~\ref{prop:DR eval connected} and \ref{prop:rigidified eval}, using the dilaton equation in place of the divisor equation. In particular, for the constant term one works on the PT side and notes that for class $(0,n)$ and $\chi=n$ we have
\[ \ch_3(\omega) = \pi_{\ast}( \ch_3(\CO_Z) ) = c_1(\CO_S^{[n]}) - \frac{1}{2} D(c_1(S)). \qedhere \]
\end{proof}

Define the $2$-point operator
$\QPT \in \mathrm{End}\big( \Fock_S \otimes \BQ((p))[[q]] )\big)$
by setting for all $\lambda, \mu \in H^{\ast}(S^{[n]})$
\[
( \QPT(\lambda), \mu)
=
\sum_{d \geq 0} Z_{\PT, (df,n)}^{(S \times \CC_{0,2},S_{0,\infty})}(\lambda, \mu) q^d,
\]
(the unstable term $(d,\chi) = (0,n)$ is excluded in the sum, the corresponding coefficient is zero).

Recall the operator $\omega_{\gamma}(p)$ from \eqref{omega gamma}.

\begin{thm} \label{thm:QPT}
\[
\QPT = -\sum_{k>0} \ln(1-p^k) \Fq_{k}\Fq_{-k}(\Delta_{\ast} c_1(S))
+ d_{\Sigma} \sum_{m, d \geq 1} \omega_{df}(p^m) \frac{q^{md}}{md}.
\]
\end{thm}
\begin{proof}
By rigidifying the $2$-point invariants by either $\ch_2(\Sigma)$ or $\ch_3(1)$ and then using Propositions~\ref{prop:rigidified eval} and~\ref{prop:rigidified eval2} respectively, we find
\begin{gather}
q \frac{d}{dq} \QPT = d_{\Sigma} T \notag \\
p \frac{d}{dp} \QPT = d_{\Sigma} \widehat{T} \label{That and QPT}.
\end{gather}
The claim then follows from the formulas for $T$ and $\widehat{T}$ and Lemma~\ref{lemma:averaging} below.
\end{proof}

\begin{lemma} \label{lemma:averaging}
For $r \geq 2$ let $a_1, \ldots, a_r \in \BZ$ such that $\sum_i a_i = 0$. Then
\[
\sum_{S \subset \{ 1, \ldots, r \}} (-1)^{|S|} A_r(a_S z)
=
-2 \sum_{k,d \geq 1} d^{r-1} q^{kd} (1-p^{a_1k}) \cdots (1-p^{a_rk})
\]
\end{lemma}
\begin{proof}
This follows from observing that
\[
\sum_{S \subset \{ 1, \ldots, r \}} (-1)^{|S|}p^{a_Sk} = (1-p^{a_1k})\cdots(1-p^{a_rk}) = (-1)^r(1-p^{-a_1k})\cdots(1-p^{-a_rk}).
\]
\end{proof}

\subsection{Hilbert scheme evaluations}
In Theorem~\ref{thm:QPT} we have obtained a formula for the
$2$-point operator in PT theory.
Here we conclude from this the $I$-function of Nesterov's Hilb/PT wall-crossing (or at least the part that is relevant for us) and the $2$-point operator on the Hilb side.

We first consider the $I$-function,
where we have complete results except for the bielliptic surfaces.
\begin{prop}(The $I$-function) \label{prop:I function}
Let $S$ be an elliptic surface which is not a bielliptic surface.
Under the variable change $p=-y$ and $q=t^f$
we have
\[
I_0(t,y) = \prod_{n \geq 1} (1-q^n)^{-d_{\Sigma}}.
\]
\[
\left[ I_1(t,y) \right]_{\deg=2}
=
D(c_1(S)) \cdot \log\left( (1-p) \prod_{r \geq 1} \frac{ (1-p q^r) (1-p^{-1} q^r)}{(1-q^r)^2} \right)
\prod_{n \geq 1} (1-q^n)^{-d_{\Sigma}}
\]
The formulas hold also for the bielliptic surface modulo $t^{[\Sigma]}$.
\end{prop}
\begin{proof}
We first prove that $I_0(t,y)$
and $\left[ I_1(t,y) \right]_{\deg=2}$
only depend on $z$ and $t^f$, or equivalently, that there are no contributions from curve classes $\beta$ other than $df$.

Indeed, if $p_g(S)=0$ and $S$ is not a bielliptic surface, then $S$ is either $\p^1 \times E$ or a rational elliptic surface.
Hence we are in the semipositive case,
in which, according to \eqref{I_0 I_1}, only the vertical curve classes contribute. If $\beta$ is vertical, but not a multiple of $f$, then by Corollary~\ref{cor:integral fibers} we can deform to a situation where $\beta$ is no longer effective,
so the contribution to the $I$-function vanishes by deformation invariance.
If $p_g(S)>0$, we can take a non-zero holomorphic $2$-form $\theta \in H^0(S, \omega_S)$ and construct a cosection on the moduli space of stable pairs as in \cite{MPT} or \cite{Nesterov2}. The associated degeneracy locus will be empty for $\beta$ not a multiple of $f$, so the $I$-function terms will vanish.

Hence we only need to consider the case $\beta = df$.
We apply Corollary~\ref{cor:I function determination}.
This corollary works also in the non-semipositive case, since here the only contributing curve classes satisfies $c_1(S) \cdot \beta \geq 0$.
By Proposition~\ref{prop:p11 evaluation}
and Theorem~\ref{thm:QPT}
the invariants on the right hand side of Corollary~\ref{cor:I function determination}
are determined, so the claim follows by a direct computation.
\end{proof}

We state the Hilb/PT wall-crossing formula,
first for any curve class, then for a multiple of $f$.
\begin{thm}
Let $S$ be an elliptic surface which is not a bielliptic surface.
Let $q=t^f$.
\begin{enumerate}
\item[(a)] For $2h-2+N>0$ we have
\begin{multline} \label{wallcross elliptic surface}
\prod_{r \geq 1}(1-q^r)^{-d_{\Sigma} (2h-2+N)}
\sum_{\beta \cdot f = r} \sum_{k \in \BZ}
{'\langle \lambda_1, \ldots, \lambda_N \rangle^{\PT}_{h,(\beta,k)}} t^{\beta} (-p)^k \\
=
\left( (1-p) \prod_{r \geq 1} \frac{ (1-p q^r) (1-p^{-1} q^r)}{(1-q^r)^2} \right)^{r d_{\Sigma}}
\sum_{\beta \cdot f = r} \langle \lambda_1, \ldots, \lambda_{N} \rangle^{S^{[n]}}_{h, \beta + kA} t^{\beta} (-p)^{k}.
\end{multline}
\item[(b)] For $h=0$ and $N=2$ and $\deg_{\BC}(\lambda_1) + \deg_{\BC}(\lambda_2) = 2n-1$ we have:
\begin{multline*}
\sum_{\beta \cdot f=r} \sum_{k \in \BZ}
{'\langle \lambda_1, \lambda_2 \rangle^{\PT}_{0,(\beta,k)}} t^{\beta} (-p)^k \\
=
\left( (1-p) \prod_{r \geq 1} \frac{ (1-p q^r) (1-p^{-1} q^r)}{(1-q^r)^2} \right)^{r d_{\Sigma}}
\sum_{\substack{(\beta,k)>0 \\ \beta \cdot f = r}} \langle \lambda_1, \lambda_2 \rangle^{S^{[n]}}_{0, \beta + kA} t^{\beta} (-p)^{k}\\
+ 
\log\left( (1-p) \prod_{r \geq 1} \frac{ (1-p q^r) (1-p^{-1} q^r)}{(1-q^r)^2} \right)
\int_{S^{[n]}} \lambda_1 \lambda_2 D(c_1(S)).
\end{multline*}
\end{enumerate}
\end{thm}
\begin{proof}
We insert the $I$-function terms computed in Proposition~\ref{prop:I function}
into Proposition~\ref{prop:wallcrossing semipositive},
which 
holds for all elliptic surfaces (not only for the semi-positive) since 
as we have seen in the proof of Proposition~\ref{prop:I function}
the only contributing curve classes in the $I$-function satisfy $c_1(S) \cdot \beta \geq 0$.
\end{proof}

In the case where we are only considering the classes $\beta = df$ we get:
\begin{cor} Let $S$ be any elliptic surface.
If $\deg_{\BC}(\lambda_1) + \deg_{\BC}(\lambda_2) = 2n-1$, then:
\begin{multline*}
\sum_{(d,k)>0} {'\langle \lambda_1, \lambda_2 \rangle^{\PT}_{0,(df,k)}} q^d (-p)^k 
=
\sum_{(d,k)>0} \langle \lambda_1, \lambda_2 \rangle^{S^{[n]}}_{0, df + kA} q^d (-p)^{k}\\
+ 
\log\left( (1-p) \prod_{r \geq 1} \frac{ (1-p q^r) (1-p^{-1} q^r)}{(1-q^r)^2} \right)
\int_{S^{[n]}} \lambda_1 \lambda_2 D(c_1(S)).
\end{multline*}
\end{cor}

Recall the $2$-point operator on the Hilbert scheme side $\QHilb$ defined in \eqref{QHilb}.
The wall-crossing above gives:
\[
( \QHilb(\lambda), \mu)
=
\sum_{d \geq 0} \sum_{\substack{k \in \BZ \\ k>0 \text{ if } d=0}} \langle \lambda, \mu \rangle^{S^{[n]}}_{0,df+kA} q^d (-p)^k.
\]
\begin{cor}  \label{cor:Q Hilb-PT comparision}
We have
\[
\QHilb = \QPT - 
\log\left( (1-p) \prod_{r \geq 1} \frac{ (1-p q^r) (1-p^{-1} q^r)}{(1-q^r)^2} \right)
e_{c_1(S)}
\]
\end{cor}

We have obtained the proof of our main formula stated in Theorem~\ref{thm:2point operator}.
\begin{proof}[Proof of Theorem~\ref{thm:2point operator}]
This follows from Corollary~\ref{cor:Q Hilb-PT comparision}, Theorem~\ref{thm:QPT} and the identity
\[
\log\left( \prod_{r \geq 1} \frac{ (1-p q^r) (1-p^{-1} q^r)}{(1-q^r)^2} \right)
=
\sum_{m,d \geq 1} \frac{1}{m} q^{md} (1-p^m) (1-p^{-m}). \qedhere
\]
\end{proof}

We also prove the quasi-Jacobi form property stated in the introduction:
\begin{proof}[Proof of Theorem~\ref{thm:structrue constants are quasi Jacobi forms}]
We consider the case $D=\delta$, the other cases are similar.
We have
\begin{align*}
(D \ast_{\pi} \lambda, \mu)
& = \mathrm{Cst} + p \frac{d}{dp} ( Q^{\Hilb} \lambda, \mu) \\
& = \mathrm{Cst'} + p \frac{d}{dp} ( Q^{\PT} \lambda, \mu)
- (D(c_1(S)) \lambda, \mu) p\frac{d}{dp} \log \Theta(p,q) \\
& = \mathrm{Cst'} + d_{\Sigma} ( \widehat{T} \lambda, \mu)
- (D(c_1(S)) \lambda, \mu) p\frac{d}{dp} \log \Theta(p,q)
\end{align*}
for some constants $\mathrm{Cst}, \mathrm{Cst}' \in \BQ$.
The first equality above is the divisor equation,
the second uses Corollary~\ref{cor:Q Hilb-PT comparision}
and the definition of the theta function \eqref{theta function},
and the third follows from \eqref{That and QPT}.
The logarithmic derivative of $\Theta$ is equal to $\A_1$, and hence a quasi-Jacobi form, and by definition of $\widehat{T}$, the term $( \widehat{T} \lambda, \mu)$ is a linear combination of terms $\A_r(p^s)$ for some $s \in \BZ$.
Since $\A_r$ are quasi-Jacobi forms with poles at lattice points, these are again quasi-Jacobi forms, but with poles at $s$-torsion points.
\end{proof}

\begin{rmk}
When rewriting the wall-crossing in the standard PT normalization,
we see that the wall-crossing terms are Jacobi forms up to a normalization factor in $q$.
Indeed, let $2h-2+N>0$ and
consider the generating series of invariants of degree $r$ over the base:
\begin{gather*}
Z^{h,r}_{\PT}(\lambda_1, \ldots, \lambda_N) = \sum_{\substack{\beta \in H_2(S,\BZ) \\ \beta \cdot f = r}}
t^{\beta}
Z^{(S \times \CC_{h,N}, S_z)}_{\PT, (\beta,n)}\left( \lambda_1, \ldots, \lambda_N \right)  \\
Z^{h,r}_{\Hilb}(\lambda_1, \ldots, \lambda_N) = \sum_{\substack{\beta \in H_2(S,\BZ) \\ \beta \cdot f = r}} \sum_{k \in \BZ}
t^{\beta} (-p)^{k}
\langle \lambda_1, \ldots, \lambda_N \rangle^{S^{[n]}}_{h,\beta+kA}
\end{gather*}
Then with $q=t^{f}$ the relation \eqref{wallcross elliptic surface} can be rewritten as
\[
Z^{h,r}_{\PT}(\lambda_1, \ldots, \lambda_N)
=
\left( i \Theta(p,q)^r \prod_{r \geq 1} (1-q^r)^{2h-2+N} \right)^{d_{\Sigma}}
Z^{h,r}_{\Hilb}(\lambda_1, \ldots, \lambda_N).
\]
\end{rmk}

\subsection{A basic check}
Let $S = \p^1 \times E$ and let $\Sigma = \p^1 \times 0_{E}$ be the section.
The Hilbert scheme of $S$ admits the isotrivial fibration
\[ \rho : S^{[n]} \to \mathrm{Sym}^n(E) \xrightarrow{+} E \]
where the first map is induced by the projection of $S$ onto the second factor
and the second map is the sum map. Consider the fiber
\[
W := \rho^{-1}(0_E) \subset S^{[n]}.
\]
\begin{lemma} $[W] = D(\alpha) \cup D(\beta)$ in $H^2(S^{[n]})$. \end{lemma}
\begin{proof}
This can be proven parallel to \cite[Lemma 3.5]{OSV}.
\end{proof}

We specialize to $n=2$. By \eqref{eqn:D(a) times D(b)} we have
\[ [W] = (\Fq_1(\Sigma) \Fq_1(1) + \Fq_1(\alpha) \Fq_1(\beta)) \vacuum. \]

We have the following basic computation,
which exactly matches \cite[Theorem 39]{HilbK3}.\footnote{The series $T$ in \cite[Thm 39]{HilbK3} computes the 
genus 0 invariants of the subspace $W \subset S^{[2]}$ in classes that push forward to $df+kA$.
The variable $y$ is related to our variable $p$ by $y=-p$, and the $q$ variables are the same. 
The insertion $D(\Sigma)$ corresponds to taking derivatives twice.
The matching hence follows from
$\langle D(\Sigma), D(\Sigma) [W] \rangle^{S^{[2]}}_{0,df+kA}
=
\langle D(\Sigma)|_{W}, D(\Sigma)|_{W} \rangle^{W}_{0,df+kA}
=
d^2 \langle 1 \rangle^W_{0,df+kA}$
and the identity
\[
2 \sum_{m,d \geq 1} \frac{d^2}{m} q^{md} (1-p^m)^2 (1-p^{-m})^2
=
\left( q \frac{d}{dq} \right)^2 T|_{y=-p}
\]
where $T$ is as in \cite[Thm 39]{HilbK3}.}
\begin{prop}
\[ \sum_{(d,k)>0} \langle D(\Sigma), D(\Sigma) [W] \rangle^{S^{[2]}}_{0,df+kA}
=
2 \sum_{m,d \geq 1} \frac{d^2}{m} q^{md} (1-p^m)^2 (1-p^{-m})^2.
\]
\end{prop}

\begin{proof}
Note that $D(\Sigma)^2|_{W} = 0$ since $D(\Sigma)^2$ can be presented as the locus of 
subschemes incident to two different sections $\Sigma_1, \Sigma_2$,
and choosing these generic, this locus is disjoint from $W$.
Hence
\[ ( e_{c_1(S)} D(\Sigma), D(\Sigma) [W] )_{S^{[2]}} = \int_{S^{[2]}}
D(c_1(S)) D(\Sigma) D(\Sigma) [W] = 0. \]
Hence we find that the left hand side is equal to
$\langle \QPT D(\Sigma), D(\Sigma) [W] \rangle$
and can be easily computed from Theorem~\ref{thm:QPT}.
\end{proof}

Similar, with $\delta = c_1(\CO_{S}^{[n]})$ we have $\delta [W] = -2 \Fq_2(\Sigma) \vacuum$,
and hence
\[
\sum_{k > 0} \langle \delta, \delta [W] \rangle^{S^{[2]}}_{0, kA}(-p)^k
= ( \Fq_2(1), \QHilb \Fq_2(\Sigma) ) = 8 ( \log(1-p) - \log(1-p^2) ).
\]
This matches the $q^0$-term of \cite[Theorem 39]{HilbK3}.

\section{The surface $E \times \BC$}\label{sec:localE}
Let $E$ be an elliptic curve and let $S = E \times \BC$.
In this section we show how our computations for proper elliptic surfaces
also determines the quantum product with divisor classes on the Hilbert scheme of points in $E \times \BC$.
This will give the proof of Theorem~\ref{thm:ExC}.

\subsection{Cohomology} \label{sec:cohomology}
Let $\BG_m$ act on $\BC$ with tangent weight $t$. This induces an action on $S$ by acting trivially in the $E$ direction. The resulting equivariant cohomology is
\[ H_{\BG_m}^{\ast}(S, \BQ) = H^{\ast}(E,\BQ) \otimes \BQ[t] \]
where we consider cohomology classes on $E$ as cohomology classes on $S$ by pullback.
The equivariant canonical class and the fiber class are
\[ K_S = -t, \quad f = t. \]
The equivariant Poincar\'e pairing is
\[
\int_{E \times \BC} \gamma_1 \gamma_2 = \frac{1}{t} \int_{E} \gamma_1 \gamma_2, \quad 
\gamma_1, \gamma_2 \in H^{\ast}(E).
\]
The equivariant diagonal class is
\[ \Delta = t \Delta_E \in H^{\ast}_{\BG_m}(S \times S)
\]
where $\Delta_E \in H^2(E \times E)$ is the class of the diagonal of $E$.

\subsection{Gromov-Witten invariants}
Let $f$ be the fiber class of $S = E \times \BC$ and let $\beta = df+kA \in \widetilde{H}_2(S^{[n]})$ .
The moduli space $\Mbar_{0,2}(S^{[n]},\beta)$ is of virtual dimension $2n-1$
but non-compact, with compact $\BG_m$-fixed locus.
The Gromov-Witten invariants of $S^{[n]}$ are hence defined by the localization formula.
Consider $H^{\ast}(E)$-weighted partitions of $n$,
\begin{gather*}
\lambda = \big( (\lambda_1, \delta_1) , \ldots, (\lambda_{\ell(\lambda)}, \delta_{\ell(\lambda)} ) \big) \\
\mu = \big( (\mu_1, \epsilon_1) , \ldots, (\mu_{\ell(\mu)}, \epsilon_{\ell(\mu)} ) \big).
\end{gather*}
Since we take the weights from $H^{\ast}(E)$, we always have
\[ \deg_{\BC}(\lambda) \leq n \]
with equality if and only if $\delta_i$ is a multiple of the point class $\pt \in H^2(E)$ for all $i$.

\begin{prop} \label{prop:degree 2n vanishing}
If $\deg_{\BC}(\lambda) + \deg_{\BC}(\mu) = 2n$, then
$\langle \lambda, \mu \rangle^{S^{[n]}}_{0,\beta} = 0$ for all $\beta$.
\end{prop}
\begin{proof}
Consider the isotrivial fibration $p : S \to E^{[n]} \to E$. All maps $f:C \to S^{[n]}$ from genus $0$ curves must map to fibers of $p$.
By assumption, we have $\deg(\lambda)=n$ and $\deg(\mu)=n$,
hence we can assume $\delta_i=\epsilon_i=\pt$ for all $i$.
Hence $\lambda$ and $\mu$ can be represented by cycles supported in a fiber of $p$.
By representing $\lambda,\mu$ on distinct fibers, we see that
$\ev_1^{\ast}(\lambda) \ev_2^{\ast}(\mu)=0$.
\end{proof}

\begin{defn}
Let $(\lambda_1, \ldots, \lambda_{\ell})$ be a partition of $n$ and let
$\Gamma_1, \ldots, \Gamma_{\ell}$ be subvarieties of $S$,
such that for all $i \neq j$ with $\Gamma_i, \Gamma_j$ given by a point we have $\Gamma_i \neq \Gamma_j$.
We define the Nakajima cycle as the closed subvariety of $S^{[n]}$ denoted by
\[ \Nak( (\lambda_1, \Gamma_1) \cdots (\lambda_{\ell}, \Gamma_{\ell}) ) \subset S^{[n]}
 \]
 which is the closure of the set of elements which are a union of punctual subschemes of length $\lambda_1, \ldots, \lambda_{\ell}$ supported at distinct points in $\Gamma_1, \ldots, \Gamma_{\ell}$ respectively.
\end{defn}

As explained in \cite{Nak} the class of $\mathrm{Nak}( (\lambda_1, \Gamma_1) \cdots (\lambda_{\ell}, \Gamma_{\ell}) )$ is given by
$\prod_{i} \Fq_{\lambda_i}([\Gamma_i]) \vacuum$ up to an automorphism factor.

Consider fixed points
\[ x_1, \ldots, x_{r}, x_1', \ldots, x'_{r} \in \BC, 
\quad y_1, \ldots, y_{\ell-r}, y'_1, \ldots, y'_{\ell'-r'} \in E. \]
For $x \in \BC$ write $f_x = \{ x \} \times E$ and for $y \in E$ write $\Sigma_y = \BC \times \{ y \}$. Consider the cycles
\begin{align*}
\lambda[x,y] & = \Nak\big( (\lambda_1, f_{x_1}) \cdots (\lambda_r, f_{x_r}) 
(\lambda_{r+1}, \Sigma_{y_1}) \cdots (\lambda_{\ell}, \Sigma_{y_{\ell-r}}) \big) \\
\mu[x',y'] & = \Nak\big( (\mu_1, f_{x'_1}) \cdots (\mu_r, f_{x'_{r'}}) 
(\mu_{r'+1}, \Sigma_{y'_1}) \cdots (\mu_{\ell'}, \Sigma_{y'_{\ell'-r'}}) \big).
\end{align*}
\begin{lemma} \label{lemma:properness}
Assume that $y_i, y'_i$ are generic. Then the closed subset
\[ \ev_1^{-1}(\lambda[x,y]) \cap \ev_2^{-1}(\mu[x',y']) \subset \Mbar_{0,2}(S^{[n]}, \beta) \]
is proper.
\end{lemma}
\begin{proof}
Let $f : C \to S^{[n]}$ be a rational curve with markings $p_1, p_2 \in C$ such that
$f(p_1) \in \lambda[x,y]$ and $f(p_2) \in \mu[x',y']$.
Let $\pi_{\Sym^n \BC} : S^{[n]} \to \mathrm{Sym}^n(\BC)$ be the map induced by the projection $S \to \BC$. The map $\pi_{\Sym^n \BC}$ is proper.
Consider the finite set:
\[ I_x = \bigcup_{\substack{a_i, b_i \geq 1 \\ \sum_i a_i + \sum_i b_i = n}} \left\{ \sum_{i=1}^{r_1} a_i x_i + \sum_{i=1}^{r_2} b_i x_i' \right\} \quad \subset \Sym^n \BC. \]
We claim that $f : C \to S^{[n]}$ maps entirely to the proper subspace $\pi_{\Sym^n \BC}^{-1}(I_x)$. If that were true, then we get
\[
\ev_1^{-1}(\lambda[x,y]) \cap \ev_2^{-1}(\mu[x',y']) \subset \Mbar_{0,2}( \pi_{\Sym^n \BC}^{-1}(I_x), \beta ).
\]
which completes the proof since the right hand side is proper.

To show the claim, note that $\pi_{\Sym^n \BC}(f(C))$ is a point.
If this point does not lie in $I_x$, then
consider the pullback $\widetilde{C} = C \times_{S^{[n]}} Z \subset S \times C$
of the universal family $Z \subset S^{[n]} \times S$.
By assumption there exists a connected component $\widetilde{C}_1 \subset \widetilde{C}$ flat of some degree $n_1$ over $C$, 
such that its projection to $\BC$ does not lie in $\{ x_1, \ldots, x_r, x'_1, \ldots, x'_{r'} \}$.
Let $f_1 : C \to S^{[n_1]}$ be the induced map.\footnote{See also \cite[Sec.1.3.2]{HilbK3} for a detailed discussion of this construction.}
By assumption $f_1(p_1)$ and $f_1(p_2)$ must lie on cycles
corresponding to parts of $\lambda$ and $\mu$ which carry only insertions of type $\Sigma$,
or in other words be incident to cycles of the form
\[ \Nak\Big( \prod_{j} (a_j, \Sigma_{y_{i_j}}) \Big),
 \Nak\Big( \prod_{j} (b_j, \Sigma_{y'_{i_j}}) \Big) \]
 for some $a_j, b_j$.
Let $\pi_{\Sym^{n_1}(E)} : S^{[n_1]} \to \Sym^{n_1}(E)$ be the map induced by $S \to E$.
Then this implies that $\pi_{\Sym^{n_1}}( f_1(p_1))$ and $\pi_{\Sym^{n_1}(E)}( f_1(p_2))$ lies in
the subsets
\begin{gather*}
I_{y} = \bigcup_{\substack{a_1,\ldots, a_r \geq 0 \\ \sum_i a_i = n_1}} \left\{ \sum_{i=1}^{r} a_i y_i \right\} \quad \subset \Sym^{n_1} E. \\
I_{y'} = \bigcup_{\substack{b_1, \ldots, b_{r'} \geq 0 \\ \sum_i b_i = n_1}} \left\{ \sum_{i=1}^{r'} b_i y'_i \right\} \quad \subset \Sym^{n_1} E.
\end{gather*}
Since the points $y_i, y'_i$ were assumed to be generic, we see that $f_1(p_1)$ and $f_1(p_2)$ lie in different fibers of
the composition $S^{[n_1]} \to \Sym^{n_1}(E) \xrightarrow{+} E$.
But this is a contradiction, since $C$ is rational.
\end{proof}

\begin{prop} \label{prop:restriction to 2n-1 case}
Assume that $\deg_{\BC}(\lambda) + \deg_{\BC}(\mu) < 2n$.
Then for $\beta > 0$ we have
\[
\langle \lambda, \mu \rangle^{S^{[n]}}_{0,\beta}
=
\sum_{ \substack{ \lambda = \lambda' \cup \rho_{\lambda} \\ \mu = \mu' \cup \rho_{\mu} }}
\left(\int_{S^{[n-n']}} \rho_{\lambda} \cup \rho_{\mu} \right)
\langle \lambda', \mu' \rangle^{S^{[n']}}_{0,\beta}
\]
where $\lambda', \mu'$ run over subpartitions such that
$|\lambda'|=|\mu'| =: n'$ and $\deg_{\BC}(\lambda') + \deg_{\BC}(\mu') = 2 n'-1$.
\end{prop}
\begin{proof}
Let us assume that
\begin{gather*}
\lambda = \big( (\lambda_1, 1) , \ldots, (\lambda_{r},1), (\lambda_{r+1}, \pt_E), \ldots,
(\lambda_{\ell}, \pt_E ) \big) \\
\mu = \big( (\mu_1, 1) , \ldots, (\mu_{r'},1), (\mu_{r'+1}, \pt_E, \ldots, (\mu_{\ell'}, \pt_E ) \big)
\end{gather*}
for some $r, r' \geq 0$ with $r+r' > 0$. Note that
\[
\deg(\lambda) + \deg(\mu) = 2n-r-r'.
\]
Let $[\bo] \in H^{2}_{\BG}(\BC)$ denote the equivariant class Poincar\'e dual to the origin in $\BC$. In cohomology we just have $[\bo]=t$.
Let us write
\begin{align*}
\lambda[\bo] & = \big( (\lambda_1, [\bo]) , \ldots, (\lambda_{r},[\bo]), (\lambda_{r+1}, \pt_E), \ldots,
(\lambda_{\ell}, \pt_E ) \big) \\
\mu[\bo] & = \big( (\mu_1, [\bo]) , \ldots, (\mu_{r'},[\bo]), (\mu_{r'+1}, \pt_E), \ldots, (\mu_{\ell'}, \pt_E ) \big).
\end{align*}

Since $S$ is holomorphic-symplectic, recall from \cite{OkPandHilbC2} that
there is a natural homology class $[ \Mbar_{0,2}(S^{[n]},\beta) ]^{\red}$ of (complex) dimension $2n$, called the reduced virtual class, such that
\[
[ \Mbar_{0,2}(S^{[n]}, \beta)  ]^{\vir} = t \cdot [ \Mbar_{0,2}(S^{[n]},\beta) ]^{\red}.
\]

By linearity and by considering the degree, we have
\begin{equation} \label{33fs}
\int_{[ \Mbar_{0,2}( S^{[n]}, \beta ) ]^{\red}} t^{r+r'} \ev_1^{\ast}(\lambda) \ev_2^{\ast}(\mu)
=
\int_{[ \Mbar_{0,2}(S^{[n]},\beta) ]^{\red}} \ev_1^{\ast}( \lambda[0] ) \ev_2^{\ast}( \mu[0] ) \in \BQ.
\end{equation}
The left hand side is formally defined by the localization formula.
By Lemma~\ref{lemma:properness} the right hand side
can be interpreted as actual equivariant pushforward over a proper space.
Since the result is an equivariant constant, we hence can specialize to a non-equivariant integral.
For the non-equivariant integral, we can replace
$\lambda[0], \mu[0]$ with $\lambda[x,y], \mu[x',y']$ for generic $x = (x_1, \ldots, x_r) \in \BC^{r}$, $x' = (x'_1, \ldots, x'_{r'}) \in \BC^{r'}$, $y \in E^{\ell-r}$, $y' \in E^{\ell'-r'}$.
The points $x_1, \ldots, x_r, x'_1, \ldots, x'_r$ are distinct.

For every point $z \in \{ x_1, \ldots, x_r, x'_1, \ldots, x'_r \}$ there is a corresponding map
$f_z : C \to S^{[n_z]}$ (obtained by the construction explained in the proof of Lemma~\ref{lemma:properness}). If $f_z$ is non-constant for more than one point,
then by the same argument as in \cite{OkPandHilbC2} there exists an additional cosection
of the obstruction theory of stable maps.
As a consequence the corresponding contribution to the invariant \eqref{33fs} vanishes.
Hence there can be only one $z_0$ for which $f_{z_0}$ is non-constant.
Let $\lambda', \mu'$ be the corresponding subpartition of $\lambda, \mu$
incident to this map. Then in $\lambda', \mu'$ there is only one cohomology weight equal to $1$,
and hence $\deg_{\BC}(\lambda') + \deg_{\BC}(\mu') = 2 n_{z_0} - 1$.
The contribution from the remaining parts is precisely $\int_{S^{[n-n_{z_0}]}} \rho_{\lambda} \rho_{\mu}$, where $\rho_{\lambda}, \rho_{\mu}$ are the component of $\lambda' \mu'$.
This completes the proof in the case of even cohomology weights.

The case of odd weights is completely parallel:
Here, since the virtual class is algebraic, we must among the cohomology weights of $\lambda,\mu$ we must have the same number of classes $\aaa$ and $\bbb$.
Then in the above argument $\lambda[\bo]$ will be obtained by multiplying all cohomology weights $1$ and $\aaa$ by the class $[\bo]$.
Similarly, Lemma~\ref{lemma:properness} holds likewise if fix the $\BC$-components of $1,\aaa$ and keep the $\BC$-components of $\bbb, \pt_E$ free: The reason is that two sums of generic cycles in class $\bbb$ still map to different points under the sum map $\Sym^n(E) \to E$.
\end{proof}

Given a $H^{\ast}(E)$-weighted partition we can view it as a cohomology class both on
$(E \times \BC)^{[n]}$ and $(\p^1 \times E)^{[n]}$. The resulting $2$-pointed invariants are related by the following result:
\begin{prop}  \label{prop:abc}
Let $\lambda, \mu$ be $H^{\ast}(E)$-weighted partitions of $n$ with
$\deg(\lambda) + \deg(\mu) = 2n-1$. Then
\[ 
\langle \lambda, \mu \rangle^{(\p^1 \times E)^{[n]}}_{0,df+kA}
=
2
\langle \lambda, \mu \rangle^{(\BC \times E)^{[n]}}_{0,df+kA}
\]
\end{prop}
\begin{proof}
We let $\BG_m$ act on $\p^1$ with fixed points $0, \infty$.
We lift the virtual class and $\lambda, \mu \in H^{\ast}( (\p^1 \times E)^{[n]})$ to equivariant cohomology classes and apply the virtual localization formula.
The result is the equality
\begin{equation} \label{localization equation}
\begin{aligned}
\langle \lambda, \mu \rangle^{(\p^1 \times E)^{[n]}}_{0,df+kA}
& =
\langle \lambda, \mu \rangle^{(\BC_0 \times E)^{[n]}}_{0,df+kA}
+ \langle \lambda, \mu \rangle^{(\BC_{\infty} \times E)^{[n]}}_{0,df+kA} \\
& \ +
\sum_{\substack{d=d_0+d_{\infty} \\ k = k_0+k_{\infty} \\ (d_0, k_0)>0, (d_{\infty}, k_{\infty})>0}}
\sum_{\substack{\lambda = \lambda_0 \sqcup \lambda_{\infty} \\ \mu = \mu_0 \sqcup \mu_{\infty}}}
\langle \lambda_0, \mu_0 \rangle^{(\BC_{0} \times E)^{[n_0]}}_{0,d_0f+k_0A}
\cdot
\langle \lambda_\infty, \mu_\infty \rangle^{(\BC_{\infty} \times E)^{[n_{\infty}]}}_{0,d_\infty f+k_\infty A}.
\end{aligned}
\end{equation}
where $\BC_0$,$\BC_{\infty}$ stands for $\BC$ equipped with the $\BG_m$-action which has tangent weight $t,-t$ on the tangent space at the origin respectively.
Moreover, $n_0 = |\lambda_0|$ and $n_{\infty} = |\lambda_{\infty}|$.
In the last term on the right hand side, we have one of the following cases:
\begin{itemize}
\item[(a)] $\deg(\lambda_0) + \deg(\mu_0) = 2n_0$ and 
$\deg(\lambda_\infty) + \deg(\mu_\infty) = 2n_{\infty}-1$, or
\item[(b)] $\deg(\lambda_0) + \deg(\mu_0) = 2n_0-1$ and 
$\deg(\lambda_\infty) + \deg(\mu_\infty) = 2n_{\infty}$.
\end{itemize}
Hence by Proposition~\ref{prop:degree 2n vanishing} we see that the third term in \eqref{localization equation} does not contribute. Since the second term lies in $\BQ$ it is invariant under changing $t$ by $-t$. Hence we get
\[ \langle \lambda, \mu \rangle^{(\p^1 \times E)^{[n]}}_{0,df+kA}
=
2 \langle \lambda, \mu \rangle^{(\BC_0 \times E)^{[n]}}_{0,df+kA} \]
as desired.
\end{proof}

\subsection{Proof of Theorem~\ref{thm:ExC}}
Proposition~\ref{prop:restriction to 2n-1 case} says that $\langle \lambda, \mu \rangle^{(\BC\times E)^{[n]}}_{0, df+kA}$ can be reconstructed from its values when $\deg(\lambda) + \deg(\mu) = 2n-1$. It is straightforward to check that the full statement of Theorem~\ref{thm:ExC} is compatible with this reconstruction, so it suffices to check the theorem when $\deg(\lambda) + \deg(\mu) = 2n-1$.

In this case Proposition~\ref{prop:abc} applies and the values of $\langle \lambda, \mu \rangle^{(\BP^1\times E)^{[n]}}_{0, df+kA}$ are given by Theorem~\ref{thm:2point operator}. After canceling factors of $2$ (from Proposition~\ref{prop:abc} and $K_{\BP^1\times E}$) and using the formulas of Section~\ref{sec:cohomology} to pull out powers of $t$, we get the claimed formula for $\langle \lambda, \mu \rangle^{(\BC\times E)^{[n]}}_{0, df+kA}$.\qed

\end{document}